%% file: vardp.tex
\documentclass[11pt]{article}
\usepackage{amssymb, authblk}
\usepackage{amsmath, bbm}
\usepackage{subfig}
\usepackage{fullpage}
\usepackage{amsthm, natbib, hyperref, multirow}
\usepackage{graphicx}
\usepackage{xspace}
\usepackage{verbatim}
\usepackage[color=yellow!80]{todonotes}

\usepackage{algorithm}
\usepackage{algpseudocode}
\usepackage{tikz}
\usetikzlibrary{arrows.meta}
\usetikzlibrary{decorations.pathreplacing}

\usepackage{graphicx}
\usepackage{algorithm}% http://ctan.org/pkg/algorithms
\usepackage{algpseudocode}% http://ctan.org/pkg/algorithmicx
\algnewcommand\algorithmicinput{\textbf{INPUT:}}
\algnewcommand\INPUT{\item[\algorithmicinput]}
\algnewcommand\algorithmicoutput{\textbf{OUTPUT:}}
\algnewcommand\OUTPUT{\item[\algorithmicoutput]}
 
\usepackage{hyperref}[]
\hypersetup{
    colorlinks=true,
    linkcolor=blue,
    filecolor=magenta,      
    urlcolor=blue,
      citecolor=blue,
    }

\usepackage{cleveref}

 \newcommand{\ou}{\"{o}}
 
\newtheorem{theorem}{Theorem}
\newtheorem{lemma}[theorem]{Lemma}
\newtheorem{proposition}[theorem]{Proposition}
\newtheorem{corollary}[theorem]{Corollary}

\newtheorem{definition}{Definition}
\newtheorem{remark}{Remark}
\newtheorem{assumption}{Assumption}
\newtheorem{model}{Model}
\newtheorem{example}{Example}

\allowdisplaybreaks

\newcommand{\op}{\text{op}}

\DeclareMathOperator*{\argmin}{arg\,min}

\DeclareMathOperator*{\esssup}{ess\,sup}
\DeclareMathOperator*{\essinf}{ess\,inf}
 
\newcommand{\hatp}{ \widehat {\mathcal P}}

 \newcommand{\LL}{\mathcal L}

\newcommand{\lasso}{{\sc Lasso}\xspace}

\date{\vspace{-5ex}}

\title{Localizing Changes in High-Dimensional Vector Autoregressive Processes}
\author[1]{Daren Wang}
\author[2]{Yi Yu}
\author[3]{Alessandro Rinaldo}
\author[1]{Rebecca Willett}

\affil[1]{\small Department of Statistics, University of Chicago}
\affil[2]{\small Department of Statistics, University of Warwick}
\affil[3]{\small Department of Statistics \& Data Science, Carnegie Mellon University}

\begin{document}

\maketitle

\begin{abstract}
Autoregressive models capture stochastic processes in which past realizations determine the generative distribution of new data; they arise naturally in a variety of industrial, biomedical, and financial settings. A key challenge when working with such data is to determine when the underlying generative model has changed, as this can offer insights into distinct operating regimes of the underlying system. This paper describes a novel dynamic programming approach to localizing changes in high-dimensional autoregressive processes and associated error rates that improve upon the prior state of the art. When the model parameters are piecewise constant over time and the corresponding process is piecewise stable, the proposed dynamic programming algorithm consistently localizes change points even as the dimensionality, the sparsity of the coefficient matrices, the temporal spacing between two consecutive change points, and the magnitude of the difference of two consecutive coefficient matrices are allowed to vary with the sample size. Furthermore, the accuracy of initial, coarse change point localization estimates can be boosted via a computationally-efficient refinement algorithm that provably improves the localization error rate.  Finally, a comprehensive simulation experiments and a real data analysis are provided to show the numerical superiority of our proposed methods.  %At the heart of the theoretical analysis lies a general framework for high-dimensional change point localization in regression settings that unveils key ingredients of localization consistency in a broad range of settings. The autoregressive model is a special case of this framework. A byproduct of this analysis are new, sharper rates for high-dimensional change point localization in linear regression settings that may be of independent interest. 

\end{abstract}

\section{Introduction}\label{sec-introduction}

High-dimensional data are routinely collected in both traditional and emerging application areas.  Time series data are by no means immune to this high dimensionality trend, and commonly arise in applications from econometrics \citep[e.g.][]{bai1998estimating, de2008forecasting}, finance \citep[e.g.][]{chen1997testing}, genetics \citep[e.g.][]{michailidis2013autoregressive}, neuroimaging \citep[e.g.][]{smith2012future, bolstad2011causal}, predictive maintenance \citep[e.g.][]{susto2014machine, swanson2001general, yam2001intelligent}, to name but a few.  

Arguably, the most popular tool in modeling high-dimensional time series is the vector autoregressive (VAR) model \citep[see e.g.][]{lutkepohl2005new}, where a $p$-dimension time series is at a given time is prescribed as white noise  perturbation a linear combinations of its past values; see Section~\ref{sec:formulation} below for a definition. In its simplest form, a VAR(1) model implies that $\mathbb{E}[X_{t+1}|X_{t}] = A X_t \in \mathbb{R}^p$ for $t \in \mathbb{Z}$ and a given $p \times p$ coefficient matrix $A$.  When $p$ is large relative to the number of observed samples in the time series, we refer to this as a {\em high-dimensional} VAR model. A large body of literature focuses estimation of parameters of these processes when the time series is stationary (as detailed in the related work section).  In this paper, we consider {\em non-stationary VAR} models in which the linear coefficients are allowed to change as a function time in  piece-wise constant manner. For a VAR(1) model, this implies that 
	\[
		\mathbb{E}[X_{t+1}|X_t] = A_t X_t
	\]
	where the entries of $A_t$ are piecewise-constant over time. The change point detection problem of interest is to estimate the times at which $A_{t+1} \neq A_t$. This model is defined precisely in \Cref{assume:AR change point} below. 

Despite the vast body of literature on different change point detection problems, the study on \Cref{assume:AR change point} is scarce \citep[e.g.][]{safikhani2017joint} and popular change point detection methods focusing on  mean or covariance changes do not perform well when the data corresponds to a non-stationary VAR model.  This claim is supported by extensive numerical experiments in \Cref{sec:numerical}; for now, we present a teaser in \Cref{eq:why transition matrix}.  
 
\begin{example} \label{eq:why transition matrix}
Let $\{X_t\}_{t=1}^n \in \mathbb{R}^p$ be generated from \Cref{assume:AR change point}, with the number of samples $n = 240$, $p = 20$, lag $L = 1$ and noise variance $\sigma^2_{\epsilon} = 1$.  The change points are $t=n/3$ and $t=2n/3$.  The coefficient matrices are
	\[
		A_t^* = \begin{cases}  
			\left(v, -v, 0_{p \times (p-2)}\right) \in \mathbb R^{p\times p}, & t = 1, \ldots, n/3-1,  \\
		    \left(-v, v, 0_{p \times (p-2)}\right) \in \mathbb R^{p\times p}, & t = n/3, \ldots, 2n/3 - 1, \\
			\left(v, -v, 0_{p \times (p-2)}\right) \in \mathbb R^{p\times p}, & t = 2n/3, \ldots, n,
		\end{cases}
	\]
	where $v \in \mathbb{R}^p$ with odd coordinates being 0.1 and even coordinates being $-0.1$ and $0_{p \times (p-2)} \in \mathbb{R}^{p \times (p-2)}$ is an all zero matrix.
\end{example}

(A systematic study of \Cref{eq:why transition matrix} is provided in \Cref{sec:numerical} scenario (ii).)  Figures~\ref{fig-1} and \ref{fig-2} show the first four coordinates of one realization of this process and the leading $4 \times 4$ sub-matrices of their sample covariance matrices, respectively.  As we can see from the black curves, the change points are not discernible  from the raw data or the sample covariances.  This phenomenon is further demonstrated in \Cref{fig-1} where two renowned competitors -- ``inspect'' \citep{inspect-R} and ``SBS-MVTS'' \citep{wbs-R} -- both fail here, but our proposed method penalized dynamic programming approach, indicated by DP and to be discussed in this paper, manages to identify the correct change points.

\begin{figure}[ht]
\begin{center}
 {\includegraphics[width=1\linewidth]{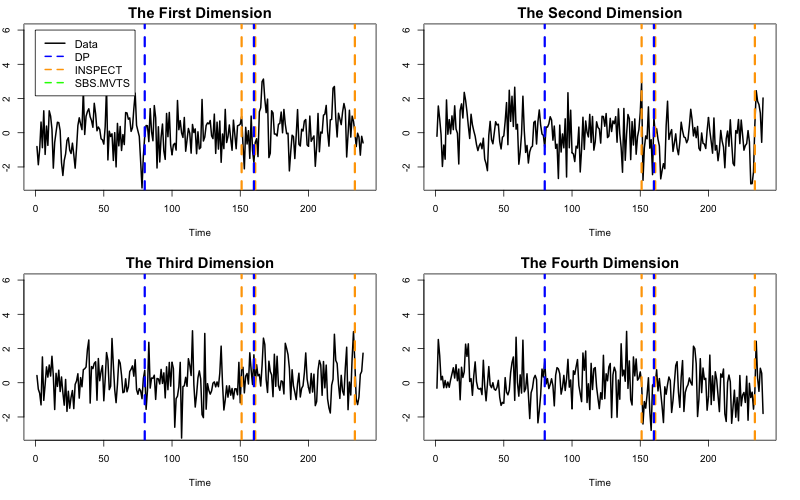}}
\caption{ \label{fig-1}  \Cref{eq:why transition matrix}.  Plots of the first four coordinates of a sample realization, $X_t[i]$ for $i = 1,2,3,4$ as well as the results of three change point detection methods: DP (dynamic programming, {\em this paper}), INSPECT \citep{inspect-R}, and SBS-MVTS \citep{wbs-R}. Unlike the other two methods, the proposed DP approach correct identifies the change points at $t=80$ and $t=160$. These plots illustrate the difficulty of detecting changes in a VAR process by looking for changes in the mean. This example is detailed in \Cref{sec:numerical}.
 }
\end{center}
\end{figure}
 
\begin{figure}[ht]
\begin{center}
 \includegraphics[width=1\linewidth]{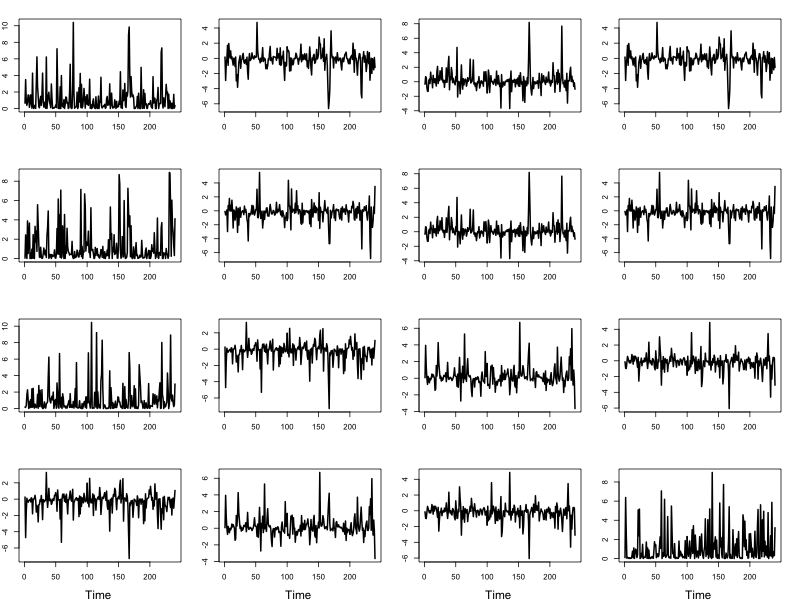}
 
\caption  {\label{fig-2}  All sixteen coordinates of the leading $4 \times 4$ sub-matrices of the sample covariance matrices of a realization from \Cref{eq:why transition matrix}, $S_t[i,j] = X_t[i]X_t[j]$ for $i,j \in \{1,2,3,4\}$. These plots illustrate the difficulty of detecting changes in a VAR process by looking for changes in the   covariance. This example is detailed in \Cref{sec:numerical}.}
 
\end{center}
\end{figure}

\subsection{Problem formulation}\label{sec:formulation}
We begin with a formal definition of a VAR process with lag $L$:
\begin{definition}[Vector autoregressive process with lag $L$.]\label{def-var-l-process}
	The sequence $\{X_t\}_{t=-\infty}^\infty \subset \mathbb{R}^p$ is generated from a Gaussian vector autoregressive process with lag $L$, $L \geq 1$, if there exists a collection of coefficient matrices $A^*[1], \ldots, A^*[L] \in \mathbb R^{p \times p}$ such that 
	\begin{equation}\label{eq-var-d}
		X_{t+1} = \sum_{l=1}^L A^*[l] X_{t+1-l } + \varepsilon_t, \quad t \in \mathbb{Z},
	\end{equation}
	where $\{\varepsilon_t\}_{t=-\infty} ^ \infty  \stackrel{\mbox{i.i.d}}{\sim } \mathcal{N}_p(0, \sigma_\epsilon^2 I_p)$ and $\sigma_{\epsilon} > 0$.

	In addition, the time series $\{ X_t\}_{t=-\infty}^\infty \subset \mathbb{R}^p$ is stable, if 
	\begin{equation}\label{eq-coef-cond-stability}
		\det \left(I_p - \sum_{l=1}^L z^lA^*[l] \right) \neq 0, \quad  \forall z \in \mathbb C, \, |z|\le 1.
	\end{equation}
\end{definition}

In this paper, we study a specific type of non-stationary high-dimensional VAR model, which possesses piecewise-constant coefficient matrices, formally introduced in \Cref{assume:AR change point}, which is built upon \Cref{def-var-l-process}.
\begin{model}[Autoregressive model]\label{assume:AR change point} 
	Let $(X_1, \ldots, X_n)$ be a time series with random vectors in $\mathbb{R}^p$ and let $1 = \eta_0 < \eta_1 < \ldots < \eta_K \leq n < \eta_{K+1} = n + 1$ be an increasing sequence of change points.  Let $\Delta$ be the minimal spacing between consecutive change points, defined as
	\begin{equation}\label{eq:Delta}
		\Delta = \min_{k = 1, \ldots, K+1} (\eta_k - \eta_{k-1}).
	\end{equation}  
	For any $k \in \{0, \ldots, K\}$, set $\mathcal{X}_k = \{X_{\eta_k}, \ldots, X_{\eta_{k+1}-1}\}$.  We assume the following.
	\begin{itemize}
	\item For each $k, l \in \{0, \ldots, K\}$, $k \neq l$, it holds that $\mathcal{X}_k$ and $\mathcal{X}_l$ are independent.
	\item For each $k \in \{0, \ldots, K\}$, $\mathcal{X}_k$ is a subset of an infinite stable time series $\mathcal{X}^{\infty}_k$ with coefficient matrices $\{A^*_{\eta_k}[l]\}_{l=1}^L$, with $1 \leq L < \Delta$ (see \Cref{def-var-l-process}).
	\item For each $k \in \{0, \ldots, K\}$, $i \in \{0, \ldots, L-1\}$, $X_{\eta_k + i}$ satisfies the model
				\[
					X_{\eta_k + i} = \sum_{l = 1}^L A_{\eta_k}^*[l]\widetilde{X}_{\eta_k + i - l}^{\top} + \varepsilon_{\eta_k + i}, 
				\]			
				where $\widetilde{X}_{\eta_k + i - l} = X_{\eta_k + i - l}$ if $i - l \geq 0$, and $\widetilde{X}_{\eta_k + i - l}$'s are unobserved latent random vectors drawn from $\mathcal{X}^{\infty}_k$ if $i - l < 0$.
	\end{itemize}
\end{model}

\begin{remark}
When defining a VAR($L$) process, one implicitly assumes the lag-$L$ coefficient matrix $A^*[L]$ is non-zero.  In \Cref{assume:AR change point}, we assume that each piece $\mathcal{X}_k$ is taken from a stable VAR($L$) process.  In fact, what we only need is that $\mathcal{X}_k$ is taken from a stable VAR($L_k$) process and $L = \max_{k = 0}^K L_k$.
\end{remark}

Given data sampled from \Cref{assume:AR change point}, our main task is to develop computationally-efficient algorithms that can consistently estimate both the unknown number $K$ of change points and the change points $\{\eta_k\}_{k = 1}^K$, at which the coefficient matrices change.  That is, we seek \emph{consistent} estimators $\{\hat{\eta}_k\}_{k = 1}^{\widehat{K}}$, such that, as the sample size $n$ grows unbounded, it holds with probability tending to 1 that 
	\[
		\widehat{K} = K \quad \mbox{and} \quad \frac{\epsilon}{\Delta} = \max_{k = 1, \ldots, K} \frac{|\widehat \eta_k -\eta_k|}{\Delta} = o(1).
	\]
	In the rest of the paper, we refer to $\epsilon$ as the {\it localization error} and $\epsilon/\Delta$ as the {\it localization rate}.

\subsection{Summary of main contributions}

We now summarize the main contributions of this paper.  First, we provide accurate   change point estimators for \Cref{assume:AR change point} in high-dimensional settings where we allow model parameters, including the dimensionality of the data $p$, the sparsity parameter $d_0$, defined to be the number of non-zero entries of the coefficient matrices, the number of change points $K$, the smallest distance between two consecutive change points $\Delta$, and the smallest difference between two consecutive different coefficient matrices $\kappa$ (formally defined in \ref{eq:kappa} below),  to change with the sample size $n$. To the best of our knowledge, the theoretical results we provide in this paper are significantly sharper than any existing literature.  

We note that  these sharp bounds require several technical advances. Specifically, most existing change point detection methods assumes subsequent observations are independent, making key technical steps such as establishing restricted eigenvalue conditions  easily verified. In contrast, our observations exhibit temporal dependence; no existing literature implies our results. Furthermore, we  cannot directly leverage high-dimensional  stationary VAR estimation bounds to obtain rates on change point detection. Analysis of our algorithm requires delicate treatment of small intervals between potential change points that allows us to prevent false discoveries and obtain sharp rates.

Second, we generalize the minimal partitioning problem \citep[e.g.][]{FriedrichEtal2008} to suit the change point localizing problem in piecewise-stable high-dimensional vector autoregressive processes.  The optimization problem can be solved using dynamic programming approaches in polynomial time.  As we will further emphasize, this optimization tool is by no means new and has been mainly used to solve change point detection in univariate data sequences.  In this paper, we have shown that for a much more sophisticated problem, this simple tool  combined with the appropriate inputs computed by our method  will still be able to produce consistent estimators.

Last, in the event that an initial estimator is at hand and satisfies mild conditions, we further devise an additional second step (\Cref{algorithm:LR}),  to deliver a provably better localization error rate, even though our initial estimator's output from the dynamic programming approach already provides the sharpest rates among the ones in the existing literature. %\todo[inline]{I think this is overstating things, because this improvement depends on certain conditions on the initial estimate being met, and  a method that produces an initial estimate that meets these conditions has slower rates. So I do not think we can claim we get a better rate. Rather, I think we can claim (a) in the event that we have an initial estimate that is sufficiently accurate, we can refine that estimate with an even better localization rate and (b) this approach yields empirical advantages in our simulations. }

\subsection{Related work}
 
For the model described in \Cref{def-var-l-process}, if the dimensionality $p$ diverges as the sample size goes unbounded, then it is a high-dimensional VAR model, the   recent literature on  estimation of  stationary  high-dimensional VAR models is vast.  \cite{hsu2008subset}, \cite{haufe2010sparse}, \cite{shojaie2010discovering}, \cite{basu2015regularized}, \cite{michailidis2013autoregressive}, \cite{loh2011high}, \cite{wu2016performance}, \cite{bolstad2011causal}, \cite{basu2015regularized}, among many others, studied  different aspects of  the \lasso penalised VAR models; \cite{han2015direct} utilized the Dantzig selector; \cite{bickel2011banded}, \cite{guo2016high} and others resorted to banded autocovariance structures for time series modeling; the low rank conditions were exploited in \cite{forni2005generalized}, \cite{lam2012factor}, \cite{chang2015high}, \cite{chang2018principal}, among many others; \cite{xiao2012covariance}, \cite{chen2013covariance} and \cite{tank2015bayesian} focused on the properties of the covariance and precision matrices; various other inference related problems were also studied in \cite{chang2017testing}, \cite{fiecas2018spectral}, \cite{schneider2016partial}, among many others.

The above list of references, far from being complete, is concerned with stationary time series.  As for non-stationary high-dimensional time series data, \cite{zhang2019identifying} and \cite{tu2018errorcorrection}, among others, studied error-correction models; \cite{wang2017optimal} and \cite{aue2009break} examined the covariance change point detection problem; \cite{cho2015multiple},\cite{cho2016change}, \cite{wang2018high}, \cite{dette2018relevant}, among many others, studied change point detection for high-dimensional time series with piecewise-constant mean; recently, \cite{leonardi2016computationally} examined the change point detection problem in regression setting and \cite{safikhani2017joint} proposed a fused-\lasso-based approach to estimate both the change points and the parameters of high-dimensional piecewise VAR models.  

%We note that as pointed out in many existing time series  literature (see, e.g.,    , \cite{safikhani2017joint}  and references therein), the efforts made   understand   whether and when any  method developed in the regression setting would work in VAR setting  was a significant contribution.

\subsection{Notation}\label{sec-notation}

Throughout this paper, we adopt the following notation.  For any set $S$, $|S|$ denotes its cardinality.  For any vector $v$, let $\|v\|_2$, $\|v\|_1$, $\|v\|_0$ and $\|v\|_{\infty}$ be its $\ell_2$-, $\ell_1$-, $\ell_0$- and entry-wise maximum norms, respectively; and let $v(j)$ be the $j$th coordinate of $v$.  For any square matrix $A \in \mathbb{R}^{n \times n}$, let $\Lambda_{\min}(A)$ and $ \Lambda_{\max}(A)$ be the smallest and largest eigenvalues of matrix $A$, respectively; for a non-empty $S \subset \{1, \ldots, n\}$, let $A_S$ be the sub matrix of $A$ consisting of the entries in coordinates $S \times S$.  For any matrix $B \in \mathbb{R}^{n \times m}$, let $\|B\|_{\op}$ be the operator norm of $B$; let $\|B\|_1 = \|\mathrm{vec}(B)\|_1$, $\|B\|_2 = \|\mathrm{vec}(B)\|_2$ and $\|B\|_0 = \|\mathrm{vec}(B)\|_0$, where $\mathrm{vec}(B) \in \mathbb{R}^{nm}$ is the vectorization of $B$ by stacking the columns of $B$.  In fact, $\|B\|_2$ corresponds to the Frobenius norm of $B$.  For any pair of integers $s, e \in \{0, 1, \ldots, n\}$ with $s < e$, we let $(s, e] = \{s + 1, \ldots, e\}$ and $[s, e] = \{s , \ldots, e\}$ be the corresponding integer intervals.

\section{Methods}\label{sec-methods}

In this section, we detail our approaches.  In \Cref{sec-minimal-partition}, we use a dynamic programming approach to solve the minimal partition problem \eqref{eq-wide-p}, which involves a loss function based on \lasso estimators of the coefficient matrices \eqref{eq:VARD oracle} and which provides a sequence of change point estimators.  In \Cref{section:pgl}, we propose an optional post-processing algorithm, which requires an initial estimate as input.  Such inputs can (but is not restricted to) be the estimator analyzed in \Cref{sec-minimal-partition}.  The core of the post-processing algorithm is to exploit a group \lasso estimation to provide refined change point estimators.

\subsection{The minimal partitioning problem and dynamic programming approach}\label{sec-minimal-partition}

To achieve the goal of obtaining consistent change point estimators, we adopt a dynamic programming approach.  In detail, let $\mathcal{P}$ be an {\it interval partition} of $\{1, \ldots, n\}$ into $K_{\cal P}$ time intervals, i.e.
	\[
		\mathcal{P} = \bigl\{\{1, \ldots, i_1-1\}, \{i_1, \ldots, i_2-1\}, \ldots, \{i_{K_{\mathcal{P}}}, \ldots, n\}\bigr\},
	\]
	for some integers $1 < i_1 < \cdots < i_{K_{\mathcal{P}}} \leq n$, where $K_{\mathcal{P}} \geq 1$.  For a positive tuning parameter $\gamma > 0$, let 
	\begin{equation}\label{eq-wide-p}
		\widehat{\mathcal{P}} \in \argmin_{\mathcal{P}} \left\{\sum_{I \in \mathcal{P}} \mathcal{L}(I) + \gamma |\mathcal{P}|\right\}, 
	\end{equation}
	where $\mathcal{L}(\cdot)$ is a loss function to be specified below, $|\mathcal{P}|$ is the cardinality of $\mathcal{P}$ and the minimization is taken over all possible interval partitions of $\{1, \ldots, n+1\}$.
	
The change point estimator induced by the solution to \eqref{eq-wide-p} is simply obtained by taking all the left endpoints of the intervals $I \in \widehat{\mathcal{P}}$, except $1$.  The optimization problem \eqref{eq-wide-p} is known as the \emph{minimal partition problem} and can be solved using dynamic programming with computational cost of order $O\{n^2 \mathcal T(n)\}$, where $\mathcal T(n)$ denotes the computational cost of solving $\mathcal{L}(I)$ with $|I| = n$ \citep[see e.g.~Algorithm 1 in][]{FriedrichEtal2008}.  Algorithms based on \eqref{eq-wide-p} are widely used in the change point detection literature.  \cite{FriedrichEtal2008}, \cite{KillickEtal2012}, \cite{Rigaill2010}, \cite{MaidstoneEtal2017}, \cite{wang2018univariate}, among others, studied dynamic programming approaches for change point analysis involving a univariate time series with piecewise-constant means. \cite{leonardi2016computationally} examined high-dimensional linear regression change point detection problems by using a version of dynamic programming approach. 

We will tackle \Cref{assume:AR change point} in the framework of \eqref{eq-wide-p}, by setting $I = [s, e]$, where $s > e$  are two positive integers and
	\begin{align}\label{eq:VARD likelihood}
		\mathcal L (I) = \begin{cases}
			\sum_{t = s + L}^e \left \|X_{t} - \sum_{l=1}^L \widehat A _I[l] X_{t-l} \right \|^2, & e-s-L+1 \ge \gamma, \\
			0, & \text{otherwise,}
		\end{cases}
	\end{align} 
	with 
	\begin{align}\label{eq:VARD oracle}
		\left\{\widehat {A}_I[l] \right\}_{l = 1}^L = \argmin_{\{A[l]\}_{l=1}^L \in \mathbb R^ {p\times p}} \sum_{t = s + L}^e \left \|X_{t} -\sum_{l=1}^L A[l] X_{t-l}\right\|^2 + \lambda \sqrt{e-s-L+1}\sum_{l=1}^L\|A[l]\|_1.
\end{align}

In the sequel, for simplicity we refer to the change point estimation pipeline based on \eqref{eq-wide-p}, \eqref{eq:VARD likelihood} and \eqref{eq:VARD oracle} as the dynamic programming (DP) approach.  In the DP approach, the estimation is twofold.  Firstly, we estimate the high-dimensional coefficient matrices adopting the \lasso procedure in \eqref{eq:VARD oracle}, where the tuning parameter $\lambda$ is used to obtain sparse estimates.  Secondly, with the estimated coefficient matrices, we summon the minimal partitioning setup in \eqref{eq-wide-p} to obtain change point estimators, where the tuning parameter $\gamma$ is deployed to penalize over-partitioning.  More discussions on the choice of tuning parameters will be provided later in the sense of both theoretical guarantees and practical guidance.

Note that in \eqref{eq:VARD likelihood} and \eqref{eq:VARD oracle}, the summation over $t$ is taken only from $s+L$ to $e$ in order to guarantee that the loss function $\mathcal{L}$ on each interval $I$ is solely depending on the information in that interval.  For notational simplicity and with some abuse of the notation, in the proofs provided in the Appendices, we will just use $t \in I$ instead of $t \in [s +L, e]$.

For completeness, we present the DP procedure in \Cref{algorithm:DP}.
\begin{algorithm}[htbp]
\begin{algorithmic}
	\INPUT Data $\{X(t)\}_{t=1}^{n}$, tuning parameters $\lambda, \zeta > 0$.
	\State $(\mathcal{B}, s, t, \mathrm{FLAG}) \leftarrow (\emptyset, 0, 2, 0)$
	\While{$s < n-1$}
		\State $s \leftarrow s + 1$
		\While{$t < n$ and $\mathrm{FLAG} = 0$}
			\State $t \leftarrow t+1$ \Comment{$\mathcal{L}(\cdot)$ is defined in \eqref{eq:VARD likelihood}-\eqref{eq:VARD oracle} and $\lambda$ is involved thereof.}
			\If{$\min_{l = s+1, \ldots, t-1}\{\mathcal{L}([s, l]) + \mathcal{L}([l+1, t]) + \gamma < \mathcal{L}([s, t])\}$} 
				\State $s \leftarrow \min \argmin_{l = s+1, \ldots, t-1}\{\mathcal{L}([s, l]) + \mathcal{L}([l+1, t]) + \gamma < \mathcal{L}([s, t])\}$
				\State $\mathcal{B} \leftarrow \mathcal{B} \cup \{s\}$
				\State $\mathrm{FLAG} \leftarrow 1$
			\EndIf
		\EndWhile
	\EndWhile
	\OUTPUT The set of estimated change points $\mathcal{B}$.
\caption{Penalized dynamic programming. DP$(\{X(t)\}_{t=1}^{n}, \lambda, \gamma)$}
\label{algorithm:DP}
\end{algorithmic}
\end{algorithm}

\subsection{Post processing through group \lasso} \label{section:pgl}

%Though, as we show below, the localization rates afforded by the DP approach in \Cref{algorithm:DP} are already sharper than any other rates previously established in the literature, 

As shown in \Cref{sec-theory} below, the DP approach in \Cref{algorithm:DP} delivers consistent change point estimators with localization error rate that are sharper than any other rates previously established in the literature. In fact, these rates can be further improved upon by deploying a more sophisticated procedure that refines the initial change point estimators via post-processing  through group \lasso (PGL), step detailed in \Cref{algorithm:LR}.

\begin{algorithm}[htbp]
\begin{algorithmic}
	\INPUT Data $\{X(t)\}_{t=1}^{n}$, a collection of time points $\{\widetilde{\eta}_k\}_{k = 1}^{\widetilde{K}}$ , tuning parameter $\zeta > 0$.
	\State $(\widetilde{\eta}_0, \widetilde{\eta}_{\widetilde{K} + 1}) \leftarrow (1, n+1)$
	\For{$k = 1, \ldots, \widetilde{K}$}  
		\State $(s_k, e_k) \leftarrow (2\widetilde{\eta}_{k-1}/3 + \widetilde{\eta}_{k}/3, 2\widetilde{\eta}_{k}/3 + \widetilde{\eta}_{k+1}/3)$
		\State 
		\begin{align} \nonumber 
			& \left(\widehat{A} , \widehat{B} , \widehat{\eta}_k\right) \leftarrow  \argmin_{\substack{\eta \in \{s_k + L, \ldots, e_k - L\} \\ \{ A[l]\}_{l=1}^L , \{ B[l] \}_{l=1}^L  \subset  \mathbb{R}^{p \times p} } }  \Bigg\{\sum_{t = s_k + L}^{\eta}\bigl\|X_{t + 1} -  \sum_{l=1}^LA[l] X_{t+1-l} \bigr\|^2_2 \nonumber  \\
			  & + \sum_{t = \eta + L}^{e_k}\bigl\|X_{t + 1} - \sum_{l=1}^L B[l] X_{t+1-l} \bigr\|_2^2  + \zeta  \sum_{l=1}^L  \sum_{i, j = 1}^p  \sqrt{(\eta - s_k)(A[l])_{ij}^2 
			+ (e_k - \eta)(B[l])_{ij}^2}\Bigg\} \label{eq-g-lasso}
		\end{align}
	\EndFor
	\OUTPUT The set of estimated change points $\{\widehat{\eta}_k\}_{k = 1}^{\widetilde {K}}$.
\caption{Post-processing through group \lasso. PGL$(\{X(t)\}_{t=1}^{n}, \{\widetilde{\eta}_k\}_{k = 1}^{\widetilde{K}}, \zeta)$}
\label{algorithm:LR}
\end{algorithmic}
\end{algorithm} 

The idea of the PGL algorithm is to refine a preliminary collection of change point estimators $\{\widetilde{\eta}_k\}_{k = 1}^{K}$, which is taken in as inputs.  The preliminary change point estimators produce a sequence of consecutive triplets $\{\widetilde{\eta}_k, \widetilde{\eta}_{k+1}, \widetilde{\eta}_{k+2}\}_{k = 0}^{\widetilde{K}-1}$, with $\widetilde{\eta}_k = 1$ and $\widetilde{\eta}_{K+1} = n+1$.  The rest of the post processing works in each interval $(\widetilde{\eta}_k, \widetilde{\eta}_{k+2})$ to refine the estimator $\widetilde{\eta}_{k+1}$, and thus can be scaled parallel.  

In each interval, we first shrink it to a narrower one $(2\widetilde{\eta}_{k-1}/3 + \widetilde{\eta}_{k}/3, 2\widetilde{\eta}_{k}/3 + \widetilde{\eta}_{k+1}/3)$, in order to avoid false positive.  The constants $1/3$ and $2/3$ are to some extent \emph{ad hoc}, but works for a wide range of initial estimators, which are not necessarily consistent.  We then conduct group \lasso procedure stated in \eqref{eq-g-lasso} to refine $\widetilde{\eta}_k$.  The improvement will be quantified in \Cref{sec-theory}.  Intuitively, the group \lasso penalty exploits the piecewise-constant property of the coefficient matrices and improves the estimation.  It is worth mentioning that one may replace the \lasso penalty in \eqref{eq:VARD oracle} with the group \lasso penalty in the initial estimation stage and to improve the estimation from there, but due to the computational cost of the group \lasso estimation, we stick to the two-step strategy we introduce in this paper.

\section{Consistency of the change point estimators}\label{sec-theory}
In this section, we provide the theoretical guarantees for the change point estimators arising from the DP approach and the PGL algorithm, based on data generated from \Cref{assume:AR change point}.

 \subsection{Assumptions}\label{sec-assumptions}
 We begin by formulating the assumptions we impose to derive consistency guarantees.  

\begin{assumption} \label{assume:AR high dim coefficient}  
 
Consider \Cref{assume:AR change point}.  We assume the following.
	\begin{itemize}
		\item [\textbf{a.}] (Sparsity).  The coefficient matrices $\{A_t^*[l],\, l = 1, \ldots, L\}$ satisfy the stability condition \eqref{eq-coef-cond-stability} and there exists a subset $S \subset \{1, \ldots, p\} ^{\times 2}$, $|S| = d_0$, such that
			\[
				(A_t^*[l])(i, j) = 0, \quad t = 1, \ldots, n, \, l = 1, \ldots, L, \, (i, j) \in S^c.
			\]
			In addition, suppose that 

		\item [\textbf{b.}] (Spectral density conditions).  For $k \in \{0, \ldots, K\}$, $h \in \mathbb{Z}$, let $\Sigma_k(h)$ be the population version of the lag-$h$ autocovariance function of $\mathcal{X}_k$ and let $f_k(\cdot)$ be the corresponding spectral density function,  defined as
			\[
				\theta \in (-\pi , \pi] \mapsto f_k(\theta) = \frac{1}{2\pi} \sum_{\ell = -\infty}^{\infty} \Sigma_k(\ell) e^{-\imath \ell \theta}.
			\]
			Assume that
				\[
					\mathcal{M} = \max_{k = 0, \ldots, K} \mathcal{M}(f_k) = \max_{k = 1, \ldots, K}\esssup_{\theta\in (-\pi, \pi]}\Lambda_{\max} (f_k(\theta)) < \infty
				\]
				and
				\[
					\mathfrak{m} = \min_{k = 0, \ldots, K} \mathfrak{m}(f_k) = \min_{k = 1, \ldots, K} \essinf_{\theta \in (-\pi, \pi]} \Lambda_{\min}(f_k(\theta)) > 0.
				\]	

		\item [\textbf{c.}] (Signal-to-noise ratio).  For any $\xi > 0$, there exists an absolute constant $C_{\mathrm{SNR}} > 0$, dependent on $\mathcal{M}$ and $\mathfrak{m}$, such that 
			\begin{equation}\label{eq:snr-vard}
				\Delta \kappa^2 \geq C_{\mathrm{SNR}} d_0^2 K \log^{1+\xi}(n \vee p),	
			\end{equation}
			where $\Delta$ is the  minimal spacing  (see \ref{eq:Delta} above) and $\kappa$ is the minimal jump size, defined as 
			\begin{equation}\label{eq:kappa}
				\kappa = \min_{k = 1, \ldots, K + 1} \sqrt{\sum_{l = 1}^L \|A^*_{\eta_k}[l] - A^*_{\eta_{k-1}}[l]\|^2_2} \quad \mbox{and} \quad \Delta = \min_{k = 1, \ldots, K + 1} (\eta_{k} - \eta_{k-1}).
			\end{equation}
			In addition, suppose that $0<\kappa < C_\kappa < \infty$ for some absolute constant $C_{\kappa} > 0$.
	\end{itemize}	
\end{assumption} 

\Cref{assume:AR high dim coefficient}(\textbf{a}) and (\textbf{b}) are imposed to guarantee that the \lasso estimators in \eqref{eq:VARD oracle} exhibit good performance, while \Cref{assume:AR high dim coefficient}(\textbf{c}) can be interpreted as a signal-to-noise ratio condition for detecting and estimating the location of the change points.  We further elaborate on these conditions next.
	\begin{itemize}
	\item \textbf{Sparsity.}	  The set $S$ appearing in \Cref{assume:AR high dim coefficient}(\textbf{a}) is a superset of the union of all the nonzero entries of all the coefficient matrices.  If, alternatively, the sparsity parameter is defined as $d_0 = \max_{t = 1, \ldots, n}\left|S^t\right|$, where $S^t \subset \{1, \ldots, p\}^{\times 2}$ and $(A_t^*[l])(i, j) = 0$, for all $(i, j) \in  \{1, \ldots, p\}^{\times 2} \setminus S^t$, $l = 1, \ldots, L$, then the signal-to-noise ratio in \eqref{eq:snr-vard} and the localization error rate in \Cref{thm-var-d} would change correspondingly, by replacing the sparsity level $d_0$  with $Kd_0$.  
	\item \textbf{Spectral Density.}  The spectral density  condition  in \Cref{assume:AR high dim coefficient}(\textbf{b}) is identical to Assumption~2.1 and the assumption in Proposition~3.1 in \cite{basu2015regularized}, which pertained to a stable VAR process without change points.  As pointed out in \cite{basu2015regularized}, this holds for a large class of general linear processes, including stable ARMA processes.  
	\item \textbf{Signal-to-noise ratio.}  We assume $\kappa$ is upper bounded for stability, but we allow $\kappa \to 0$ as $n \to \infty$, which is more challenging in terms of detecting the change points.  Since $\kappa <C_\kappa $, \eqref{eq:snr-vard} also implies that
		\[
			\Delta \geq C'_{\mathrm{SNR}} d_0^2 K \log^{1+\xi}(n\vee p), 
		\]
		where $C'_{\mathrm{SNR}} =C_{\mathrm{SNR}}C_{\kappa}^{-2}$.  For convenience, we will assume without loss of generality that $C_{\kappa} = 1$.  	
		
		To facilitate the understanding of the signal-to-noise ratio condition, consider the case that $K = d_0 = 1$, \eqref{eq:snr-vard} becomes
		\[
			\Delta \kappa^2 \gtrsim \log^{1 + \xi}(n \vee p),
		\]	
		which matches be the minimax optimal signal-to-noise ratio (up to constants and logarithmic terms) for the univariate mean change point detection problem \citep[see e.g.][]{chan2013, FrickEtal2014, wang2018univariate}.
	
	 \Cref{assume:AR high dim coefficient}(\textbf{c}) can be expressed using a normalized jump size
		\[
			\kappa_0 = \kappa/\sqrt{d_0},
		\]
		which leads to the equivalent signal-to-noise ratio condition
		\[%begin{equation}\label{eq-snr-var1-2}
			\Delta \kappa_0^2 \geq C_{\mathrm{SNR}} d_0 K \log^{1+\xi}(n \vee p).
		\]%end{equation}
	Similar conditions are required in other change point detection problems, including high-dimensional mean change point detection \citep{wang2018high}, high-dimensional covariance change point detection \citep{wang2017optimal}, sparse dynamic network change point detection \citep{wang2018optimal}, high-dimensional regression change point detection \citep{wang2019statistically}, to name but a few.  We remark that in these papers, which deploy variants of the wild binary segmentation procedure \citep{fryzlewicz2014wild}, additional knowledge is needed in order to eliminate the dependence of the term $K$ in the signal-to-noise ratio condition.  We refer the reader to \cite{wang2018optimal} for more discussions regarding this point.
	
The constant $\xi$ is needed to guarantee consistency and can be set to zero if $\Delta = o(n)$.  We may instead replace it by a weaker condition of the form
	\[
		\Delta \kappa^2 \gtrsim C_{\mathrm{SNR}} d_0^2 K \{\log(n \vee p) + a_n\},
	\]
	where $a_n \to \infty$ arbitrarily slow as $n \to \infty$.  We stick with the signal-to-noise ratio condition \eqref{eq:snr-vard} for simplicity.
	\end{itemize}

\subsection{The consistency of the dynamic programming approach}\label{sec-main-results}

In our first result, whose proof is given in \Cref{appendix:var 1}, we demonstrate consistency of the dynamic programming approach introduced in \Cref{sec-methods}, under the assumptions listed in \Cref{sec-assumptions}. 
 
\begin{theorem} \label{thm-var-d}
Assume \Cref{assume:AR change point} and let $L\in \mathbb Z^+$ be constant. Them, under \Cref{assume:AR high dim coefficient}, the change point estimators $\{\widehat {\eta}_k\}_{k = 1}^{\widehat{K}}$ from the dynamic programming approach in \Cref{algorithm:DP} with tuning parameters 
	\begin{equation}\label{eq-tuning-para-thm1}
		\lambda = C_{\lambda} \sqrt{\log(n\vee p)} \quad \text{and} \quad \gamma = C_\gamma (K+1) d_0^2\log(n \vee p)   	
	\end{equation}
	are such that
	\[
		\mathbb{P}\left\{\widehat K = K \quad \mbox{and} \quad \max_{k = 1, \ldots, K}|\widehat {\eta}_k - \eta_k| \leq \frac{KC_{\epsilon} d^2_0 \log(n\vee p)}{\kappa^2} \right\} \geq 1 -2 (n\vee p)^{-1},
	\]
		where $C_{\lambda}, C_{\gamma}, C_{\epsilon} > 0$ are absolute constants depending only on $\mathcal{M}$, $\mathfrak{m}$ and $L$.
\end{theorem}
	
\begin{remark}
We do not assume that \Cref{algorithm:DP} admits a unique optimal solution $\widehat {\mathcal P}$.  In fact the consistency result in \Cref{thm-var-d} holds for {\it any} minimizer of DP.  %This is common in the high-dimensional statistical literature \citep[e.g.][]{negahban2012unified} that all the empirical minimizers of the penalized loss function are known to be consistent in various settings.   	
\end{remark}

\Cref{thm-var-d} implies that with probability tending to 1 as $n$ grows,
	\[
		\max_{k = 1, \ldots, K}\frac{|\widehat{\eta}_k - \eta_k|}{\Delta} \leq \frac{KC_{\epsilon} d^2_0 \log(n \vee p)}{\kappa^2 \Delta} \leq \frac{C_{\epsilon}}{C_{\mathrm{SNR}}\log^{\xi}(n \vee p)} \to 0,
	\]
	where the second inequality follows from \Cref{assume:AR high dim coefficient}(\textbf{c}).  Thus, the localization error rate converges to zero in probability, yielding consistency.  
	
The tuning parameter $\lambda$ affects the performance of the \lasso estimator, as elucidated in \Cref{prop:deviation AR change point}.  The second tuning parameter $\gamma$ prevents overfitting while searching the optimal partition as a solution to the problem \eqref{eq-wide-p}.  In particular, $\gamma$ is determined by the squared $\ell_2$-loss of the \lasso estimator and is of order $\lambda^2 d_0^2$.  
	
We now compare the results in \Cref{thm-var-d} with the guarantees established in \cite{safikhani2017joint}.
	\begin{itemize}
	\item In terms of the localization error rate, \cite{safikhani2017joint} proved  consistency of their methods by  assuming  that the minimal magnitude of the structural changes $\kappa$ is a sufficiently large constant independent of $n$, while our dynamic programming approach is valid even when $\kappa$ is allowed to decrease with the sample size $n$. In addition, \cite{safikhani2017joint} achieve the localization error bound of order 
		\[
			K \widetilde{\Delta} d_n^2,
		\]	
		where $\widetilde{\Delta}$ satisfies $K^2d_0^2 \log(p) \lesssim \widetilde{\Delta} \lesssim \Delta$.   Translating to our notation, their best localization error is at least
		\[
			K^5 d_0^4\log(p),
		\]
		which is larger than our rate $Kd_0^2 \log(n\vee p)/\kappa^2$ even in their setting where $\kappa$ is a constant.
	\item In terms of methodology, \cite{safikhani2017joint} adopted a two-stage procedure: first, a penalized least squares estimator with a total variation penalty is deployed to obtain an initial estimator of the change points; then an information-type criterion is applied to identify the significant estimators and to remove any false discoveries.  The change point estimators in \cite{safikhani2017joint} are selected from fused \lasso estimators, which are sub-optimal for change point detection purposes, especially when the size of the structure change $\kappa$ is small \citep[see e.g.][]{lin2017sharp}.  In addition, the theoretically-valid selection criterion proposed in \cite{safikhani2017joint} has a computational cost growing exponentially in $\widetilde K$, where $\widetilde K$ is the number of change points estimated by the fused \lasso and in general one has that $\widetilde K \gg K$.
	\end{itemize}
 
\subsection{The improvement of the post processing through group \lasso}

We now show that the PGL refinement procedure (\Cref{algorithm:LR}) applied to the estimators $\{\widetilde{\eta}_k\}_{k = 1}^{\widehat{K}}$ obtained wih \Cref{algorithm:DP} delivers a smaller localization rate with no direct dependence on $K$.

\begin{theorem}\label{theorem:lr}
Assume the same conditions in \Cref{thm-var-d}.  Let $\{\widetilde{\eta}_k\}_{k = 1}^{K}$ be any set of time points satisfying
	\begin{equation}\label{eq-lr-cond}
		\max_{k = 1, \ldots, K} |\widetilde{\eta}_k - \eta_k| \leq \Delta/7.
	\end{equation}
	Let $\{\widehat{\eta}_k\}_{k = 1}^{K}$ be the change point estimators generated from \Cref{algorithm:LR}, with $\{\widetilde{\eta}_k\}_{k = 1}^{K}$ and the tuning parameter 
	\[
		\zeta = C_{\zeta}\sqrt{\log(n \vee p)},
	\]
	as inputs.  Then
	\[
		\mathbb{P}\left\{\max_{k = 1, \ldots, K}|\hat{\eta}_k - \eta_k| \leq \frac{C_{\epsilon} d_0 \log(n \vee p)}{\kappa^2} \right\} \geq 1 - (n\vee p)^{-1},
	\]
	where $C_{\zeta}, C_{\epsilon}> 0$ are absolute constants depending only on $\mathcal{M}$, $\mathfrak{m}$ and $L$.
\end{theorem}

As we have already shown in \Cref{thm-var-d}, the estimators of the change points obtained of DP detailed in \Cref{algorithm:DP} satisfy, with high probability, condition \eqref{eq-lr-cond} and therefore can serve as qualified inputs to \Cref{algorithm:LR}.  However, we would like to emphasize that \Cref{theorem:lr} holds even when the inputs are a sequence of inconsistent initial estimators of the change points.  

Compared to the localization errors given in \Cref{thm-var-d}, the estimators refined by the PGL algorithm provide a substantial improvement by reducing the localization rate by a factor of $d_0 K$.  We provide some intuitive explanations for the success of the PGL algorithm.
	\begin{itemize}
		\item The intuition of the group \lasso penalty in \eqref{eq-g-lasso} is as follows.  Suppose that for simplicity $L = 1$.  The use of the penalty term
			\[
				\zeta \sum_{i, j = 1}^p  \sqrt{(\eta - s_k) A _{ij}^2 + (e_k - \eta) B _{ij}^2}
			\]
			implies that for any $(i, j)$th entry, the group lasso solution is such that either the $A_{ij}$'s' and $B_{ij}$'s are simultaneously $0$ or they are simultaneously nonzero.  This penalty effectively conforms  to \Cref{assume:AR high dim coefficient}({\bf a}), which requires that all the population coefficient matrices share the common support.
		\item Even if the coefficient matrices do not share a common support, the group \lasso penalty \eqref{eq-g-lasso} still works.  Suppose again for simplicity that  $L = 1$.  In addition, let $A^*_t$, $t \in [1,\eta_1)$ have support $S_1$ and  $A^*_t$, $t\in [\eta_1, \eta_2)$ have support $S_2 \neq S_1$.  Then the common support of $A_t^*$, $t \in [1, \eta_2)$ is $S_1 \cup S_2$.  Since $|S_1 \cup S_2| \le 2 d_0$, the localization rate in \Cref{theorem:lr} will only be inflated by a constant factor. In fact, the numerical experiments in \Cref{sec-numerical} also suggest that the PGL algorithm is robust even when the coefficient matrices do not share a common support.
		\item The condition \eqref{eq-lr-cond} is to ensure that there is one and only one true change point in every working interval used by the local refinement algorithm.  The true change points can then be estimated separately using $K$ independent searches, in such a way that the final localization rate does not depend on the number of searches. 
	\end{itemize}

\section{Sketch of the proofs}\label{sec-rm}

In this section we provide a high-level summary of the technical arguments we use to prove \Cref{thm-var-d} \Cref{theorem:lr}.  The complete proofs are given in the Appendices.  The main arguments consists of three main three steps.  Firstly, we will express a VAR$(L)$ process with arbitrary but fixed lag  $L$ into an appropriate VAR(1) process (this is a standard reduction technique).  Then we will lay down the key ingredients  needed to prove \Cref{thm-var-d} and, lastly, we will pinpoint the main task to prove \Cref{theorem:lr}.

\subsection*{Transforming VAR($L$) processes to VAR(1) processes}\label{sec-trans}

For any general $L \in \mathbb{Z}^+$, a VAR($L$) process can be rewritten as a VAR(1) process in the following way.  Assuming \Cref{assume:AR change point}, let
	\begin{equation}\label{eq-transform-l-to-1}
		Y_t = (\widetilde{X}_t^{\top}, \ldots, \widetilde{X}_{t - L + 1}^{\top})^{\top},
	\end{equation}
	where, for $l \in \{0, \ldots, L-1\}$, $\widetilde{X}_{t-l} = X_{t-l}$ if $t - \max\{\eta_k:\, \eta_k \leq t\} \geq l$, and $\widetilde{X}_t$ is the unobserved random vector from drawn from $\mathcal{X}^{\infty}_k$, $k = \max\{j: \eta_j \leq t\}$ otherwise.  In addition, let $\zeta_t = (\varepsilon_t^{\top}, \ldots, \varepsilon_{t-L+1}^{\top})^{\top}$ and 
	\begin{equation}\label{eq-mathcal-A}
		\mathcal{A}^*_t = \left(\begin{array}{ccccc}
			A^*_t[1] & A^*_t[2] & \cdots & A^*_t[L-1] & A^*_t[L] \\
			I & 0 & \cdots & 0 & 0 \\
			0 & I & \cdots & 0 & 0 \\
			\vdots & \vdots & \ddots & \vdots & \vdots \\
			0 & 0 & \cdots & I & 0 
		\end{array}
		\right).
	\end{equation}
	Then we can rewrite \eqref{eq-var-d} as
	\begin{equation}\label{eq-Yt}
		Y_{t + 1} = \mathcal{A}^*_t Y_t + \zeta_{t+1},
	\end{equation}
	which is now a VAR(1) process.   We remark that the stability condition \eqref{eq-coef-cond-stability} is now equivalent to all the eigenvalues of $\mathcal{A}^*_t$ as defined in \eqref{eq-mathcal-A} having a modulus strictly less than $1$ \citep[e.g.][Rule (7) in Section A.6 in]{lutkepohl2005new}.  This further implies that \eqref{eq-coef-cond-stability} implies that $\|\mathcal{A}^*_t\|_{\mathrm{op}} \leq 1$.
	
The rest of the proof will be established based on this transformation, therefore, we assume $L = 1$.

\subsection*{Sketch of the Proof of \Cref{thm-var-d}}

\Cref{thm-var-d} is an immediate consequence of the following Propositions~\ref{proposition:dp step 1} and \ref{prop-2}.

\begin{proposition}\label{proposition:dp step 1}
Under all the conditions in \Cref{thm-var-d}, the following holds with probability at least $1 - (n \vee p)^{-1}$. 
	\begin{itemize}
		\item [(i)] For each interval $\widehat I = (s, e] \in \widehat{\mathcal{P}}$ containing one and only one true change point $\eta$, it holds that
			\[
				\min\{e - \eta,\, \eta - s\} \leq C_{\epsilon} \left(\frac{d_0 ^2 \lambda^2 + \gamma}{\kappa^2} \right);
			\]
		\item [(ii)] for each interval $\widehat I = (s, e] \in \widehat{\mathcal{P}}$ containing exactly two true change points, say  $\eta_1 < \eta_2$, it holds that
			\[
				\max\{e - \eta_2, \eta_1 - s\} \leq C_{\epsilon} \left(\frac{d_0 ^2 \lambda^2 + \gamma}{\kappa^2} \right);
			\]  
		\item [(iii)] for any two consecutive intervals $\widehat I, \widehat {J} \in \widehat{ \mathcal P}$, the interval $\widehat I \cup \widehat{J}$ contains at least one true change point; and
		\item [(iv)] no interval $\widehat I \in \widehat{\mathcal{P}}$ contains strictly more than two true change points.
	 \end{itemize}
\end{proposition} 

The four cases in \Cref{proposition:dp step 1} are proved in Lemmas~\ref{lemma:one change point}, \ref{lemma:two change point}, \ref{lemma:no change point} and \ref{lemma:three change point}, respectively, and \Cref{proposition:dp step 1} is proved consequently.

\begin{proposition}\label{prop-2}
Under the same conditions of \Cref{thm-var-d}, if $K \leq |\widehat{\mathcal{P}}| \leq 3K$, then with probability at least $1 - (n \vee p)^{-1}$, it holds that $|\widehat{\mathcal{P}}| = K + 1$.	
\end{proposition}

\begin{proof}[Proof of \Cref{thm-var-d}]
	It follows from \Cref{proposition:dp step 1} that, $K \leq |\widehat{\mathcal{P}}| - 1\leq 3K$.  This combined with \Cref{prop-2} completes the proof.
\end{proof}

The key insight for the DP approach is that, for an appropriate value of $\lambda$, the estimator $\widehat A^\lambda_{I} $ in \eqref{eq:VARD oracle} based on {\it any} time interval $I$ is a ``uniformly good enough'' estimator, in the sense that it is sufficiently close  to its population counterpart regardless of the choice of $I$ and even if $I$ contains multiple true change points.  For instance, if $I$ contains three change points, then $I$ should not be a member in $\widehat{\mathcal{P}}$.  This is guaranteed by the following arguments.  Provided that $\widehat A^\lambda_{I}$ is close enough to its population counterpart, under the signal-to-noise ration condition in \Cref{assume:AR high dim coefficient}{\bf (c)}, breaking this interval $I$ will return a smaller objective function value in \eqref{eq-wide-p}.

In order to show that $\widehat A^\lambda_{I} $ is a ``good enough'' estimator, we take advantage of the assumed sparsity of its population counterpart $A^*_ {I} $ in \Cref{lemma:prorpulation AR sparsity}.  In particular, we show that $A^*_ {I}$ is the unique solution to the equation
 	\begin{align*} 
		\left(\sum_{t \in I} \mathbb{E}[X_t X_t^{\top}] \right)(A_I^*)^{\top} = \sum_{t \in I} \mathbb{E}[X_t X_t^{\top}](A_t^*)^{\top}.
	\end{align*}
	The above linear system implies that in general $A_I^*\not = |I|^{-1} \sum_{t\in I } A_t^*$, as the covariance matrix $\mathbb{E}[X_t X_t^{\top}]$ changes if $A_t^*$ changes at the change points.  However, we can   quantify the size of  support of $A_I^* $ on any generic interval $I$ as long as $\{A_t^*\}_{t \in I }$ are sparse.  With this at hand, we can further show that $\| \widehat A^\lambda_{I}  - A^*_I\|^2 _2 = O_p(|I|^{-1})$, which is what makes $\widehat A^\lambda_{I} $ a ``good enough'' estimator for the purposes of change point localization. 
  
The key ingredients of the proofs of both Propositions~\ref{proposition:dp step 1} and \ref{prop-2} are two types of deviation inequalities as follows. 
	\begin{itemize}
		\item \textbf{Restricted eigenvalues.} In the literature on high-dimensional regression problems, there are several versions of the restricted eigenvalue conditions \citep[see, e.g.][]{buhlmann2011statistics}.  In our analysis, such conditions amount to controlling the probability of the event
		\begin{align*} 
			\left\{\|B X_{t}\|_2^2 \ge \frac{|I|c_x }{4} \|B \|_2^2 - C_x \log(p)     \|B\|_1^2, \quad \forall B \in \mathbb R^{p\times p}, \, I \subset \{1, \ldots, n\} \right\}.
	 	\end{align*} 		
		\item \textbf{Deviations bounds.}  In addition, we need to control the deviations of the quantities of the form
			\begin{align*}
			 \left\|\sum_{t \in I} \varepsilon_{t + 1} X_{t  }^{\top}\right\|_{\infty}  
			\end{align*}
  		for any interval $I \subset \{1, \ldots, n\}$.
 	\end{itemize}

Using well-established arguments to demonstrate the performance of the \lasso estimator, as detailed in e.g.~Section 6.2 of \cite{buhlmann2011statistics}, the combination of restricted eigenvalues conditions and large probability bounds on the noise lead to oracle inequalities for the estimation and prediction errors in situations where there exists no change point and the data are independent.  	In the existing time series analysis literature, there are several versions of the aforementioned bounds for stationary VAR models \citep[e.g.][]{basu2015regularized}.  We have extended this line of arguments to the present, more challenging settings, to derive analogous oracle inequalities.  See    \Cref{sec:inequalities} for more details. 
 
\subsection*{Sketch of the Proof of \Cref{theorem:lr}}

The proof of \Cref{theorem:lr} is based on an oracle inequality of the group \lasso estimator.  Once it is established that for each interval $(s_k, e_k]$ used in \Cref{algorithm:LR},
	\begin{equation}\label{eq-intuit-g-lasso}
		\sum_{t = s_k}^{e_k} \|\widehat{A}_t - A^*_t\|_2^2 \leq \delta \leq \kappa^2 \Delta,
	\end{equation}
	where $\delta \asymp d_0 \log(n \vee p)$ and that there is one and only one change point in the interval $(s_k, e_k]$ for both the sequence $\{\widehat{A}_t\}$ and $\{A^*_t\}$, then the final claim follows immediately that the refined localization error $\epsilon$ satisfies
	\[
		\epsilon \leq \delta/\kappa^2.
	\]
	The group \lasso penalty is deployed to prompt \eqref{eq-intuit-g-lasso} and the designs of the algorithm guarantee the desirability of each working interval.

\section{Numerical studies} \label{sec-numerical}

In this section, we conduct a numerical study of the performance of the DP approach of \Cref{algorithm:DP}, of the PGL \Cref{algorithm:LR} as well as some competing methods in a variety of different settings to support our theoretical findings.   We will first provide numerical comparisons in the simulated data experiments then in a real data example.  All the implementations of the numerical experiments can be found at  \href{http://www.sharelatex.com}{https://github.com/darenwang/vardp}.

We quantify the performance of the change point estimators  $\{\widehat \eta_k\}_{k=1}^{\widehat K}$ relatively to the set  $\{\eta_k\}_{k=1}^K $ of true change point using the absolute error $|\widehat{K} - K|$ and the scaled Hausdorff distance 
	\[
		\mathcal D(\{\widehat \eta_k\}_{k=1}^{\widehat K} , \{\eta_k\}_{k=1}^K) = \frac{d(\{\widehat \eta_k\}_{k=1}^{\widehat K} , \{\eta_k\}_{k=1}^K)}{n},
	\]
	where $ d(\cdot, \cdot)$ denotes the Hausdorff distance between two compacts sets in $\mathbb{R}$, given by
	\[
		d(A ,B) = \max\left\{   \max_{a  \in A } \min_{b\in B }   |a-b| , \, \max_{b\in B }  \min _{a  \in A }   |a-b|       \right\} .
	\]
	Note that if $\widehat{K}, K \ge 1$, then $\mathcal D(\{\widehat \eta_k\}_{k=1}^{\widehat K} , \{\eta_k\}_{k=1}^K) \le 1$.  For convenience we set $\mathcal D (\emptyset, \{\eta_k\}_{k=1}^K) = 1$. 

As discussed in \Cref{sec-introduction}, algorithms targeting at high-dimensional VAR change point detection are scarce.  In addition, we cannot provide any numerical comparisons with \cite{safikhani2017joint}, because their algorithm  is NP-hard in the worst-case scenario.  To be more precise, their combinatorial algorithm scales exponentially in $\widehat{K}$, which in general may be much bigger than $K$.  In all the numerical experiments, we compare our methods with the SBS-MVTS procedure proposed in \cite{cho2015multiple} and the INSPECT method proposed in \cite{wang2018high}.  The SBS-MVTS is designed to detect covariance changes in VAR processes, and the INSPECT detects mean change points in high-dimensional sub-Gaussian vectors.  We follow the default setup of the tuning parameters for the competitors, specified in the R \citep{R} packages wbs \citep{wbs-R} and InspectChangepoint \citep{inspect-R}, respectively.

For the tuning parameters used in our approaches, throughout this section, let $\lambda = 0.1  \sqrt {\log(p)}  $, $\gamma=15\log(n)p$ and  $ \zeta=0.3  \sqrt {\log(p)} $.  

\subsection{Simulations} \label{sec:numerical}

We generate data according to \Cref{assume:AR change point} in three settings, each consisting of multiple setups, and for each setup we carry out 100 simulations.  These three settings are designed to have varying minimal spacing $\Delta$, minimal jump size $\kappa$ and dimensionality $p$, respectively and are as follows:

\begin{itemize}
\item [(i)] Varying $\Delta$. Let $n  \in  \{ 100, 120, 140, 160 \} $, $p = 10$, $\sigma_{\epsilon} = 1$ and $L = 1$. The only  change point  occurs at $n/2$.  The coefficient matrices  are defined as 
	\begin{align*}
		A^*_t = \begin{cases}   
			\begin{pmatrix}
				0.3 & -0.3  &  &   \\
			 	&  \ddots    & \ddots  &  \\
		      	& &  0.3  & -0.3  \\
		         & &    & 0.3  	
			\end{pmatrix}, & t \in \{1, \ldots, n/2-1\},  \\
			\begin{pmatrix}
				-0.3 & 0.3 &  &   \\
				 &  \ddots    & \ddots  &  \\
			      & &  -0.3  & 0.3  \\
		         & &    &  -0.3   	
			\end{pmatrix}, & t \in \{n/2, \ldots, n\},
		\end{cases} 
	\end{align*}
	where the omitted entries that are zero.
	
\item [(ii)] Varying $\kappa$. Let $n =240 $, $p  =20 $, $\sigma_{\epsilon} = 1$ and $L = 1$.  The change points occur at $n/3$ and $2n/3$.  The coefficient matrices are
	\begin{align}\label{eq:transition matrix in 2}
		A_t^* = \begin{cases}  
			\rho (v, -v, 0_{p-2}), & t \in \{1, \ldots, n/3-1\}, \\
		    \rho (-v, v, 0_{p-2}), & t \in \{n/3, \ldots, 2n/3-1\}, \\
			\rho (v, -v, 0_{p-2}), & t \in \{2n/3, \ldots, n\}, 
		\end{cases}
	\end{align}
	where $v \in \mathbb R^p$ has odd coordinates equal 1 and even coordinates equal to $-1$, $0_{p-2} \in \mathbb{R}^{p \times (p-2)}$ is an all zero matrix and $\rho \in \{0.05,0.1, 0.15, 0.2,0.25 \}$.

\item [(iii)]  Varying $p$.  Let $T = 240$, $p \in \{  15, 20,25,30, 35\}$, $\sigma_{\epsilon} = 1$ and $L = 1$.  The change points occur at $n/3$ and $2n/3$.  The coefficient matrices are
	\begin{align*}
		A_t^* = \begin{cases}
			(v_1, v_2 , v_3, 0_{p-3}), & t \in \{1, \ldots, n/3-1\}, \\
			(v_2, v_3 , v_1, 0_{p-3}), & t \in \{n/3, \ldots, 2n/3-1\}, \\
			(v_3, v_2 , v_1, 0_{p-3}), & t \in \{2n/3, \ldots, n\}, 
		\end{cases}	
	\end{align*}
	where $v_1, v_2, v_3 \in \mathbb{R}^p$ with 
	\[
		v_1 = (-0.15 ,  0.225, 0.25 , -0.15 , 0  ,\ldots,0)^\top, \quad v_2 = (0.2 ,  -0.075,  -0.175 ,  -0.05 ,0 ,\ldots,0 )^\top
	\]
	and 
	\[
		v_3 = (-0.15 ,  0.1, 0.3 ,  -0.05 , 0 ,\ldots,0)^\top.
	\]	
\end{itemize}

\begin{table}[htbp!]
\begin{center}
\begin{tabular}{cccccc}
Setting  & Metric &  PDP & PGL  & SBS-MVTS  & INSPECT  \\
\hline \hline
    (i)  n=100  & $\mathcal D$   &0.000 (0.000)       &    0.000 (0.000)  &0.840 (0.348) &  1.000   (0.000)   \\
 (i) n=120  &  &   0.000 (0.000)      &  0.000 (0.000)   & 0.600 (0.457)  &   0.989 (0.081) \\
  (i) n=140 &  &     0.000 (0.000)    &   0.000 (0.000) & 0.536 (0.470) &   0.988 (0.082)\\
  (i) n=160  &  &   0.000 (0.000)      &    0.000 (0.000)  &   0.367  (0.432)  &   0.995   (0.051) \\
 \hline
  (i)  n=100  & $ |\widehat{K} - K|$   &0.000 (0.000)       &    0.000 (0.000)  &0.820 (0.386) &  1.000   (0.000)   \\
(i) n=120  &     &   0.000 (0.000)      &  0.000 (0.000)   & 0.560 (0.499)  &   0.980 (0.141) \\
  (i) n=140  &  &     0.000 (0.000)    &   0.000 (0.000) & 0.500 (0.502) &   0.990 (0.100)\\
   (i) n=160   &  &   0.000 (0.000)      &    0.000 (0.000)  &   0.310  (0.465)  &   0.990   (0.100) \\
  \hline\hline
   (ii) $\rho =0.05 $  & $\mathcal D$  &    0.049 (0.026)   &  0.041  (0.030)   &     0.990 (0.076) & 1.000 (0.000)     \\
  (ii) $\rho =0.10$ &   &      0.035(0.027)  & 0.022 (0.027)    & 0.994 (0.059)   &  1.000 (0.000)	  \\
  (ii) $\rho =0.15 $ &  &     0.022 (0.024)   &  0.009 (0.017)   &  1.000 (0.000)	&  0.996 (0.043)   \\
  (ii) $\rho =0.20$ & &   0.010  (0.018)     &  0.004 (0.012)   &   0.995 (0.054)&   0.955 (0.147)  \\
  (ii) $\rho =0.25$ &  &   0.004  (0.013)     &  0.002 (0.009)   &   1.000 (0.000)&   0.831 (0.260)  \\
 \hline
   (ii) $\rho =0.05 $   & $ |\widehat{K} -K|$  &    0.000 (0.000)   &   0.000 (0.000)   &     1.980 (0.140) & 2.000 (0.000)     \\
      (ii) $\rho =0.10 $   & &     0.000 (0.000) &   0.000 (0.000)   &     1.990 (0.100) & 2.000 (0.000)     \\
 (ii) $\rho =0.15 $ &  &    0.000 (0.000)  &  0.000 (0.000)  &     2.000 (0.000) & 1.999 (0.100)     \\
 (ii) $\rho =0.20$ &   &    0.000 (0.000)    &  0.000 (0.000)  &   1.990 (0.100)&   1.890 (0.373)  \\
(ii) $\rho =0.25$ &   &   0.000 (0.000)    &  0.000 (0.000)   &   2.000 (0.000)&   1.650 (0.575)  \\
 \hline\hline
    (iii) $p =15 $  & $\mathcal D$  &    0.049 (0.034)   &  0.026 (0.031)   &     1.000 (0.000) & 1.000(0.000)     \\
  (iii) $p=20$ &  &      0.059 (0.031)  & 0.036 (0.032)    & 1.000 (0.000)   &  1.000 (0.000)	  \\
  (iii) $p=25$ &   &     0.057 (0.028)   &  0.034 (0.029)   &  1.000 (0.000)	&  0.992(0.073)   \\
  (iii) $p= 30$ &  &   0.058 (0.028)     &  0.027 (0.030)   &   1.000 (0.000)&    1.000  (0.000)  \\
 (iii) $p= 35$ &  &   0.043 (0.024)     &  0.026 (0.031)   &   0.986 (0.097)&   0.987 (0.092)  \\
\hline 
 (iii) $p=15 $   & $ |\widehat{K} -K|$    &    0.000 (0.000)   &   0.000 (0.000)   &     2.000 (0.000) & 2.000 (0.000)     \\
 (iii) $p=20 $   & &     0.000 (0.000) &   0.000 (0.000)   &     2.000 (0.000) & 2.000 (0.000)     \\
  (iii) $p=25$ &   &    0.000 (0.000)  &  0.000 (0.000)  &     2.000 (0.000) & 1.990 (0.100)     \\
 (iii) $p= 30$ &  &    0.000 (0.000)    &  0.000 (0.000)  &   2.000 (0.000)&   2.000 (0.000)  \\
 (iii) $p= 35$ &   &    0.000 (0.000)    &  0.000 (0.000)  &   1.980 (0.141)&   1.980 (0.141)  
\end{tabular}
\end{center}
\caption{Simulation results.  Each cell is based on 100 repetitions and is in the form of mean(standard error). \label{tab-sim}}
\end{table}

\begin{figure}
\includegraphics[width = \textwidth]{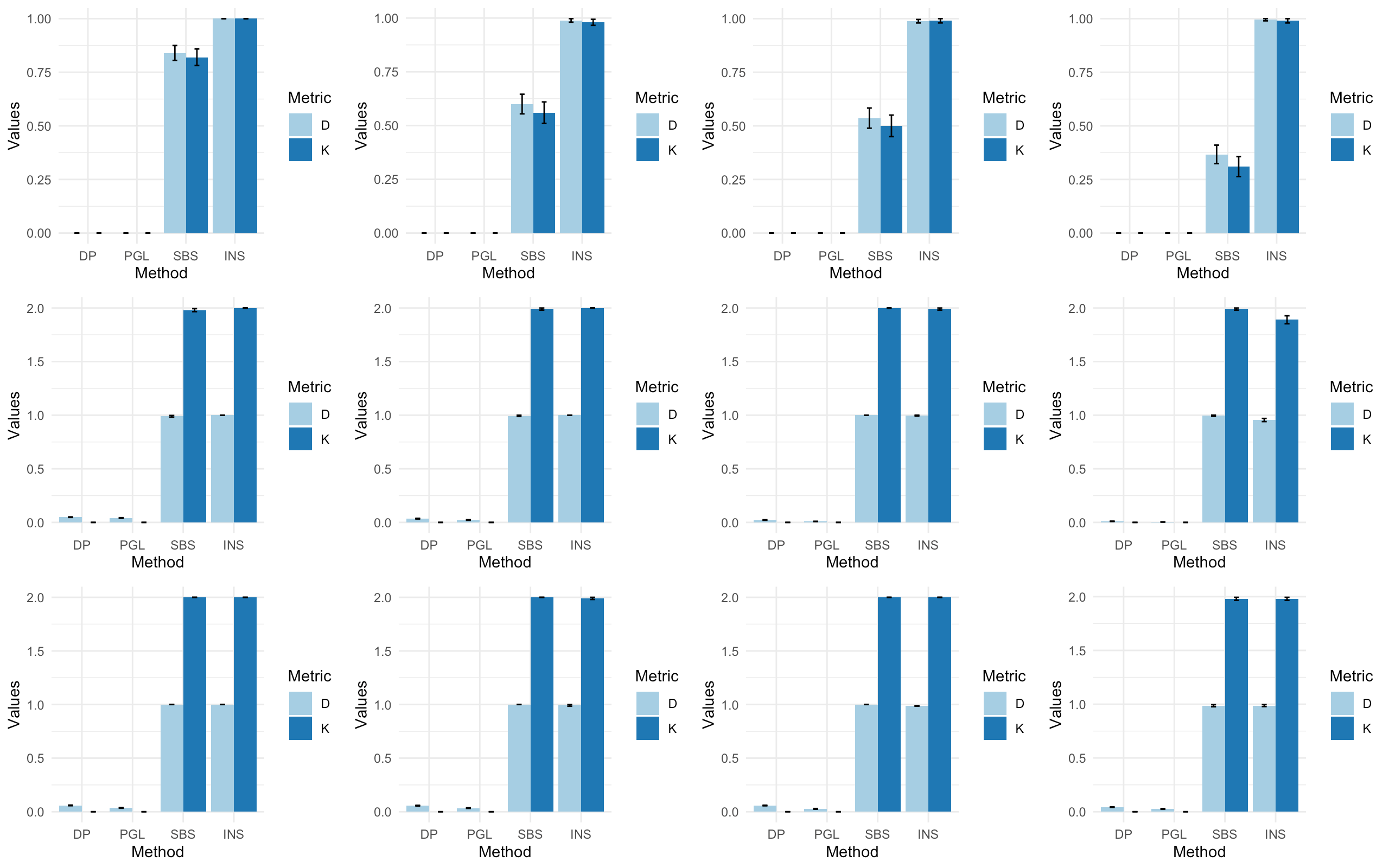}	
\caption{Bar plots visualizing the results collected in \Cref{tab-sim}.  From left to right and from top to bottom are Setting (i) with $n \in \{100, 120, 140, 160\}$, Setting (ii) with $\rho \in \{0.05, 0.1, 0.15, 0.2\}$ and Setting (iii) with $p \in \{20, 25, 30, 35\}$, respectively.  The four methods are: DP, \Cref{algorithm:DP}; PGL, \Cref{algorithm:LR}, SBS, SBS-MVTS; INS, Inspect.  The two metrics are: D, $\mathcal{D}$; K, $ |\widehat{K} -K|$. \label{fig-simulations}}
\end{figure}

The numerical comparisons of all algorithms in terms of the absolute error $|\widehat{K} - K|$ and the scaled Hausdorff distance $\mathcal D$ are reported in  \Cref{tab-sim} and are also displayed in \Cref{fig-simulations}.  We can see that even though the dimension $p$ is moderate, neither  SBS-MVTS  nor  INSPECT can consistently estimate the change points. In fact, both algorithms tend to infer that there is no change points in the time series, confirming the intuition discussed in \Cref{eq:why transition matrix}.

In all of our approaches, we need to estimate the coefficient matrices, therefore  the effective dimensions are $p^2$, but not $p$.  The competitors SBS-MVTS and INSPECT, however, estimate the change points without estimating the coefficient matrices.  This reduces the effective dimension to $p$, which might be more efficient in terms of computation, but neither SBS-MVTS nor INSPECT can consistently estimate the change points in these settings.

\subsection{A real data example: S\&P 100}

We study the daily closing prices of the constituents of S\&P 100 index from Jan 2016 to Oct 2018.  The companies in the S\&P 100 are selected for sector balance and represent about 51\% of the market capitalization of the U.S.~equity market.  They tend to be the largest and most established firms in   the U.S.~stock market.  Since the stock prices always exhibit  upward trends, we therefore follow the  standard procedure and de-trend the data by taking the first order difference.  After removing missing values, our final dataset is  a multivariate time series with $n= 700$ and $p = 93$, where each dimension corresponds the daily stock price change of one firm during the aforementioned period. 
   
In order to apply our algorithms, we rescale the data so that the average variance over all dimensions equals 1.  We apply the DP and PGL algorithms to the data set and depict the detected change points in \Cref{fig:sp100change1}. Since  the outputs of these two algorithms do not differ by much, we focus on the ones output by DP.  The five estimated  change points by DP  are 3rd November 2016, 23rd March 2017, 15th August 2017, 29th December 2017 and 10th May 2018.  There are some times, e.g. around Feb. 2018, where we see big changes in the S\&P 100  index but none of the methods describes it as a change point. This is because our change point analysis is targeting changes in the interactions among companies, as captured by a VAR process, and not directly the mean price changes. Thus occasional large fluctuations in the data do not necessarily reflect the type of structural changes we seek to identify.

To argue that our methodology has lead to meaningful findings, we suggest a justification of each estimated change points.  Four days after the first change point estimator, Trump won the presidential election.  Eight days after the second change point estimator, President Trump signed two executive orders increasing tariffs.  This is a key date in the U.S.-China trade war.  The third change point estimator is associated with  President Trump signing another executive order in March 2017, authorizing U.S.~Trade Representatives to begin investigations into Chinese trade practices, with particular focus on intellectual property and advanced technology.  The fourth change point estimator lines up with the best opening of the  U.S.~stock market in 31 years \citep[e.g.][]{CNBC}.  The last change point estimator is associated with the White House announcement  that it would impose a 25\% tariff on \$50 billion of Chinese goods with industrially significant technology at the end of May 2018.  For comparison, we also show the change points found by SBS-MVTS and INSPECT  in \Cref{fig:sp100change2}. 
 
\begin{figure}[H]
\begin{center}
\subfloat[]{\includegraphics[width= 0.8\linewidth]{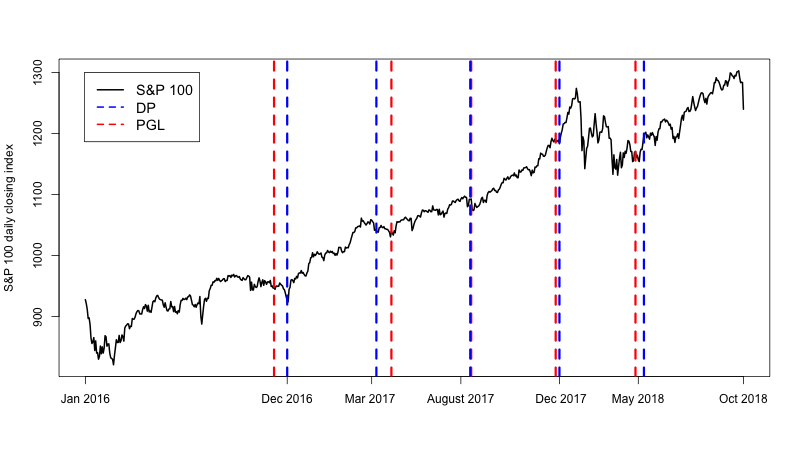}\label{fig:sp100change1}}~
\\
\subfloat[]{\includegraphics[width=0.8\linewidth]{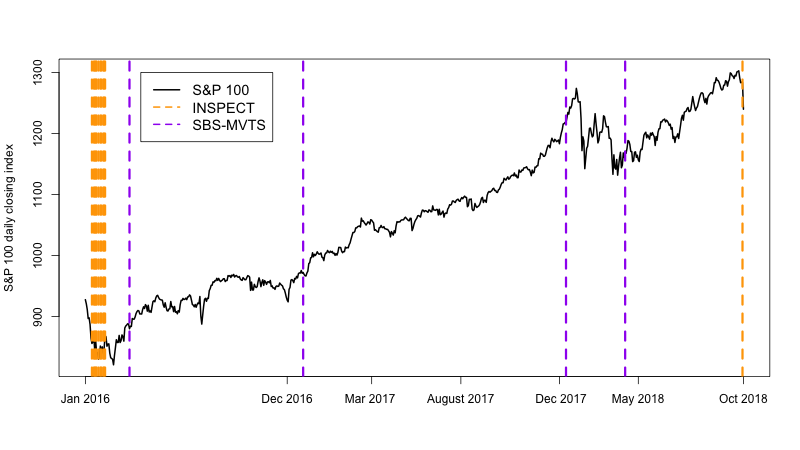}\label{fig:sp100change2}}
	\caption{S\&P 100 example. (a) Change points  estimated  by the DP and PGL algorithms. (b) Change points  estimated  by SBS-MVTS and INSPECT algorithms.  While the plot is only showing the daily closing index, our method is using the first order difference  of $p=93$ stocks that comprise the index. All the change point algorithms analyzed in this section focus on   the interactions among those 93 companies, {\it not} in   changes in the S\&P 100 stock index. }
	
	\end{center}
\end{figure}

\section{Conclusions}

This paper considers change point localization in high-dimensional vector autoregressive process models.  We have developed two procedures for change point localization that can be characterized as solutions to a common optimization framework and that can be efficiently implemented using a combination of dynamic programming and \lasso-type estimators.  We have demonstrated that our methods yield the sharpest localization rates of autoregressive processes and match the best known rates for change point localization. We further conjecture that the localization rate of the PGL algorithm is minimax optimal.  Both  minimax rates and extensions of this framework beyond sparse models to other models of low-dimensional structure remain important open questions for future research.

\bibliographystyle{ims}
\bibliography{citations}

\appendix 
\section*{Appendix}
In the proofs, we do not track every single absolute constant.  Except those which have specific meanings, used in the main text, all the others share a few pieces of notation, i.e.~the constant $C$ are used in different places having different values.
	 
\section{Proof of \Cref{thm-var-d}}	\label{appendix:var 1} 
\input{appendixA}

\section{Proof of \Cref{theorem:lr}}

\begin{lemma}
Let $\mathcal{R}$ be any linear subspace in $\mathbb{R}^n$ and $\mathcal{N}_{1/4}$	be a $1/4$-net of $\mathcal{R} \cap B(0, 1)$, where $B(0, 1)$ is the unit ball in $\mathbb{R}^n$.  For any $u \in \mathbb{R}^n$, it holds that
	\[
		\sup_{v \in \mathcal{R} \cap B(0, 1)} \langle v, u \rangle \leq 2 \sup_{v \in \mathcal{N}_{1/4}} \langle v, u \rangle,
	\]
	where $\langle \cdot, \cdot \rangle$ denotes the inner product in $\mathbb{R}^n$.
\end{lemma}

\begin{proof}
Due to the definition of $\mathcal{N}_{1/4}$, it holds that for any $v \in \mathcal{R} \cap B(0, 1)$, there exists a $v_k \in \mathcal{N}_{1/4}$, such that $\|v - v_k\|_2 < 1/4$.  Therefore,
	\begin{align*}
		\langle v, u \rangle = \langle v - v_k + v_k, u \rangle = \langle x_k, u \rangle + \langle v_k, u \rangle \leq \frac{1}{4} \langle v, u \rangle + \frac{1}{4} \langle v^{\perp}, u \rangle + \langle v_k, u \rangle,
	\end{align*}
	where the inequality follows from  $x_k = v - v_k = \langle x_k, v \rangle v + \langle x_k, v^{\perp} \rangle v^{\perp}$.  Then we have
	\[
		\frac{3}{4}\langle v, u \rangle \leq \frac{1}{4} \langle v^{\perp}, u \rangle + \langle v_k, u \rangle.
	\]
	It follows from the same argument that 
	\[
		\frac{3}{4}\langle v^{\perp}, u \rangle \leq \frac{1}{4} \langle v, u \rangle + \langle v_l, u \rangle,
	\]
	where $v_l \in \mathcal{N}_{1/4}$ satisfies $\|v^{\perp} - v_l\|_2 < 1/4$.  Combining the previous two equation displays yields
	\[
		\langle v, u \rangle \leq 2 \sup_{v \in \mathcal{N}_{1/4}} \langle v, u \rangle,
	\]
	and the final claims holds.
\end{proof}

\begin{lemma}\label{lem-wang-lem-3}
	For data generated from \Cref{assume:AR high dim coefficient}, for any interval $I = (s, e] \subset \{0, \ldots, n\}$, it holds that for any $\tau > 0$, $i \in \{1, \ldots, p\}$ and $1\le l \le L$, 
	\begin{align*}
		\mathbb{P}\left\{\sup_{\substack{v \in \mathbb{R}^{(e-s)}, \, \|v\|_2 = 1\\ \sum_{t = 1}^{e-s-1} \mathbbm{1}\{v_i \neq v_{i+1}\} = m}} \left|\sum_{t = s+1}^e v_t \varepsilon_{t+1}(j) X_{t+1-l} (i)\right| > \tau \right\} \\
		\leq C(e-s-1)^m 9^{m+1} \exp \left\{-c \min\left\{\frac{\tau ^2}{4C_x^2}, \, \frac{\tau }{2C_x \|v\|_{\infty}}\right\}\right\}.
	\end{align*}
\end{lemma}

\begin{proof}
We show that \Cref{lem-wang-lem-3} holds for $l=1$, as the rest of cases are exactly the same.
For any $v \in \mathbb{R}^{(e-s)}$ satisfying $\sum_{t = 1}^{e-s-1}\mathbbm{1}\{v_i \neq v_{i+1}\} = m$, it is determined by a vector in $\mathbb{R}^{m+1}$ and a choice of $m$ out of $(e-s-1)$ points.  Therefore we have,
	\begin{align*}
		& \mathbb{P}\left\{\sup_{\substack{v \in \mathbb{R}^{(e-s)}, \, \|v\|_2 = 1\\ \sum_{t = 1}^{e-s-1} \mathbbm{1}\{v_i \neq v_{i+1}\} \le  m}} \left|\sum_{t = s+1}^e v_t \varepsilon_{t+1}(j)  X_t(i)\right| > \delta \right\} \\
		\leq & {(e-s-1) \choose m} 9^{m+1} \sup_{v \in \mathcal{N}_{1/4}}	\mathbb{P}\left\{\left|\sum_{t = s+1}^e v_t \varepsilon_{t+1}(j) X_t(i)\right| > \delta/2 \right\} \\
		\leq & {(e-s-1) \choose m} 9^{m+1} C\exp \left\{-c \min\left\{\frac{\delta^2}{4C_x^2}, \, \frac{\delta}{2C_x \|v\|_{\infty}}\right\}\right\} \\
		\leq & C(e-s-1)^m 9^{m+1} \exp \left\{-c \min\left\{\frac{\delta^2}{4C_x^2}, \, \frac{\delta}{2C_x \|v\|_{\infty}}\right\}\right\},
	\end{align*}
where the  second inequality follows a similar argument as that in  \Cref{coro:restricted eigenvalue AR}.
\end{proof}

\begin{proof}[Proof of \Cref{theorem:lr}]
For each $k \in \{1, \ldots, K\}$ and $1\le l \le L$, let
	\[
		 \widehat{\alpha}_t [l] = \begin{cases}
 				\widehat{A} [l], &\text{when }  t \in \{s_k + 1, \ldots, \widehat{\eta}_k\}, \\
 				\widehat{ B}[l],   &\text{when }  t \in \{\widehat{\eta}_k + 1, \ldots, e_k\}.
 			\end{cases}		
	\]
	Without loss of generality, we assume that $s_k < \eta_k < \widehat{\eta}_k \leq e_k$.  We proceed the proof discussing two cases.

\vskip 3mm
\noindent \textbf{Case (i).}  If 
	\[
				\widehat{\eta}_k - \eta_k \leq C  d_0    \log(n \vee p) /\kappa^2 ,
	\]	
	then the result holds.

\vskip 3mm
\noindent \textbf{Case (ii).}  Suppose
	\begin{equation}\label{eq-pf-cor-6-int-con}
		\widehat{\eta}_k - \eta_k \geq  Cd_0   \log(n \vee p)/\kappa^2  
	\end{equation} 
	Since $\kappa $ is bounded, $  \widehat{\eta}_k - \eta_k  \ge  C d_0    \log(n \vee p)   $.
	\\
	\\
{\bf Step 1.} 	Let $\Delta_t = \widehat{\alpha}_t - A^*_t$ and therefore 
 it holds that
	\[
		\sum_{t = s_k + 1}^{e_k - 1}\mathbbm{1}\left\{\Delta_t \neq \Delta_{t+1}\right\} = 3.
	\]	
	We first to prove that with probability at least $1 - C(n \vee p)^{-6}$,
	\[
		\sum_{t = s_k + 1}^{e_k} \sum_{l=1}^L \|\Delta_t  [l]  \|_2^2 \leq Cd_0 \zeta^2  .
	\]	
%We show this by contradiction, suppose that 
%\[
%		\sum_{t = s_k + 1}^{e_k} \sum_{l=1}^L \|\Delta_t[l]  \|_2^2 > C_1d_0 \zeta^2  .
%	\]	
%	Then it holds that  for any $i,j$
%	\begin{align} \label{eq:linfty delta}
%	 \frac{\Delta_t[l]_{i,j }}{ \sum_{t = s_k + 1}^{e_k} \sum_{l=1}^L \|\Delta_t[l]  \|_2^2 } \le \frac{c_1}{ d_0 \zeta^2} ,
%	\end{align}
%	for some $c_1$ depending on $C_1$.
Due to \eqref{eq-g-lasso}, it holds that
	\begin{align}\label{eq-lem19-pf-1}
		& \sum_{t = s_k + 1}^{e_k} \sum_{l=1}^L  \|X_{t+1 } - \widehat{\alpha}_t [l]  X_{t+1-l}  \|_2^2 + \zeta  \sum_{l=1}^L \sum_{i ,j = 1}^p \sqrt{\sum_{t = s_k + 1}^{e_k} \bigl(\widehat{\alpha}_t [l]_{i,j} \bigr)^2} \nonumber\\
		 \leq & \sum_{t = s_k + 1}^{e_k} \sum_{l=1}^L  \|X_{t+1  } -  A ^* _t [l]X_ {t+1-l}  \|_2^2 + \zeta  \sum_{l=1}^L \sum_{i ,j = 1}^p \sqrt{\sum_{t = s_k + 1}^{e_k} \bigl( A ^*_t[l] _{i,j} \bigr)^2}.
	\end{align}
	which implies that
	\begin{align}\label{eq-lem19-pf-2}
		\sum_{t = s_k + 1}^{e_k}  \sum_{l=1}^ L\|X_{t+1-l}\Delta_t[l]  \|_2^2 + \zeta  \sum_{l=1}^L \sum_{i ,j = 1}^p \sqrt{\sum_{t = s_k + 1}^{e_k} \bigl(\widehat{\alpha}_t [l] _{i,j} \bigr)^2} \nonumber \\
		\leq 2\sum_{l=1}^L \sum_{t = s_k + 1}^{e_k}   \varepsilon_{t+1} ^\top \Delta_tX_{t+1-l}  + \zeta  \sum_{l=1}^L \sum_{i ,j = 1}^p \sqrt{\sum_{t = s_k + 1}^{e_k} \bigl(A^*_t [l] _{i,j}^ 2\bigr) }.
	\end{align}

Note that
	\begin{align}
		&    \sum_{l=1}^L \sum_{i ,j = 1}^p \sqrt{\sum_{t = s_k + 1}^{e_k} \bigl(A^*_t [l]_{i,j}\bigr) ^2} 
		-   \sum_{l=1}^L \sum_{i ,j = 1}^p \sqrt{\sum_{t = s_k + 1}^{e_k} \bigl(\widehat \alpha_t [l] _{i,j} \bigr) ^2}.  \nonumber \\
		 = & \sum_{l=1} ^L   \left( \sum_{(i,j)  \in S} \sqrt{\sum_{t = s_k + 1}^{e_k} \bigl(A^*_t [l]  _{i,j} \bigr) ^2} - \sum_{ ( i, j) \in S} \sqrt{\sum_{t = s_k + 1}^{e_k} \bigl(\widehat{\alpha}_t[l]  _{i,j}\bigr) ^2}  - \sum_{ ( i , j )\in S^c} \sqrt{\sum_{t = s_k + 1}^{e_k} \bigl(\widehat{\alpha}_t[l]  _{i, j }\bigr) ^2}   
		 \right) 
		  \nonumber \\
		\leq & \sum_{l=1} ^L   \left(   \sum_{ ( i  ,j) \in S} \sqrt{\sum_{t = s_k + 1}^{e_k} \bigl(\Delta_t [l]  _{i,j } \bigr) ^2}  
		-
		 \sum_{(i,j)  \in S^c} \sqrt{\sum_{t = s_k + 1}^{e_k} \bigl(\Delta_t[l]  _{i, j } \bigr) ^2} \right).
		\label{eq-lem19-pf-3}
	\end{align} 

We then examine the cross term, with probability at least $1 - C(n \vee p)^{-6}$, which satisfies the following
	\begin{align}
		& \left|\sum_{l=1}^L \sum_{t = s_k + 1}^{e_k}   \varepsilon_{t+1} ^\top \Delta_t [l] X_{t+1-l}   \right| 
		 \nonumber  \\
		= &  \sum_{i ,j= 1}^p \left\{\left| \sum_{l=1}^L \frac{\sum_{t = s_k + 1}^{e_k} \varepsilon_{t+1}(i) \Delta_t [l]_{i, j  }  X_{t+1 -l}(j)}{\sqrt{ \sum_{l=1}^L \sum_{t = s_k + 1}^{e_k} \left(\Delta_t[l]_{ i, j } \right)^2}}\right| \sqrt{ \sum_{l=1}^L \sum_{t = s_k + 1}^{e_k} \left(\Delta_t[l]_{ i, j } \right)^2}\right\}
		 \nonumber 
		 \\
		\leq & L  \sup_{1\le i ,j\le p , 1\le l \le L } 
\left|\frac{\sum_{t = s_k + 1}^{e_k} \varepsilon_{t+1}(i ) \Delta_t [l]_{i , j  }  X_{t+1 -l}( j)}{\sqrt{\sum_{t = s_k + 1}^{e_k} \left(\Delta_t[l]_{i  , j } \right)^2}}\right| \sum_{l=1}^L \sum_{i ,j= 1}^p \sqrt{\sum_{t = s_k + 1}^{e_k} \left(\Delta_t[l]_{ i, j } \right)^2} 
		 \nonumber \\
		 \leq & (\zeta/4) \sum_{l=1}^L \sum_{i ,j= 1}^p   \sqrt{\sum_{t = s_k + 1}^{e_k} \left(\Delta_t[l]_{i,j}\right)^2}, \label{eq-lem19-pf-4}
	\end{align}
	where the first inequality follows from the inequality that $\sqrt {a+b} \le \sqrt a + \sqrt b  $ and the second inequality follows from \Cref{lem-wang-lem-3}.% and \eqref{eq:linfty delta}.

Combining \eqref{eq-lem19-pf-1}, \eqref{eq-lem19-pf-2}, \eqref{eq-lem19-pf-3} and \eqref{eq-lem19-pf-4} yields  
	\begin{equation}\label{eq-lem19-pf-5}
		\sum_{t = s_k + 1}^{e_k}  \sum_{l=1}^L \|\Delta_t [l] X_{t+1-l}  \|_2^2	 + \frac{\zeta}{2} \sum_{l=1}^L  \sum_{( i , j)\in S^c} \sqrt{\sum_{t = s_k + 1}^{e_k} \bigl(\Delta_t[l] _{i,j }\bigr) ^2} \leq \frac{3\zeta}{2}\sum_{l=1}^L \sum_{ ( i ,j)\in S}    \sqrt{\sum_{t = s_k + 1}^{e_k} \bigl(\Delta_t[l] _{i,j }\bigr) ^2} 
	\end{equation}

Now we are to explore the restricted eigenvalue inequality.  Let
	\[
		I_1 = (s_k, \eta_k], \quad I_2 = (\eta_k, \widehat{\eta}_k], \quad I_3 = (\widehat{\eta}_k, e_k].
	\]
	Due to \eqref{eq-lr-cond} and \Cref{coro:restricted eigenvalue AR} we have that with probability at least $1 - n^{-6}$,
	\begin{align}	\label{eq:size of interval PGL} 
		\min\{|I_1|, \, |I_3|\} > (1/3)\Delta >    (C_{SNR}/3)  d_0  K \log(n \vee p) .
 	\end{align}
	Therefore, with probability at least $1 - 3(n \vee p)^{-6}$,
	\begin{align} \nonumber 
		& \sum_{l=1}^L  \sum_{t = s_k + 1}^{e_k} \|\Delta_t [l] X_{t+1-l} \|_2^2 
		=   \sum_{l=1}^L \sum_{m = 1}^3 \sum_{t \in I_ m }\|\Delta_{I_ m }[l] X _{t+1-l} \|_2^2
		\\ \nonumber 
		\geq  & \sum_{ m = 1 }  ^3  \sum_{l=1}^L  \left(  \frac{c_x |I_m | }{2} \|\Delta_{I_m }[l] \|_2 ^2  -  C_x  \log(p)  \|\Delta_{I_m}[l] \|_1 ^2  \right) 	
		\\ \label{eq:PGL step 3}
				= &    \frac{c_x  }{2}    \sum_{l=1}^L \sum_{t = s_k + 1}^{e_k} \|\Delta_t [l] X_{t+1-l} \|_2^2 
				-
\sum_{t=s_k+1 }^{ e_k}		\sum_{l=1}^L 		  C_x  \log(p)  \|\Delta_{t}[l] \|_1 ^2  
	\end{align}
	where the last inequality follows from \Cref{coro:restricted eigenvalue AR}{(\bf c)}, \eqref{eq-pf-cor-6-int-con}  and   \eqref{eq:size of interval PGL}.
In addition, note that for any $1\le l\le L$, the square root of the second term of \eqref{eq:PGL step 3} can be bounded as 
	\begin{align*}
	 \sum_{t=s_k+1 }^{ e_k}		 	   \|\Delta_{t}[l] \|_1 ^2  
	=  & \sum_{t=s_k+1 }^{ e_k}		 \left(  	\sum_{1\le i,j \le p  }    |\Delta_{t}[l]_{i,j}  | \right) ^2    
	\le  \left( \sum_{1\le i,j \le p  } \sqrt {    \sum_{t=s_k}^{e_k} (  \Delta_t  [l]_{i,j}  )^2  }      \right)^2 
\\	
	\le & 
	 2  \left( \sum_{ i,j  \in S  } \sqrt {   \sum_{t=s_k}^{e_k}  ( \Delta_t [l]_{i,j} )  ^2  }      \right)^2 
	 +	 2  \left( \sum_{ i,j  \in S^c   } \sqrt {   \sum_{t=s_k}^{e_k} ( \Delta_t  [l] _{i,j})^2    }      \right)^2 
	\end{align*}
where the last inequality follows from 
generalized H\ou lder inequality in \Cref{lemma:gholder}. 
Therefore, if $\zeta \ge 8 C_x \log(p)$, the previous display and \eqref{eq:PGL step 3} give 
\begin{align}    \nonumber 
		&  \sum_{l=1}^L  \sum_{t = s_k + 1}^{e_k} \|\Delta_t [l] X_{t+1-l} \|_2^2 
				\geq     \frac{c_x  }{2}    \sum_{l=1}^L \sum_{t = s_k + 1}^{e_k} \|\Delta_t [l]  \|_2^2  
				\\
				-&	   ( \zeta/4) \sum_{l=1}^L \left( \sum_{ i,j  \in S  } \sqrt {   \sum_{t=s_k}^{e_k}  ( \Delta_t [l]_{i,j} )  ^2  }      \right)^2   -
 ( \zeta/4)   \sum_{l=1}^L   \left( \sum_{ i,j  \in S^c   } \sqrt {   \sum_{t=s_k}^{e_k} ( \Delta_t  [l] _{i,j})^2    }      \right)^2 \label{eq:PGL step 31}
	\end{align}
Then \eqref{eq-lem19-pf-5} and \eqref{eq:PGL step 31} together imply that 
	\begin{align*} 
		&   \sum_{l=1}^L  \sum_{t = s_k + 1}^{e_k} \|\Delta_t [l] \|_2^2    	 + \frac{\zeta}{4} \sum_{l=1}^L  \sum_{( i , j)\in S^c} \sqrt{\sum_{t = s_k + 1}^{e_k} \bigl(\Delta_t[l] _{i,j }\bigr) ^2}  
		\\
		 \leq  & 2\zeta \sum_{l=1}^L \sum_{ ( i ,j)\in S}    \sqrt{\sum_{t = s_k + 1}^{e_k} \bigl(\Delta_t[l] _{i,j }\bigr) ^2} 
		\le   
     C L\zeta^2 d_0 	 +    \frac{1}{2}	  \sum_{l=1}^L  \sum_{t = s_k + 1}^{e_k} \|\Delta_t [l]  \|_2^2     
	\end{align*}  
	and therefore 
		\begin{equation*}\label{eq-lem19-pf-6}
		 \sum_{l=1}^L  \sum_{t = s_k + 1}^{e_k} \|\Delta_t [l]  \|_2^2     \le C d_0 \zeta^2.
	\end{equation*}  
\\
\\
{\bf Step 4.}
 Denote $A^*  = A^*_{I_1}$ and that $B^*  = A^*_{I_2} =  A^*_{I_3}$.  We have that
	\[
	 	 \sum_{l=1}^L  \sum_{t = s_k + 1}^{e_k} \|\Delta_t [l]  \|_2^2    = |I_1| \|A^*  - \widehat A \|_2^2 + |I_2| \| B^* -  \widehat A \|_2^2 + |I_3| \|B ^*  - \widehat B \|_2^2.
	\]
Therefore  we have that
	\[
		\Delta\|\beta^*_1 - \widehat{\beta}_1\|_2^2/3 \leq |I_1| \|\beta^*_1 - \widehat{\beta}_1\|_2^2 \leq C_1 d_0 \zeta^2 \leq \frac{C_1  \Delta \kappa^2} {C_{\mathrm{SNR}}  \log^{\xi}(n \vee p) } \leq  \frac{c_1}{48} \Delta \kappa^2,
	\]
  Therefore we have
	\[
		\|A^* - \widehat{A} \|_2^2 \leq  \kappa^2/16.
	\]
	Similarly, 
In addition we have
	\[
	\|B^* - \widehat{B} \|_2^2 \leq  \kappa^2/16.
	\]	
	Therefore, it holds that 	
	$$ \| B^* -  \widehat A \|_2  \ge \| A^*-B^* \|_2 - \|A^* - \widehat A  \|  \ge \kappa -  ( 1/4) \kappa  \ge ( 3/4) \kappa $$
	which implies that 	
	\[
		3\kappa^2 |I_2|/4 \leq |I_2| \|\beta^*_2 - \widehat{\beta}_1\|_2^2 \leq C d_0 \zeta^2,
	\]
This directly implies that 
	\[
		|I_2| = |\widehat{\eta}_k - \eta_k| \leq \frac{Cd_0 \zeta^2}{\kappa^2},
	\]
as desired. 
\end{proof}

\section{Technical lemmas}
\label{sec:inequalities}
\input{appendixD}

\subsection{Additional lemmas for general $L \in \mathbb{Z}_+$} 
 \label{appendix:varL}
 \input{appendixB}

\end{document}

%% file: appendixA.tex
%!TEX root = ./vardp.tex

As we mentioned in \Cref{sec-rm}, for any positive integer $L$, a VAR($L$) process can be written as a VAR(1) process.  Based on the transformation introduced in \Cref{sec-trans}, in this section, we will first lay down the preparations, then prove Propositions~\ref{proposition:dp step 1} and \ref{prop-2} based on $L = 1$.  All the technical lemmas are deferred to \Cref{sec:inequalities}.

We let $A_I^*$ be the solution to
	\begin{align} \label{eq:ar population 1}
		\left(\sum_{t \in I} \mathbb{E}(X_t X_t^{\top}) \right)(A_I^*)^{\top} = \sum_{t \in I} \mathbb{E}(X_t X_t^{\top})(A_t^*)^{\top}.
	\end{align}
	Note that when $I$ contains no change points $A_I^* =A_t^*$, $t \in I$.  Let
	\[
		S_1 = \{i: \, (i,j) \in S\} \subset \{1, \ldots, p\} \quad \mbox{and} \quad S_2 =  \{j:\, (i,j) \in S\} \subset \{1, \ldots, p\}. 
	\]
	Therefore by \Cref{assume:AR high dim coefficient}(\textbf{a}), $\max\{|S_1|, \, |S_2|\} \le d_0$.  With a permutation if necessary, without loss of generality, we have that $S_1\cup S_2 \subset  \{1, \ldots, 2d_0\}$, which implies that each $A_t^*$ has the block structure 
	\begin{equation} \label{eq-astar-prep}
		A_t^* = \begin{pmatrix}
			\mathfrak{a}_{t}^* &   0 \\
			0 & 0 \\
		\end{pmatrix}    \in \mathbb{R}^{p \times p},
	\end{equation}
	where $\mathfrak{a}_{t}^*  \in \mathbb R^{2d_0\times 2d_0}$.  Denote 
	\begin{align}\label{eq:effective support}
		\mathcal S = (S_1 \cup S_2)^{\times 2} \subset  \{1, \ldots, 2d_0\}^{\times 2}
	\end{align}  
	satisfying that $|\mathcal S| \le 4 d_0^2$ and that $A^*_t (i,j) =0 $ if $(i,j ) \in \mathcal S^c$. 

\subsection{Proof of \Cref{proposition:dp step 1}}

\begin{proof}[Proof of \Cref{proposition:dp step 1}]

The four cases in \Cref{proposition:dp step 1} are straightforward consequences of applying union bound arguments to Lemmas~\ref{lemma:one change point}, \ref{lemma:two change point}, \ref{lemma:no change point} and \ref{lemma:three change point}, respectively.  

For illustration, we show how to apply the union bound argument to \Cref{lemma:one change point} and prove Case (i).  Consider the collection of integer intervals 
	\[
		\mathcal I = \{ I \subset \{1, \ldots, n\}: \, I \text{ satisfies all the conditions in \Cref{lemma:one change point} and } |I| \ge \gamma \}.
	\]
	With the notation in \Cref{lemma:one change point}, let $\mathcal E_I$ be 	
	\begin{align*} 
		\mathcal E_I = \Bigg\{\min\{|I_1 |,\, |I_2|\} \leq C_{\epsilon} \left(\frac{\lambda^2 d_0 ^2 + \gamma}{\kappa^2}\right)\Bigg\}.
	\end{align*}
	By \Cref{lemma:one change point} and the fact that $|\mathcal I| \le n^2$, it holds that 
	\[
		\mathbb P \left (\bigcup_{I \in \mathcal I}    \mathcal E  _I     \right)  \le (n\vee p)^{-3}.
	\]
	The proof is completed by noticing that \eqref{eq:one change point} holds since $\widehat{\mathcal{P}}$ is a minimizer of \eqref{eq-wide-p} and that all $I \in \widehat{\mathcal{P}}$ satisfies that $|I| \geq \gamma$, which follows from \Cref{coro:RSS AR lasso} and \eqref{eq-tuning-para-thm1}.	
	
\end{proof}

\begin{lemma}[Case (i)]\label{lemma:one change point} 
With all the conditions and notation in \Cref{proposition:dp step 1}, assume that $I = (s, e]  $ has one and only one true change point $\eta$.  Denote $I_1 = (s, \eta]$, $I_2 = (\eta, e]$ and $\|A_{I_1}^* - A_{I_2}^*\|_2 = \kappa$.  If, in addition, it holds that  
	\begin{align}\label{eq:one change point} 
		 \mathcal L ( I ) \leq   \LL (I_1 )  + \LL (I_2)  +  \gamma, 
	\end{align}
	then with probability at least  $	1 -   (n \vee p)^{-5}$, it holds that with $C_{\epsilon} > 1$, 
	\begin{equation}\label{eq-case-1-lemma-result}
		\min\{|I_1 |,\, |I_2|\} \leq C_{\epsilon} \left(\frac{\lambda^2 d_0 ^2 + \gamma}{\kappa^2} \right).
	\end{equation}
\end{lemma} 

\begin{proof} 
If $|I| < \gamma$, then \eqref{eq-case-1-lemma-result} holds automatically.  If $|I| \geq \gamma$ and $\max\{|I_1|, \, |I_2|\} < \gamma$, then \eqref{eq-case-1-lemma-result} also holds.  In the rest of the proof, we assume that $\max\{|I_1|, \, |I_2|\} \geq \gamma$.  To show \eqref{eq-case-1-lemma-result}, we prove by contradiction, assuming that
	\[%begin{equation}\label{eq-contra-correct}
		\min\{|I_1|, \, |I_2|\} > C_{\epsilon} \left(\frac{\lambda^2 d_0^2  + \gamma}{\kappa^2}\right).
	\]%end{equation}
	Due to \eqref{eq-tuning-para-thm1} and the above, it holds that
	\begin{equation}\label{eq-contra}
		\min\{|I_1|, \, |I_2|\} > CK d_0^2  \log(n \vee p)/\kappa^2.
	\end{equation}
	Then \eqref{eq:one change point} leads to
	$$  \sum_{ t\in I}(X_{t}  -   \widehat{A }^\lambda_IX_{t-1} )^2 \leq   \sum_{t \in I_1} (  X_{t}  -   \widehat{A }^\lambda_{I_1} X_{t-1} )^2 + \sum_{t \in I_2}(X_{t}  -   \widehat{A }^\lambda_{I_2}  X_{t-1} )^2   + \gamma.  $$
		
It follows from \Cref{lemma:RSS AR lasso} and \eqref{eq:one change point} that, with probability at least $1 -    (n \vee p)^{-6}  $ that,  
	\begin{align}
		& \sum_{t \in I_1}(X_{t+1}  -   \widehat{A }^\lambda_IX_t )^2 +
		 \sum_{t \in I_2}(X_{t+1}  -   \widehat{A }^\lambda_I X_t)^2 =
		  \sum_{ t\in I}(X_{t+1}  -   \widehat{A }^\lambda_IX_t )^2 
		   \nonumber 
		   \\
		\leq & \sum_{t \in I_1} (  X_{t+1}  -   \widehat{A }^\lambda_{I_1} X_t )^2 
		+ \sum_{t \in I_2}(X_{t+1}  -   \widehat{A }^\lambda_{I_2}  X_t  )^2 +   \gamma   \nonumber \\
		\leq & \sum_{t \in I_1} (  X_{t+1}  -   A ^* _{I_1} X_t )^2 
		+ \sum_{t \in I_2}(X_{t+1}  -    A ^* _{I_2}  X_t  )^2 + \gamma  + 2C_1 \lambda^2 d_0^2 . \label{eq-lem-1cp-pf-1}
	\end{align}

Denoting $\Delta_i = \widehat{A}^\lambda_{I} - A _{I_i}^*$, $i = 1, 2$, \eqref{eq-lem-1cp-pf-1} leads to that  
	\begin{align}
		& \sum_{t\in I_1} (\Delta_1X_t )^2 + \sum_{t\in I _2 } ( \Delta_1X_t  )^2 
		\leq 
		2\sum_{t\in I_1} \varepsilon_{t+1 } ^\top  \Delta_1X_t + 2\sum_{t\in I_2} \varepsilon_{t+1 } ^\top  \Delta_1X_t     
+ 2 \gamma  		
		 + 2C_1 \lambda^2 d_0 ^2 
		\nonumber 
		\\
		\leq & 2 \left\|\sum_{t\in I_1}  \varepsilon_{t+1}  X_t^\top \right\|_{\infty} \|\Delta_1\|_1 + 2 \left\|\sum_{t\in I_2} \varepsilon_{t+1}  X_t^\top  \right\|_{\infty} \|\Delta_2\|_1 +  \gamma   	  + 2C_1 \lambda^2 d_0 ^2 
		 \nonumber 
		 \\
		\leq & 2 \left\|\sum_{t\in I_1} \varepsilon_{t+1}  X_t^\top  \right\|_{\infty} \bigl(\|\Delta_1(\mathcal S)\|_1 + \|\Delta_1(\mathcal S^c)\|_1\bigr)
		 \nonumber 
		 \\
		& \hspace{1cm} + 2 \left\|\sum_{t\in I_2} \varepsilon_{t+1}  X_t^\top  \right\|_{\infty} \bigl(\|\Delta_2( \mathcal S)\|_1 + \|\Delta_2( \mathcal S^c)\|_1\bigr) +   \gamma 		  + 2C_1 \lambda^2 d_0  ^2 
		\nonumber 
		\\
		\leq & 2 \left\|\sum_{t\in I_1}  \varepsilon_{t+1}  X_t^\top   \right\|_{\infty} \bigl( d_0 \|\Delta_1(\mathcal S)\|_2 + \|\Delta_1(\mathcal S^c)\|_1\bigr)  \nonumber \\
		& \hspace{1cm} + 2 \left\|\sum_{t\in I_2} \varepsilon_{t+1}  X_t^\top  \right\|_{\infty} \bigl( d_0 \|\Delta_2( \mathcal S)\|_2 + \|\Delta_2( \mathcal S^c)\|_1\bigr)  +  \gamma  	  + 2C_1 \lambda^2 d_0 ^2 .\label{eq:one change point first step}
	\end{align}
	
By \Cref{coro:restricted eigenvalue AR}({\bf a}), it holds that with at least probability $1 -(n\vee p)^{-5} $
	\begin{align}
		\eqref{eq:one change point first step} & \leq \frac{\lambda}{2} \big(\sqrt{|I_1| } d_0 \|\Delta_1( \mathcal S)\|_2 + \sqrt{|I_1|}\|\Delta_1(  \mathcal S^c)\|_1 + \sqrt{|I_2| } d_0  \|\Delta_2(  \mathcal S)\|_2 
		+ \sqrt{|I_2|}\|\Delta_2(
		 \mathcal S^c)\|_1\big) \nonumber \\
		 & \hspace{2cm} +   \gamma	 + 2C_3 \lambda^2 d_0 ^2  \nonumber
		 \\
		& \leq    \frac{ 4  \lambda^2 d_0 ^2  }{c_x}   + \frac{c_x  |I_1| \|\Delta_1\|_2^2  }{16}+  
		 \frac{ 4 \lambda^2 d_0 ^2 }{c_x}     + \frac{c_x  |I_2| \|\Delta_2\|_2^2 }{16} +  \lambda \frac { ( \sqrt{|I_1| }  + 
		  \sqrt{|I_2|})}{2} \|\widehat{A }^{\lambda}_I(  \mathcal S^c)\|_1 \nonumber \\
		  & \hspace{2cm} +  \gamma	 + 2C_3 \lambda^2 d_0 ^2  
		 \nonumber 
		 \\
		& \leq \frac{8 \lambda^2 d_0 ^2 }{c_x} + \frac{c_x  |I_1| \|\Delta_1\|_2^2}{16} + \frac{c_x  |I_2| \|\Delta_2\|_2^2}{16}  +  C  \lambda^2 d_0^2  +  \gamma   + 2C_2  \lambda^2 d_0 ^2  , \label{eq-upper-long}
	\end{align}
	where the second inequality follows from  H\ou lder's inequality
	and the last inequality follows from \Cref{lemma: AR oracle lasso} that with probability at least $1 - (n \vee p)^{-6}$ that
	$$\|\widehat A^\lambda_I(\mathcal S^c) \|_1 \le \frac{C\lambda d_0^2   }{\sqrt{|I|}}.$$

Therefore with probability at least $1 -  n^{-5}$
	\begin{align*}
	 \sum_{t \in I_1} (\Delta_1X_t )^2 \ge  & \frac{c_x |I_1|} {2} \|\Delta_1\|_2^2  -C_x \log(p) \|\Delta_1\|_1^2 \\
		\ge & \frac{c_x |I_1|} {2} \|\Delta_1\|_2^2  -  2C_x d_0^2 \log (p) \| \Delta_1 (\mathcal S )\|_2^2 
		-2C_x  \log (p) \| \Delta_1 (\mathcal S^c )\|_1^2 
		\\
		\ge &  \frac{c_x |I_1|} {4} \|\Delta_1\|_2^2 
		 - \frac{2C C_x^2\lambda^2 d_0^{4}   \log(p)  }{ |I|}
		 \\
		 \ge &  \frac{c_x |I_1|} {4} \|\Delta_1\|_2^2 -2C \lambda^2 d_0^2
	\end{align*}
where the first inequality follows from  \Cref{coro:restricted eigenvalue AR}({\bf b}),  the third inequality follows from \eqref{eq-contra}  and  \Cref{lemma: AR oracle lasso}, and the last inequality follows from \eqref{eq-contra}.

Then back to \eqref{eq:one change point first step}, we have 
\begin{align*}
 \frac{c_x |I_1|} {4} \|\Delta_1\|_2^2 -2C \lambda^2 d_0^2  + \frac{c_x |I_2|} {4} \|\Delta_2\|_2^2 -2C \lambda^2 d_0^2 
 \le C  \lambda^2 d_0 ^2 + \frac{c_x  |I_1| \|\Delta_1\|_2^2}{16} + \frac{c_x  |I_2| \|\Delta_2\|_2^2}{16}   +   \gamma	 ,
\end{align*}
which directly gives
\begin{align*}
 |I_1|  \|\Delta_1\|_2^2   +  |I_2|  \|\Delta_2\|_2^2  
 \le C  \lambda^2 d_0 ^2   +2\gamma		  ,
\end{align*}
Since 
$$ |I_1|  \|\Delta_1\|_2^2  + |I_2 |  \|\Delta_2\|_2^2 \ge \inf_{B \in \mathbb R^{p\times p } } |I_1| \|A_{I_1}^* -B\|_2^2 
+    |I_2| \|A_{I_2}^* -B\|_2^2   =\frac{|I_1||I_2| }{|I_1|+|I_2|}\kappa^2 \ge \min\{|I_1|,|I_2|\}\kappa^2/2,
 $$
we have that
 $$  \min\{|I_1|,|I_2|\} \le \frac{C d_0 ^2\lambda^2  +\gamma  }{  \kappa^2 } + \frac{2 \gamma}{ \kappa^2}, $$
which completes the proof.

\end{proof}

\begin{lemma}[Case (ii)] \label{lemma:two change point}
With all the conditions and notation in \Cref{proposition:dp step 1}, assume that $I = [s, e)$ containing exactly two change points $\eta_1$ and $\eta_2$.  Denote $ I_1 =[s, \eta_1)$, $I_2 = [\eta_1, \eta_2)$, $I_3 = [\eta_2, e)$, $\|  A_{I_1} ^*  -A _{I_2} ^* \|_2 = \kappa_1$ and $\|  A_{I_2} ^*  -A _{I_3} ^* \|_2 = \kappa_2$.  If in addition it holds that  
\begin{align}\label{eq:two change point start}
		 \LL (I)  \le \LL(I_1)   + \LL(I_2 )    + \LL(I_3)  +   2\gamma, 
 \end{align}
	then 
	\begin{align}\label{eq-case-2-lemma-result}
		\max\{|I_1|, \, |I_3|\} \leq C_{\epsilon} \left(\frac{\lambda^2 d_0^2 + \gamma}{\kappa^2} \right)
	\end{align} 
	with probability at least $1 -   (n \vee p)^{-5}$.
\end{lemma} 

\begin{proof} 
Since $|I|, |I_2| \geq \Delta$, due to \Cref{assume:AR high dim coefficient}, it holds that $|I|, |I_2| > \gamma$.  In the rest of the proof, without loss of generality, we assume that $|I_1| \leq |I_3|$.  To show \eqref{eq-case-2-lemma-result}, we prove by contradiction, assuming that
	\begin{equation}\label{eq-contra-2}
		|I_1| > C_{\epsilon} \left(\frac{\lambda^2 d_0^2 + \gamma}{\kappa^2} \right).
	\end{equation}
	Due to \eqref{eq-tuning-para-thm1}, we have that $|I_1| >  C  d_0^2 \log(n \vee p)$.  Denote $\Delta_i = \widehat A ^\lambda_I  - A _{I_i}^*$, $i = 1, 2, 3$.  We then consider the following two cases.

\vskip 3mm
\noindent {\bf Case 1.}  Suppose
	\[
		|I_3| > \gamma .
	\]
	Then 
	\eqref{eq:two change point start} implies that  
	$$
		  \sum_{t\in I }  \|  X_{t+1} -\widehat A^\lambda_ I X_t \|_2^2
		\leq   \sum_{t\in I_1 }  \| X_{t+1} -\widehat A^\lambda_{ I_1}  X_t  \|_2^2 
		+
		 \sum_{t\in I_2 }  \| X_{t+1} -\widehat A^\lambda_ {I_2}  X_t \|_2 ^2 
		  +
		   \sum_{t\in I_3} \|  X_{t+1} -\widehat A^\lambda_{ I_3}  X_t    \|_2^2 + 2  \gamma   
	$$  
The above display and  \Cref{lemma:RSS AR lasso}   give 
$$
		  \sum_{t\in I }  \|  X_{t+1} -\widehat A^\lambda_ I X_t  \| _2 ^2
		\leq   \sum_{t\in I_1 }  \| X_{t+1} -  A ^* _{I_1}  X_t \|_2^2 
		+
		 \sum_{t\in I_2 }  \| X_{t+1} - A ^* _{I_ 2}  X_t  \|_2^2 
		  +
		   \sum_{t\in I_3} \| X_{t+1} - A ^* _{I_3 } X_t \| _2^2 +2  \gamma   
     + 3C_1 \lambda^2 d_0^2 
	$$   
	which in term implies that
	\begin{align}
	& \sum_{i=1}^3 \sum_{t\in I_i} ( \Delta_iX_t )^2 \leq 2 \sum_{i=1}^3 \sum_{t\in I_i} \varepsilon_{t+1}^\top   \Delta_i X_t   + 3C \lambda^2 d_0 ^2  +  2 \gamma \nonumber \\
	\leq & 2\sum_{i=1}^3 \left \| \frac{1}{\sqrt {|I_i| }} \sum_{t\in I_i } \varepsilon_{t+1}  X_t^ \top   \right \|_{\infty } \| \sqrt { | I_i|}\Delta_i \|_1 + 3C \lambda^2 d_0 ^2  +   2 \gamma  \nonumber \\
   \le &  \frac{\lambda}{2} \sum_{i=1}^3 \left( d_0  \sqrt { |I_i| }\|\Delta_i(\mathcal S)\|_2  + \sqrt {|I_i| } \|\Delta_i(\mathcal S^c) \|_1 \right) + 3C \lambda^2 d_0 ^2  +  2\gamma,  \nonumber
	\end{align}
	where the last inequality follows from   \Cref{coro:restricted eigenvalue AR}({\bf a}).
It follows from identical arguments in \Cref{lemma:one change point} that, with probability at least $1 - 6 (n \vee p)^{-5}   $,
	\[
		\min\{|I_1|,\, |I_2|\} \le C_{\epsilon} \left(\frac{\lambda^2 d_0 ^2  + \gamma}{\kappa^2}\right).
	\]
	Since $|I_2| \ge \Delta$ by   \Cref{assume:AR high dim coefficient},	 it holds that
	\[ 	|I_1|  \le C_{\epsilon} \left(\frac{\lambda^2 d_0 ^2+ \gamma}{\kappa^2}\right),
	\]
	which contradicts \eqref{eq-contra-2}.
 
\vskip 3mm
\noindent {\bf Case 2.}  Suppose that 
	\[
		|I_3|     \le   \gamma. 
	\]   
Then 
	\Cref{eq:two change point start} implies that  
	$$
		  \sum_{t\in I } ( X_{t+1} -\widehat A^\lambda_ I X_t )^2
		\leq   \sum_{t\in I_1 } (X_{t+1} -\widehat A^\lambda_{ I_1}  X_t  )^2 
		+
		 \sum_{t\in I_2 } (X_{t+1} -\widehat A^\lambda_ {I_2}  X_t )^2 
		  +
		2  \gamma  
	$$  	
	The above display and  \Cref{lemma:RSS AR lasso}   give 
$$  \sum_{t\in I_1 \cup  I_2 } ( X_{t+1} -\widehat A^\lambda_ I X_t )^2 \le 
		  \sum_{t\in I } ( X_{t+1} -\widehat A^\lambda_ I X_t )^2
		\leq   \sum_{t\in I_1 } (X_{t+1} -  A ^* _{I_1}  X_t  )^2 
		+
		 \sum_{t\in I_2 } (X_{t+1} - A ^* _{I_ 2}  X_t )^2 		+   
		 2\gamma     + 2C \lambda^2 d_0^2 
	$$   
	
	It follows from identical arguments in \Cref{lemma:one change point} that, with probability at least $1 -   (n \vee p)^{-5}   $,
	\[
		\min\{|I_1|,\, |I_2|\} \le C_{\epsilon} \left(\frac{\lambda^2 d_0 ^2  + \gamma}{\kappa^2}\right).
	\]
	Since $|I_2| \ge \Delta$ by   \Cref{assume:AR high dim coefficient},	 it holds that
	\[ 	|I_1|  \le C_{\epsilon} \left(\frac{\lambda^2 d_0 ^2+ \gamma}{\kappa^2}\right),
	\]
	which contradicts \eqref{eq-contra-2}.
\end{proof}

\begin{lemma}[Case (iii)] \label{lemma:no change point}
With all the conditions and notation in \Cref{proposition:dp step 1}, assume that there exists no true change point in $I = (s, e]$.  With probability at least $1-(n\vee p)^{-4}$, it holds that 
$$ \LL (I)  <  \min_{t\in ((s,e)}  \left\{\LL(s,t]) +\LL((t,e]) \right\} + \gamma .$$ 
\end{lemma}

\begin{proof} 
For any fixed $t \in (s, e)$, let $I_1 = (s,t]$ and $I_2 = (t,e]$.  By \Cref{coro:RSS AR lasso}, it holds that with probability at least $1-  (n\vee p)^{-5}$,
	\begin{align*}
		\max_{J \in \{I_1, I_2, I\}} \left |  \LL(J)   - \LL^ * (J)   \right|  
		\leq   C d_0\lambda^2 <\gamma/3,
	\end{align*}
	where $\mathcal{L}^*(J)$ is the population counterpart of $\mathcal{L}(J)$, replacing the coefficient matrix estimator with its population counterpart.  
	Since $ A^*_ {I} =A ^*_ {I_1} =A ^*_ {I_2}$,  we have that with probability at least $1-  (n\vee p)^{-5}$,
	$$ \LL (I)  <     \LL(I_1) +\LL(I_2 )  + \gamma .$$
	Then, using a union bound argument, with probability at least $1-(n\vee p)^{-4}$, it holds that 
$$ \LL (I)  <  \min_{t\in (s,e)}  \left\{\LL((s,t]) +\LL((t,e]) \right\} + \gamma .$$
\end{proof}

\begin{lemma}[Case (iv)] \label{lemma:three change point}
With all the conditions and notation in \Cref{proposition:dp step 1}, assume that $I = [s, e)$ contains  $J$ true change points $\{ \eta_{j}\}_{j=1}^J$, where $|J| \ge 3$.  Then with probability at least $1 -  C(n \vee p)^{-4}$, with an absolute constant $C > 0$,
	\[
		\LL(I)  > \sum_{j=1}^{J+1}  \LL(I_j ) + J \gamma, 
	\]
	where $I_1 =(s,\eta_{1}]$, $I_j = [\eta_{j-1},\eta_{j})$ for any $j \in\{2, \ldots, J\}$ and $I_{J+1} = [\eta_J, e)$.
\end{lemma} 

\begin{proof}
Since $J \geq 3$, we have that $|I| > \gamma$ and $\LL(I) =  	\sum_{t\in I }  \| X_{t+1} - \widehat A ^\lambda_ I X_t  \|_2^2$.  We prove by contradiction, assuming that
	\begin{align}\label{eq:case 4 contradiction start}
			\sum_{t\in I }  \| X_{t+1} - \widehat A ^\lambda_ I X_t  \|_2^2 \leq \sum_{j=1}^{J+1}    \LL(I_j ) + J \gamma. 
	\end{align}
	Note that $ |I_j| \ge \Delta   \ge \gamma  $ for all $ 2 \le j \le J$.		Let $\Delta_i = \widehat A^\lambda_I  - A _{I_i}^*$, $i = 1, \ldots, J+1$.  From \Cref{coro:RSS AR lasso}, we have that with probability at least $1 - (n \vee p)^{-4}$,
	\begin{align}\label{eq:case 4 contradiction start 1}
			\sum_{t\in I }  \| X_{t+1} - \widehat A ^\lambda_ I X_t  \|_2^2 
			\leq  \sum_{j=1}^{J+1}    \LL^* (I_j ) +JC\lambda^2 d_0^2    + J \gamma,
	\end{align}
	where $\mathcal{L}^*(\cdot)$ denotes the population counterpart of $\mathcal{L}(\cdot)$, by replacing the coefficient matrix estimator with its population  counterpart.  In the rest of the proof, without loss of generality, we assume that $|I_1| \leq |I_{J+1}|$.  

There are three cases: (1) $|I_1| \geq \gamma$, (2) $|I_1| < \gamma \leq |I_{J+1}|$ and (3) $|I_{J+1}| < \gamma$.  All these three cases can be shown using very similar arguments, and cases (2), (3) are both simpler than case (1), so in the sequel, we will only show case (1).

Suppose that $\min\{|I_1|,\, |I_{J+1} |\} \ge \gamma $.  Then \eqref{eq:case 4 contradiction start 1}  gives 
\begin{align}\label{eq:case 4 contradiction start 2}
\sum_{t\in I }  \| X_{t+1} - \widehat A ^\lambda_ I X_t  \|_2^2 \leq \sum_{j=1}^{J+1}   \sum_{t\in I_{j }  } \| X_{t+1} -
  A ^*_ {I_j}  X_t  \|_2^2 + JC\lambda^2 d_0^2  + J \gamma . 
\end{align}
	which implies that 
	\begin{align} 
		\sum_{j=1}^{J+1} \sum_{t\in I_j}  \|  \Delta_j X_t\| _2 ^2 \leq 2
		 \sum_{j=1}^{J+1}  \sum_{t\in I_j} \varepsilon_{t+1}  ^\top \Delta_j  X_t  + JC\lambda^2 d_0^2  + J \gamma . \label{eq:three change point first step}
	\end{align}
By \Cref{coro:restricted eigenvalue AR}({\bf a}), and $|I_j| \ge C_\gamma d_0^2 \lambda^2 $ for all $1\le j\le J+1$,
   it holds that with probability at least $1 - (n \vee p)^{-4}$,
	\begin{align}  
		& \sum_{t\in I_j}  \varepsilon_{t+1}  ^\top \Delta_j  X_t  
		\leq 
		\left\|\frac{1}{\sqrt{|I_j|}}  \sum_{t\in I_j } \varepsilon_{t+1}  X_t ^\top  \right\|_{\infty} \| \sqrt { | I_j| } \Delta_j\|_1 \nonumber \\
		 \leq  & \lambda/4 \left(d_0 \sqrt{ |I_j|}\|\Delta_j(\mathcal S)\|_2 + \sqrt{|I_j|} \|\Delta_j( \mathcal S^c)\|_1\right) \nonumber  \\
		\leq & \frac{4\lambda^2 d_0 ^2 }{c_x^2} + \frac{c_x^2|I_j|}{16} \|\Delta_j\|_2^2  + \lambda/4 \sqrt {|I_j|}
		 \|\widehat A^\lambda_{I}(S^c)\|_1 \nonumber \\
		\le  & \frac{4\lambda^2 d_0}{c_x^2} + \frac{c_x^2|I_j|}{16} \|\Delta_j\|_2^2  + \lambda/4 \sqrt {|I_j|}   
     \frac{Cd_0^2 \lambda }{ \sqrt {|I| }}		
		 \nonumber \\
		\leq & C \lambda ^2 d_0^2   + \frac{c_x^2|I_j|}{256} \|\Delta_j\|_2^2    \label{eq:three change points standard equality 1}
	\end{align}
	where \Cref{lemma: AR oracle lasso} and H\ou lder inequality are used in the third inequality. 
	  In addition,  by \Cref{coro:restricted eigenvalue AR}({\bf b}), it holds that with probability at least $1 - (n \vee p)^{-4}$,
	\begin{align}
		&  \sum_{t\in I_j} \| \Delta_j X_t \|_2^2  \geq \frac{c_x  |I_j|  }{ 2} \|\Delta_j\|_2 ^2  -  C_x  \log(p)  \|\Delta_j\|_1 ^2 \nonumber \\
		\geq &  \frac{c_x  |I_j|  }{ 2} \|\Delta_j\|_2 ^2   - 2 C_x  \log(p) d_0^2  \|\Delta_j (\mathcal S) \|_2 ^2 - 2C_x  \log(p)  \|\Delta_j( \mathcal S^c)\|_1 ^2 \nonumber \\
		\geq &  \frac{c_x  |I_j|  }{ 4 } \|\Delta_j\|_2 ^2 - 2C_x  \log(p)  \| \widehat A^\lambda _I  ( \mathcal S^c)\|_1 	^2 	 \nonumber 
		\\
		\geq &  
		  \frac{c_x  |I_j|  }{ 4 } \|\Delta_j\|_2 ^2 - \frac{2 C_x  \log(p)  \lambda^ 2d_0^4      }{ |I| }, \nonumber  
		  \\
		  \ge & \frac{c_x  |I_j|  }{ 4 } \|\Delta_j\|_2 ^2 - 2 \lambda^2 d_0^2 ,
		   \label{eq:REC three change}
	\end{align}
	where the third inequality follows from  \Cref{lemma: AR oracle lasso}   and the last follows from $|I| \ge \gamma$.

Since for any $j \in \{2, \ldots, J-1\}$, it holds that
	\begin{align*} 
		|I_j|\|\Delta_j\|_2^2 + |I_{j+1}|\|\Delta_{j+1}\|_2^2 & \geq   \inf_{B  \in \mathbb{R}^p}\bigl\{|I_j|\|A _{I_j}^* - B \|_2^2 + |I_{j+1}| \|A_{I_{j+1}}^* -B \|_2^2\bigr\} \\
		& \geq \frac{|I_j||I_{j+1}|}{|I_j|+|I_{j+1}|} \kappa^2 \geq \min\{|I_j|,\, |I_{j+1}|\}\kappa^2/2.
	\end{align*}
	 Based on the  the same arguments in \Cref{lemma:one change point}, \eqref{eq:three change point first step}, \eqref{eq:three change points standard equality 1} and \eqref{eq:REC three change} together imply that with probability at least $1 - C(n \vee p)^{-4}$,
	\[
		\min_{j = 2, \ldots, J-1} |I_j| \leq C_{\epsilon} \left(\frac{\lambda^2 d_0 ^2 + \gamma}{\kappa^2}\right),
	\]
	which is a contradiction to \eqref{eq:case 4 contradiction start}.
\end{proof}
	 
\subsection	{Proof of  \Cref{prop-2}}
\begin{lemma} \label{lem-case-5-prop-2-needed}
Under the assumptions and notation in \Cref{prop-2}, suppose that there exists no true change point in the interval $I$.  For any interval $J \supset I$, it holds that with probability at least $1 -(n \vee p)^{-4}$,
\begin{align*}
		\LL^*(I) - \sum_{t \in I} (X_{t+1} - \widehat A^\lambda _ J  X_{t} )^2 \leq C\lambda^2 d_0^2,
	\end{align*}
	where $C > 0$ is an absolute constant and $\mathcal{L}^*(I)$ is the population counterpart of $\mathcal{L}(I)$ by replacing the coefficient matrix estimator with its population counterpart.
\end{lemma}
 
\begin{proof}  Let 
$\Delta_I =A ^*_I  -\widehat A ^{\lambda}_J$
\\
\\
\textbf{Case 1.} If $| I| \le \gamma$, then $\mathcal{L}^*(I) = 0$ and the claim holds automatically.
\\
\\
\textbf{Case 2.} If
 $$| I| \ge \gamma \ge C_\gamma d_0^2  \log(p\vee n)  , $$
 then $|J|\ge \gamma$ and by \Cref{coro:restricted eigenvalue AR}({\bf b}),
   we have with probability at least $1 - (n \vee p)^{-4}$,
	\begin{align}
		&  \sum_{t \in I} \|  \Delta_IX_t  \| _2^2  \geq \frac{c_x |I| }{2} \|\Delta_I\|_2 ^2 -  C_x \log(p)  \|\Delta_I\|_1^2 \nonumber \\
		= & \frac{c_x |I| }{2} \|\Delta_I\|_2  ^2 - 2  C_x \log(p)  \|\Delta_I( \mathcal S)\|_1 ^2 -  2C_x \log(p)  \|\Delta_I(\mathcal S^c)\|_1 ^2 
		 \nonumber
		  \\
		\geq & \frac{c_x |I| }{ 2} \|\Delta_I\|_2^2 - 2 C_x d_0 ^2 \log(p)  \|\Delta_I\|_2 ^2  - 2 C_x \log(p)  \|\Delta_I( \mathcal S^c)\|_1 ^2\nonumber \\
		\geq & \frac{c_x |I|}{4} \|\Delta_I\|_2 ^2- 2 C_x \log(p)  \|\widehat{A }^{\lambda}_J(\mathcal S ^c)\|_1 ^2  \geq 
		\frac{c_x|I| }{4} \|\Delta_I\|_2 ^2 - 18 C_x d_0 ^2 \lambda ^2  , \label{eq-lem16-pf-1}
	\end{align} 
	where the last inequality follows from \Cref{lemma: AR oracle lasso} and $|J|\ge \gamma \ge C_\gamma d_0^2  \log(p\vee n) $.  We then have on the event in \Cref{coro:restricted eigenvalue AR}({\bf a}),
	\begin{align*}
	& \sum_{t \in I} (X_{t+1} - A^*_I  X_t   )^2 - \sum_{t \in I} (X_{t+1} - \widehat{A }^{\lambda}_JX_t  )^2  = 2 \sum_{t \in I} \varepsilon_{t+1}  \Delta_I  X_t  - \sum_{t \in I}  \| \Delta_I X_t  \| _2 ^2
	 \\
	\leq &
	 2\left \|\sum_{t \in I }  \varepsilon_{t+1}  X_t^ \top \right\|_{\infty }  \left(  d_0   \| \Delta_I(\mathcal S) \|_2  + \|\widehat A^{\lambda}_J (\mathcal S^c) \|_1  \right)  
 - \frac{c_x  |I|}{ 4}\|\Delta_I\|_2^2  +  18  C_x \lambda^2 d_0^2 
	\\  
	\leq & \frac{\lambda \sqrt {|I|}  }{2}  \left(   d_0  \|\Delta_I\|_2  +   \frac{C\lambda d_0^{2}   }{\sqrt{|J|}}  \right)   - \frac{c_x  |I|}{ 4}\|\Delta_I\|_2^2  +  18  C_x \lambda^2 d_0^2  
	 \\
	\leq & 
	\frac{\lambda \sqrt {|I| } d_0 }{2} \|\Delta_I\|_2 +C ' \lambda^2  d_0^{2}  -  \frac{c_x  |I|}{ 4}\|\Delta_I\|_2^2  +  18  C_x \lambda^2 d_0^2   \\
	\leq & \frac{c_x  |I|}{ 4}\|\Delta_I\|_2^2   +  C'' \lambda^2 d_0^2  +C ' \lambda^2  d_0^{ 2}  -  \frac{c_x  |I|}{ 4}\|\Delta_I\|_2^2  +  18  C_x \lambda^2 d_0^2   \\
	\leq & C_6\lambda^2 d_0^2.
	\end{align*}
	where the first inequality follows from \eqref{eq-lem16-pf-1}, the second inequality follows from \Cref{coro:restricted eigenvalue AR}({\bf a}) and \Cref{lemma: AR oracle lasso}, the third follows from  H\ou lder inequality and 
	$| J| \ge \gamma \ge C_\gamma d_0^2  \log(p\vee n)$. 
 
\end{proof}

\begin{proof}[Proof of \Cref{prop-2}] 
Let $\mathcal{L}^*(\cdot)$ be the population counterpart of $\mathcal{L}(\cdot)$ by replacing the coefficient matrix estimator with its population counterpart.  This proof is based on the events
\[
\mathcal G_1^{I, J} = \left\{ 	\LL^* (I)   - \sum_{t \in I} \|X_{t+1} - \widehat A^\lambda _ J  X_{t} \|^2_2 \leq C\lambda^2 d_0^2 \right\}
\]
and
\[
\mathcal G_2^{I} = \left\{  \left| 	\LL^*(I) - \LL (I)  \right|   \leq C d_0 ^2\lambda^2    \right\}.
\]
In addition denote
\begin{align*}
\mathcal G_1 = \bigcup_{I, J \subset[1, \ldots,n], \  I \subset \mathcal I  }\mathcal G_1^{I, J} \quad \mbox{and} \quad \mathcal G_2 = \bigcup_{I\subset[1, \ldots,n], \  I \subset \mathcal I }\mathcal G_2^{I},
\end{align*}
where
$$ \mathcal I = \{ I \subset \{1,\ldots,n\}:\,  |I| \ge \gamma, \text{ and there exists $k$ such that 
}I \subset [\eta_{k-1}, \eta_{k})  \}.$$
Note that by \Cref{lem-case-5-prop-2-needed}  and  \Cref{coro:RSS AR lasso} and union bounds
$$ \mathbb P(\mathcal G_1)  \ge 1-(n\vee p)^{-1} \quad \text{and} \quad 
\mathbb P(\mathcal G_2)  \ge 1-(n\vee p)^{-3}.
$$
  Let $\{A_t^*\}_{t=1}^T  \subset \mathbb R^{p\times p }$ be such that 
$ A_t ^* = \mathcal A_k ^*  $ for any $\eta_{k-1} <t \le \eta_{k}
 $. 
Denote $$S^*_n = \sum_{k=0}^K    \LL ^* ( (\eta_k ,\eta_{k+1}  ] )  . $$ 
  Given any collection $\{t_1, \ldots, t_m\}$, where $t_1 < \cdots < t_m$, and $t_0 = 0$, $t_{m+1} = n$, let 
	\begin{equation}\label{eq-sn-def}
		S_n(t_1, \ldots, t_{m}) = \sum_{k=1}^{m} \LL ( (t_{k}, t_{k+1}])   . 
	\end{equation}
	For any collection of time points, when defining \eqref{eq-sn-def}, the time points are sorted in an increasing order.

	In addition, since
	$$ \widehat S_n = \sum_{k=0}^{\widehat K } \LL ( (\widehat \eta_k , \widehat \eta _{k+1}])  ,$$
  therefore $ \widehat S_n  + (\widehat K + 1)\gamma$ is the minimal value of the objective function in  \eqref{eq-wide-p}.
	\\
	\\
 {\bf Step 1.} Let $\{ \widehat \eta_{k}\}_{k=1}^{\widehat K}$ denote the change points induced by $\widehat {\mathcal P}$.  If one can justify that 
	\begin{align} \nonumber 
	  & 	S^*_n  + K\gamma 
		\\
		 \ge  &S_n(\eta_1,\ldots,\eta_K)    + K\gamma - C(K+1) d_0 \lambda^2 \label{eq:K consistency step 1} \\ 
		\ge &  \widehat S_n +\widehat K \gamma - C(K+1) d_0 \lambda^2 \label{eq:K consistency step 2} \\ 
		\ge &   S_n ( \widehat \eta_{1},\ldots, \widehat \eta_{\widehat K } , \eta_1,\ldots,\eta_K )   + \widehat K \gamma  -  C(K+\widehat K +1)d_0^2 \lambda^2   \label{eq:K consistency step 3}
	\end{align}
	and that 
	\begin{align}\label{eq:K consistency step 4}
		S^*_n     -S_n ( \widehat \eta_{1},\ldots, \widehat \eta_{\widehat K } , \eta_1,\ldots,\eta_K )   \le C (K + \widehat{K} + 1) \lambda^2 d_0^2    ,
	\end{align}
	then it must hold that $| \hatp | = K + 1$, as otherwise if $\widehat K \ge K+1 $, then  
	\begin{align*}
		C (K + \widehat{K} + 1) \lambda^2 d_0 ^2   & \geq S^*_n  -S_n ( \widehat \eta_{1},\ldots, \widehat \eta_{\widehat K } , \eta_1,\ldots,\eta_K ) \\
		& \geq - C (K +\widehat K +1) \lambda^2 d_0   + (\widehat K - K)\gamma 
	\end{align*} 
	Therefore due to the assumption that $| \hatp| - 1 =\widehat K \le 3K$, it holds that 
	\begin{align} \label{eq:Khat=K} 
		C(8K + 2)\lambda^2d_0 ^2      \geq (\widehat K - K)\gamma \geq \gamma ,
	\end{align}
	Note that \eqref{eq:Khat=K} contradicts with the choice of $\gamma$.
\\
\\
 {\bf Step 2.} 
Observe  that \eqref{eq:K consistency step 1}  is implied by 
	\begin{align}\label{eq:step 1 K consistency}  
		\left| 	S^*_n   -   S_n(\eta_1,\ldots,\eta_K)    \right| \le  C(K+1) d_0 ^2 \lambda^2,
	\end{align}
	which is  immediate consequence  of $\mathcal G_2$.  Since $\{ \widehat \eta_{k}\}_{k=1}^{\widehat K}$ are the change points induced by $\widehat {\mathcal P}$, \eqref{eq:K consistency step 2} holds because $\hatp$ is a minimizer.
\\
\\
 {\bf Step 3.} For every $\widehat I =(s,e]\in \hatp$,  let $\{ \eta_{p+l}\}_{l=1}^{q+1}  =\widehat I \ \cap \ \{\eta_k\}_{k=1}^K$
	\[
		  (s ,\eta_{p+1}]  = J_1,   (\eta_{p+1},   \eta_{p+2} ]=J_2, \  \ldots,  \   (\eta_{p+q} ,e]  =  J_{q+1}.
	\]
	Then \eqref{eq:K consistency step 3} is an immediate consequence of the following inequality
	\begin{align} \label{eq:one change point step 3}
	 \LL (\widehat I )   \ge 
		 \sum_{l=1}^{q+1}      \LL  (J_l ) - C(q+1) \lambda^2 d_0 ^2    .
	\end{align}
	{\bf Case 1.} If $| \widehat I| \le \gamma$, then 
	\begin{align*} 
	 \LL (\widehat I )   =  0  = 
		 \sum_{l=1}^{q+1}      \LL  (J_l ) ,
	\end{align*}
	where $|J_l| \le |\widehat I| \le \gamma$ is used in the last inequality.
	\\
	\\
	{\bf case 2.} If $| \widehat I| \ge \gamma$, then it suffices to show that 
	\begin{align} \label{eq:one change point step 4}
		 \sum_{t\in \widehat I }
	 \| X_{t+1} -  \widehat A^\lambda_{\widehat I}X_t     \|_2^2    \ge 
		 \sum_{l=1}^{q+1}      \LL  (J_l ) - C(q+1) \lambda^2 d_0 ^2  .
	\end{align}
	On $\mathcal G_2$, it holds that
	\begin{align} 
			 \sum_{l=1}^{q+1}      \LL(J_l ) 
			\le 
		  \sum_{l=1}^{q+1}    \LL^*(J_l )    + (q+1)  C d_0 ^2 \lambda^2  
	   \label{eq:K consistency step 3 inequality}
	\end{align}
	In addition for each $l \in \{1, \ldots, q+1\}$,
	\[ 
	 \sum_{t\in J_l }   \| X_{t+1} -  \widehat A^\lambda_{\widehat I}X_t    \|_2^2   
	 \ge 
	 \sum_{t\in J_l }  \LL^*  (J_l)   - C\lambda^2 d_0^2    ,
	\]
	where  the  inequality follows from  $\mathcal G_1$.  Therefore the above inequality implies that 
	\begin{align}    
		\sum_{t\in \widehat I }\| X_{t+1} -  \widehat A^\lambda_{\widehat I}X_t     \|_2^2  \ge \sum_{l=1}^{q+1} \sum_{t \in J_l}  \| X_{t+1} - \widehat A^\lambda_{\widehat I}X_t     \|_2 ^2 \ge   \sum_{l=1}^{q+1}  \LL^*  (J_l)   
     -  C(q+1)\lambda^2 d_0^2  .
     \label{eq:K consistency step 3 inequality  3}  
	\end{align}
	Note that \eqref{eq:K consistency step 3 inequality} and  \eqref{eq:K consistency step 3 inequality 3} imply \eqref{eq:one change point step 4}.

Finally, to show \eqref{eq:K consistency step 4}, observe that from \eqref{eq:step 1 K consistency}, it suffices to show that 
	\[
		S_n(\eta_1,\ldots,\eta_K)  -  S_n ( \widehat \eta_{1},\ldots, \widehat \eta_{\widehat K } , \eta_1,\ldots,\eta_K )  \le  (K+\widehat K +1 ) \lambda^2d_0^2 ,
	\]
	the analysis of which follows from a similar but simpler argument as above.
\end{proof}

%% file: appendixD.tex
%!TEX root = ./vardp.tex

In this section, we provide technical lemmas.  

\begin{lemma}  \label{prop:deviation AR change point}
For \Cref{assume:AR change point}, under \Cref{assume:AR high dim coefficient}, the following holds.
	\begin{itemize}
		\item [(\textbf{a})]	  There exists an absolute constant $c > 0$, such that for any $u, v \in \{w \in \mathbb{R}^p:\, \|w\|_0 \leq s, \, \|w\|_2 \leq 1\}$ and any $\xi > 0$, it holds that
			\begin{align*}%begin{align}\label{eq:subgaussian s-sparse}
				\mathbb{P} \left\{\left|v^{\top}\sum_{t \in I}\left(X_t X_t^{\top} - \mathbb{E}\bigl\{X_t X_t^{\top}\bigr\}\right)v\right| \geq 2\pi \mathcal{M}\xi\right\} \leq 2 \exp\{-c \min\{\xi^2/|I|, \, \xi\}\}
			\end{align*}
			and
			\begin{align*} %begin{align} \label{eq:subgaussian 2s-sparse}
				\mathbb{P} \left\{\left|v^{\top}\sum_{t \in I}\left(X_t X_t^{\top} - \mathbb{E}\bigl\{X_t X_t^{\top}\bigr\}\right)u\right| \geq 6\pi \mathcal{M}\xi\right\} \leq 6 \exp\{-c \min\{\xi^2/|I|, \, \xi\}\};
			\end{align*}
			in particular, for any $i, j \in \{1, \ldots, p\}$, it holds that
			\begin{align}\label{eq:subgaussian s-sparse-1}
				\mathbb{P} \left\{\left|\left(\sum_{t \in I}\left(X_t X_t^{\top} - \mathbb{E}\bigl\{X_t X_t^{\top}\bigr\}\right)\right)_{ij}\right| \geq 6\pi \mathcal{M}\xi\right\} \leq 6 \exp\{-c \min\{\xi^2/|I|, \, \xi\}\}.
			\end{align}
		
		\item [(\textbf{b})]	  Let $\{Y_t\}_{t = 1}^n$ be a $p$-dimensional, centered, stationary process.  Assume that for any $t \in \{1, \ldots, n\}$,  $\mathrm{Cov}(X_t, Y_t) = 0$.  The joint process $\{(X_t^{\top}, Y_t^{\top})^{\top}\}_{t = 1}^n$ satisfies \Cref{assume:AR high dim coefficient}(\textbf{b}).  Let $f_Y$ be the spectral density function of $\{Y_t\}_{t = 1}^n$, and $f_{k, Y}$ be the cross spectral density function of $\mathcal{X}_k$ and $\{Y_t\}_{t = 1}^n$, $k \in \{0, \ldots, K\}$.  There exists an absolute constant $c > 0$, such that for any $u, v \in \{w \in \mathbb{R}^p:\, \|w\|_2 \leq 1\}$ and any $\xi > 0$, it holds that
			\begin{align*}
				& \mathbb{P} \Bigg\{\left|v^{\top}\sum_{t \in I}\left(X_t Y_t^{\top} - \mathbb{E}\bigl\{X_t Y_t^{\top}\bigr\}\right)u\right|  \geq 2\pi \max_{k = 0, \ldots, K}\bigl(\mathcal M (f_k) + \mathcal M (f_Y) + \mathcal M (f_{k,Y} ) \bigr) \xi \Bigg\} \\
				\leq & 6 \exp\{-c \min\{\xi^2/|I|, \, \xi\}\}.
			\end{align*}
	\end{itemize}
\end{lemma}

Although there exists one difference between \Cref{prop:deviation AR change point} and Proposition 2.4 in \cite{basu2015regularized} that we have $K+1$ different spectral density distributions, while in \cite{basu2015regularized}, $K = 0$, the proof can be conducted in  a similar  way, only noticing that the largest eigenvalue should be taken as the largest over all $K+1$ different spectral density functions due to independence.

\begin{proof} 
The proof is an extension of Proposition 2.4 in \cite{basu2015regularized} to the change point setting.  As for (\textbf{a}), let $w_t = X_t^{\top}v$ and $W = (w_1, \ldots, w_n)^{\top} \sim \mathcal{N}_n(0, Q)$, where $Q\in \mathbb R^{n\times n} $ has the following block structure due to the independence,
	$$
		Q = \begin{pmatrix}
 			Q_1 & 0  & \cdots & 0\\
 			0 & Q_2 & \cdots & 0 \\
 			\vdots & \vdots  &\ddots & \vdots\\
			0 & 0  & \cdots & Q_{K+1}\\
		\end{pmatrix},  
	$$ 
	where $Q_k \in \mathbb R^{(\eta_k-\eta_{k-1}) \times (\eta_k-\eta_{k-1})}$, $k = 1, \ldots, K+1$.  By the proof of proposition 2.4 in \cite{basu2015regularized}, $\| Q_k \|_{\op}\le \mathcal M$.  Note that 
	$$ 
		v^{\top} \widehat \Sigma_n (X, X) v = n^{-1} W^{\top}W = n^{-1} Z^{\top}QZ,
	$$
	where $Z \sim \mathcal{N}_n(0, I)$.  The rest follows the proof of Proposition 2.4 in \cite{basu2015regularized}.
	
As for (\textbf{b}), denote $w_t = X_t^{\top}u$ and $z_t = Y_t^{\top}v$.  Observe that for $t \in [\eta_k, \eta_{k+1})$, $k = 0, \ldots, K$, the process $\{w_t + z_t\}$ has the spectral density function as follows, for $\theta \in (-\pi, \pi]$,
	$$ 
		f_{w_k+z_k}(\theta) = (u, v)^{\top} \begin{pmatrix}
			f_k(\theta) & f_{k, Y}(\theta) \\
			f_{k, Y}^* (\theta) & f_Y(\theta)   
		\end{pmatrix}  \begin{pmatrix}
			u\\
			v
		\end{pmatrix} = u^{\top}f_k(\theta)u + v^{\top}f_Y(\theta) v + u^{\top}f_{k, Y} (\theta) v + v^{\top}f^*_{k, Y}(\theta) u. 
	$$
	Therefore $\mathcal M ( f_{w_k+z_k}) \le  \mathcal M (f_k) + \mathcal M (f_Y) + \mathcal M (f_{k,Y})$, and the rest follows from (\textbf{a}). 
\end{proof}

\begin{lemma}\label{coro:restricted eigenvalue AR} 
For $\{X_t\}_{t=1}^n$ satisfying \Cref{assume:AR change point} and  \Cref{assume:AR high dim coefficient}, we have the following.
	\begin{itemize}
		\item [(\textbf{a})]	 For any interval $I \subset \{1, \ldots, n\}$ and any $l \in \{1, \ldots, L\}$, it holds that 
			\begin{align*}
				\mathbb{P}\left\{\left\|\sum_{t \in I} \varepsilon_{t + 1} X_{t+1-l}^{\top}\right\|_{\infty} \leq C\max\{\sqrt{|I|\log(n\vee p)}, \, \log(n\vee p)\}\right\} > 1 -  (n\vee p)^{-6},
			\end{align*}
			where  $C > 0$ is a constant depending on $\mathcal M (f_k)$, $\mathcal M (f_Y)$, $\mathcal M (f_{k, \varepsilon})$, $k = 0, \ldots, K$.

		\item [(\textbf{b})] Let $L=1$ and $C$ be some constant depending on $\mathfrak{m}$ and $\mathcal M$.  For any interval $I \subset \{1, \ldots, n\}$ satisfying
			\begin{equation}\label{eq-I-size}
				|I| > C \log(p) ,
			\end{equation}
			with probability at least $1 -  n^{-6}$, it holds that for any $B \in \mathbb{R}^{p \times p}$,
			\begin{align}\label{eq:good event restricted eigenvalue}
				\sum_{t \in I} \|BX_t\|_2^2 \geq \frac{|I| c_x }{2} \|B\|_2^2 - C_x\log(p) \|B\|_1^2,
			\end{align}
			where $C_x, c_x > 0$ are absolute positive constants depending on all the other constants.
			
		\item [(\textbf{c})] Let  $L\in \mathbb Z^+$. For any interval $I \subset \{1, \ldots, n\}$ satisfying \begin{equation}\label{eq-I-size L>1}
				|I| > C \max\{\log(p) , KL\},
			\end{equation}
					with probability at least $1 -  n^{-6}$, it holds that for any $\{B[l], \, l = 1, \ldots, L\} \subset \mathbb{R}^{p \times p}$,
			\begin{align}\label{eq:good event restricted eigenvalue}
				\sum_{t\in I} \|  \sum_{l =1}^L B[l]  X_{t+1-l} \|_2^2 \ge \frac{|I|c_x }{4}     \sum_{l =1}^L   \|B[l]\|_2^2 - C_x \log(p)\sum_{l=1}^L   \|B[l] \|_1^2,
		\end{align} 
		where $C, c_x, C_x > 0$ are absolute constants. 			
	\end{itemize}
\end{lemma}

\begin{proof} 
The claim (\textbf{a}) is a direct application of \Cref{prop:deviation AR change point}(\textbf{b}), by setting $Y_t = \varepsilon_t$. 

As for (\textbf{b}), let $\widehat{\Sigma}_I = (|I|)^{-1}\sum_{t \in I}X_t X_t^{\top}$ and $\Sigma^*_I = \mathbb{E}(\widehat{\Sigma}_I)$.  It is due to \eqref{eq:subgaussian s-sparse-1} that with probability at least $1 - 6n^{-5}$, it holds that
	\begin{align}
		& (|I|)^{-1} \sum_{t \in I} \|BX_t\|_2^2 = (|I|)^{-1} \sum_{t \in I} \|(X_t^{\top} \otimes I) \mathrm{vec}(B)\|_2^2 = (\mathrm{vec}(B))^{\top} \left(\widehat{\Sigma}_I \otimes I_p\right)\mathrm{vec}(B) \nonumber \\	
		\geq & (\mathrm{vec}(B))^{\top} \left(\Sigma^*_I \otimes I_p\right)\mathrm{vec}(B) - \left|(\mathrm{vec}(B))^{\top} \left\{\left(\widehat{\Sigma}_I - \Sigma^*_I\right) \otimes I_p\right\}\mathrm{vec}(B)\right| \nonumber \\
		 \geq & c_x^2/2 \|B\|_2^2 - \frac{\log(p)}{|I|}\|B\|_1^2.\label{eq-derive-re}
	\end{align}
	The last inequality in \eqref{eq-derive-re} follows the proof of Lemmas~12 and 13 in the Supplementary Materials in \cite{loh2011high}, and the proof of Proposition~4.2 in \cite{basu2015regularized}, by taking 
	\[
		\delta = \frac{6\pi \mathcal{M} \sqrt{\log(p)}}{\sqrt{c |I|}} \leq \frac{c_x^2}{54}, 
	\]
	where the inequality holds due to \eqref{eq-I-size}, and by taking 
	\[
		s = \Bigg\lceil \frac{2 \times (27 \times 6\pi \mathcal{M})^2}{C_x 5^{3/2} c_x^2} \Bigg\rceil.
	\]

As for (\textbf{c}), let $\{Y_t\}_{t=1}^T$ be defined as in \eqref{eq-transform-l-to-1}.  With the notation in \Cref{sec-trans}, we have the VAR(1) process defined as
	\[
		Y_{t+1} = \mathcal{A}^*_t Y_t + \zeta_{t+1}.
	\]
	Since $|\mathcal I|\le KL$, it follows from {(\bf b)} that for any $\{B[l], \, l = 1, \ldots, L\} \subset \mathbb{R}^{p \times p}$, with 
	\[
		B = \left(\begin{array}{ccccc}
			B[1] & B[2] & \cdots & B[L-1] & B[L] \\
			I & 0 & \cdots & 0 & 0 \\
			0 & I & \cdots & 0 & 0 \\
			\vdots & \vdots & \ddots & \vdots & \vdots \\
			0 & 0 & \cdots & I & 0 
		\end{array}
		\right),
	\]
	$$
		\sum_{t \in I - \mathcal I }   \left  \|B Y_t   \right \| _2^2\ge \frac{|I| -| \mathcal I |}{2}  c_x \|B\|_2^2 - C\log(p) \|B\|_1^2 \ge \frac{|I| c_x }{4}   \|B\|_2^2 - C_x\log(p) \|B\|_1^2, 
	$$
	where last inequality holds because $|\mathcal I| \le |I|/2$.  This implies that
	$$  
		\sum_{t\in I}  \left \|  \sum_{l=1}^L B[l]  X_{t}\right \|_2^2 \ge \frac{|I|c_x }{4}   \sum_{l=1}^L \|B[l]\|_2^2 - C_x \log(p)\sum_{l=1}^L   \|B[l] \|_1^2.
	$$
\end{proof} 

\begin{lemma}[Generalized H\ou lder inequality] \label{lemma:gholder}
For any function $f(t, i): \, \{1, \ldots, T\} \times \{1, \ldots, p\} \to \mathbb R$, it holds that
	$$ 
		\sqrt {   \sum_{t=1}^T \left ( \sum_{i=1}^p f(t,i) \right) ^2   } \le \sum_{i=1} ^p \sqrt { \sum_{t=1}^T f(t,i)^2 } .
	$$
\end{lemma}

\begin{lemma} \label{lemma: AR lasso}  
For $\{X_t\}_{t=1}^n$ satisfying \Cref{assume:AR change point} with $L = 1$ and  \Cref{assume:AR high dim coefficient}, and an integer interval $I \subset \{1, \ldots, n\}$, we suppose that there exists no true change point in $I$ and
	\begin{equation}\label{eq-lem-cond-i-length-lower-bound}
		|I| > \frac{40C_x \log(n\vee p) d_0}{c_x^2},
	\end{equation}
	where $C_x, c_x > 0$ are specified in \Cref{coro:restricted eigenvalue AR}.  It holds with probability at least $1 -  2(n \vee p)^{-5}$ that
	\[
		\left\|\widehat A_{I}^\lambda - A_I^*\right\|_2 \leq  \frac{C \lambda\sqrt{d_0}}{\sqrt{|I|}}, \quad \left\|\widehat{A}_{I}^\lambda - A_I^* \right\|_1 \leq \frac{C d_0 \lambda}{\sqrt {|I|}}   \quad\mbox{and}    \quad  \left\|\widehat{A}_{I}^\lambda (S^c ) \right\|_1 \leq \frac{C d_0 \lambda}{\sqrt {|I|}} 
	\]
	where $C_1 > 0$ is an absolute constant depending on all the other constants.  
\end{lemma}
 
\begin{proof}
Denote $A^* = A^*_I$ and $\widehat{A} = \widehat{A}_I^{\lambda}$.   
\vskip 3mm
\noindent {\bf Step 1.} Due to the definition of $\widehat{A}$, it follows that
	\begin{align*}
		\sum_{t \in I} \|X_{t} - \widehat{A} X_{t-1}\|_2^2 + \lambda  \sqrt  {     |I|    }  \|\widehat{A}\|_1 \leq  \sum_{t \in I} \|X_{t} - A^* X_{t-1}\|_2^2 + \lambda  \sqrt  {    |I|    }  \|A^*\|_1,
	\end{align*} 
	by \Cref{coro:restricted eigenvalue AR}({\bf a}) and \eqref{eq-lem-cond-i-length-lower-bound}, we have that with probability at least $1 - n^{-6}$,
	\begin{align}
		& \sum_{t \in I} \|\widehat{A}X_t - A^*X_t\|_2^2 +\lambda  \sqrt  {   |I|    }  \|\widehat{A}\|_1 \leq 2 \sum_{t \in I} \varepsilon_{t+1}^{\top}(\widehat{A} - A^*)X_t  + \lambda   \sqrt  {   |I|    }   \|A^*\|_1 \nonumber \\
		\leq & 2 \|\widehat{A} - A^*\|_1 \left\|\sum_{t \in I}\varepsilon_{t+1}X_t^{\top}\right\|_{\infty} + 
		\lambda \sqrt  {   |I|    }   \|A^*\|_1  \nonumber 
		\\ 
		\leq 
		& 
		  \frac{ \lambda}{2} \sqrt  {   |I|    }    \|\widehat{A} - A^*\|_1 + \lambda  \sqrt  {   |I|    }   \|A^*\|_1.  
		\label{eq:AR standard lasso step 1}
	\end{align} 
	Since 
	$$ \|\widehat{A} - A^*\|_1 = \|\widehat{A} (S) - A^*(S) \|_1 
+ 	 \| \widehat{A} (S^c) - A^*(S^c) \|_1   ,
	$$ 
	 \eqref{eq:AR standard lasso step 1} implies  that 
	\begin{align}\label{eq:var sparsity of estimators}
	\|\widehat{A} (S^c) - A^*(S^c)  \|_1 = \|\widehat{A} (S^c)\|_1 \le  3 \|\widehat A(S ) -  A^*(S)\|_1  .
	\end{align} 

\vskip 3mm
\noindent {\bf Step 2.}   Observe that 
\begin{align*} & \sum_{t \in I} \|\widehat{A}X_t - A^*X_t\|_2^2    \ge 
\frac{|I|c_x^2}{2} \| \widehat{A} - A^*\|_2^2 - C_x \log(p) \| \widehat{A} - A^*\|_1^2
\\
\ge 
& \frac{|I|c_x^2}{2} \| \widehat{A} - A^*\|_2^2  -2C_x\log(p) \| \widehat{A}(S) - A^*(S)\|_1^2 - 2C_x\log(p) \| \widehat{A}(S^c) - A^*(S^c)\|_1^2 
\\
\ge 
& 
 \frac{|I|c_x^2}{2} \| \widehat{A}  - A^*\|_2^2  -20 C_x\log(p) \| \widehat{A}(S) - A^*(S)\|_1^2
 \\
 \ge 
 & \left (  \frac{|I|c_x^2}{2} - 20 C_x\log(p) d_0  \right) \| \widehat{A}  - A^*\|_2^2 
\\
 \ge 
 &   \frac{|I|c_x^2}{4}\| \widehat{A}  - A^*\|_2^2 ,
\end{align*}
where   the first inequality follows from  \Cref{coro:restricted eigenvalue AR}  ({\bf b}),  the second inequality follows from 
\eqref{eq:var sparsity of estimators} and the last inequality follows from \eqref{eq-lem-cond-i-length-lower-bound}. 
	The above display, \eqref{eq:AR standard lasso step 1} and \eqref{eq-lem-cond-i-length-lower-bound} imply that 
	\begin{align*}
		&  \frac{|I|c_x^2}{4} \| \widehat{A} - A^*\|_2^2   
		\leq  	  \lambda/2\sqrt{|I|} \|\widehat{A}  (S)- A^*(S) \|_1  
	\end{align*}
	which directly gives 
		\begin{align*}	 
	\left\|\widehat A  - A ^*\right\|_2  \le  \frac{C_1 \lambda\sqrt{d_0}}{\sqrt{|I|}} .
	\end{align*} 
Since  $$\left\|\widehat A  - A ^*\right\|_1 \le  \left\|\widehat A  (S) - A ^*( S)\right\|_1  \le 4  \left\|\widehat A  (S) - A ^*( S)\right\|_1 
\le  	\sqrt {d_0}\left\|\widehat A  - A ^*\right\|_2, $$
the second desired result immediately follows. 
\end{proof}

\begin{lemma}\label{lemma:RSS AR lasso 1}  
Under all the conditions in in \Cref{lemma: AR lasso}, it holds with probability at least $1 - 6(n \vee p)^{-5}$ that
	\[
		\left|\sum_{t \in I} \|X_{t+1} - A_{I}^*X_t\|^2 - \sum_{t \in  I} \|X_{t+1} - \widehat{A}^\lambda_{I}X_t\|^2\right| \leq C d_0 \lambda^2,
	\] 
	where $C > 0$ is an absolute constant.
\end{lemma} 

\begin{proof}
It follows from  \Cref{lemma: AR lasso} that with probability at least $1 - 6(n \vee p)^{-5}$, 
	\begin{align*}
		& \sum_{t \in  I} \|X_{t+1} - \widehat{A}^\lambda_{I}X_t\|^2 - \sum_{t \in I} \|X_{t+1} - A_{I}^*X_t\|^2 \leq -\lambda \sqrt{|I|} \|\widehat{A}^\lambda_{I}\|_1 + \lambda \sqrt{|I|} \|A_{I}^*\|_1 \\
		\leq & \lambda \sqrt{|I|}\|\widehat{A}^\lambda_{I} - A_{I}^*\|_1 \leq Cd_0\lambda^2	
	\end{align*} 
	and
	\begin{align*}
  		& \sum_{t \in I} \|X_{t+1} - A_{I}^*X_t\|^2 - \sum_{t \in  I} \|X_{t+1} - \widehat{A}^\lambda_{I}X_t\|^2 \\
  		= & - \sum_{t \in I} \|\widehat{A}^{\lambda}_I X_t - A^*_I X_t\|^2 + 2\sum_{t \in I} \varepsilon_t^{\top}(A_{I}^* - \widehat{A}^\lambda_{I})X_t \\
		\leq & 2\|\widehat{A} - A^*\|_1 \left\|\sum_{t \in I}\varepsilon_{t+1}X_t^{\top}\right\|_{\infty} \leq \lambda \sqrt {|I|}\|\widehat{A}^\lambda_{I} - A_{I}^*\|_1  \leq Cd_0\lambda^2,
	\end{align*} 	
	where the second inequality is due to \Cref{coro:restricted eigenvalue AR}(\textbf{a}).
\end{proof}
      
\begin{lemma}\label{lemma:prorpulation AR sparsity} 
For \Cref{assume:AR change point}, suppose \Cref{assume:AR high dim coefficient} with $L=1$  holds. For $A^*_I$ defined in \eqref{eq:ar population 1}, we have that $\| A_I^*\|_0\le 4d_0^2$ and 
	\begin{align} \label{eq:operator norm of oracle}
		\| A_I^*\|_{\op} \leq \frac{\Lambda_{\max}\left(\sum_{t \in I} \mathbb{E}(X_t X_t^{\top}) A_t^*\right)}{\Lambda_{\min}\left(\sum_{t \in I} \mathbb{E}(X_t X_t^{\top}) \right)} \leq \max_{k = 0, \ldots, K} \frac{\Lambda_{\max}(\Sigma_k(0) )}{\Lambda_{\min}(\Sigma_k(0))}.
	\end{align}
\end{lemma}

\begin{proof}
Due to \eqref{eq:ar population 1}, \eqref{eq-astar-prep} and \eqref{eq:effective support}, for any realization $X_t$ of such VAR(1) process with transition matrix $A_t$, the covariance of $X_t$ is of the form 
	\[
		\mathbb{E}(X_tX^\top _t) = \begin{pmatrix}
			\sigma_ t & 0 \\
			0 & I\\
		\end{pmatrix}, 
	\]
	where $\sigma_t  \in \mathbb R^{2d_0 \times 2d_0}$.  Since $\sigma_t$ is invertible, the matrix $A^*_I$ is unique and is of the same form as in \eqref{eq-astar-prep}.  Since $\|A_t ^*\|_{\op} \leq 1$ for all  $t \in I$, by assumption it holds that
	\begin{equation}\label{eq-lem19-pf-upper}
		\Lambda_{\max}\left(\sum_{t \in I} \mathbb{E}(X_t X_t^{\top}) (A_t^*)^{\top}\right) \leq \sum_{t \in I} \Lambda_{\max}\left(\mathbb{E}(X_t X_t^{\top})\right)
	\end{equation}
	and
	\begin{equation}\label{eq-lem19-pf-lower}
		\Lambda_{\min}\left(\sum_{t \in I} \mathbb{E}(X_t X_t^{\top}) \right) \geq \sum_{t \in I} \Lambda_{\min}\left(\mathbb{E}(X_t X_t^{\top})\right).   
	\end{equation}
	Combining \eqref{eq:ar population 1}, \eqref{eq-lem19-pf-upper} and \eqref{eq-lem19-pf-lower} leads to 
	\[
		\| A_I^*\|_{\op} \leq \frac{\Lambda_{\max}\left(\sum_{t \in I} \mathbb{E}(X_t X_t^{\top}) A_t^*\right)}{\Lambda_{\min}\left(\sum_{t \in I} \mathbb{E}(X_t X_t^{\top}) \right)} \leq \max_{k = 0, \ldots, K} \frac{\Lambda_{\max}(\Sigma_k(0) )}{\Lambda_{\min}(\Sigma_k(0))}.
	\]
\end{proof}

The difference between the lemma below and \Cref{lemma: AR lasso} is that \Cref{lemma: AR lasso} is concerned about the interval containing no true change points, but \Cref{lemma: AR oracle lasso} deals with more general intervals.
 
\begin{lemma} \label{lemma: AR oracle lasso} 
For \Cref{assume:AR change point}, suppose \Cref{assume:AR high dim coefficient} with $L=1$  holds.  For any integer interval $I \subset \{1, \ldots, n\}$, suppose \eqref{eq-lem-cond-i-length-lower-bound} holds.  With probability at least $1-(n\vee p)  ^{-6}$,  we have that
	\[
		\|A^*_I - \widehat A^\lambda_I\|_2 \leq \frac{C\lambda   d_0}{\sqrt{|I|}} \quad \mbox{and} \quad \|A^*_I - \widehat A^\lambda_I\|_1 \leq \frac{C\lambda d_0 ^2 }{\sqrt{|I|}},
	\]	
	where $C > 0$ is an absolute constant and $A^*_I$ is the matrix defined in \eqref{eq:ar population 1}.
	 As a result 
	 $$  \|\widehat A^\lambda_I(\mathcal S^c) \|_1 \le \frac{C\lambda d_0^2   }{\sqrt{|I|}}.$$

\end{lemma}

\begin{proof}
Due to \Cref{assume:AR change point}, we have that $X_{\eta_k}$ and $X_{\eta_k-1}$ are independent, and that
 	\[
		X_{\eta_k} - A_{\eta_k}^*  \widetilde X_{\eta_k-1}  = \varepsilon_{\eta_k + 1}.
	\] 
	\Cref{lemma:prorpulation AR sparsity} implies that $\|A^*_I\|_0 \leq 4d_0^2$, and that the support  $\mathcal S$ of $A^*_I$ is such that $S\subset \mathcal S $.   

Let $\Delta_I = A^*_I - \widehat A^\lambda_I$.  From standard Lasso calculations, we have 
	\begin{align}\label{eq:standard AR lasso}
		& \sum_{t \in I} \|\Delta_I X_t\|^2 + 2 \sum_{t \in I} (X_{t+1} - A^*_I X_t)^{\top}(\Delta_I X_t) + \lambda  \sqrt{ |I| } \|\widehat{A}^{\lambda}_I\|_1 \nonumber \\
		\leq &  \lambda \sqrt{   |I|  } \|A^*_I \|_1.
	\end{align}
	Note that
	\begin{align*} 
		& \sum_{t \in I}(X_{t} - A^*_I X_{t-1})^{\top}(\Delta_I X_{t-1}) \\
		= & \sum_{t \in I} (X_{t} - A_{t-1}^* X_{t-1})^{\top} (\Delta_I X_{t-1}) + \sum_{t \in I} \{(A_{t-1}^* - A_I^*)X_{t-1} \}^{\top} (\Delta_I X_{t-1}) \\ 
		= & \sum_{t \in I\setminus\{\eta_k-1\}} \varepsilon_t^{\top} \Delta_I X_t + \sum_{t \in \{\eta_k-1\} \cap I} (A_t^*\widetilde{X}_t - A_t^* X_t)^{\top} \Delta_I X_t + \sum_{t \in I } \{(A_t^* - A_I^*)X_t \}^{\top} (\Delta_I X_t) \\
		= & (I) + (II) + (III). 
	\end{align*} 

As for $(I)$,  by  \Cref{coro:restricted eigenvalue AR}(\textbf{a}),   with probability at least $1 - 6(n \vee p)^{-5}$,
	\[
		|(I)| \leq \|\Delta_I\|_1 C\max\{\sqrt {|I| \log(n \vee p)}, \, \log(n \vee p)\}.  
	\]

As for $(III)$, we have
	\begin{align*}
		|(III)| \leq \|\Delta_I\|_1 \left\|\sum_{t \in I} X_t X_t^{\top}(A_t^* - A_I^*)^{\top} \right\|_{\infty} \\
		\leq \|\Delta_I\|_1 \max_{j, l \in \{1, \ldots, p\}}\left|\sum_{t \in I}X_t(j) X_t^{\top} (A_t^* - A_I^*)_l\right|.
	\end{align*}
	In addition, it holds that
	\[
		\mathbb{E}\left(\sum_{t \in I} X_t X_t^{\top}(A_t^* - A_I^*)^{\top}\right) = 0,
	\]
	due to \eqref{eq:ar population 1}.  Let  $v_{t}^l $ to be the $l$-th column of  $(A_t^* - A_I^*)$.  Then  $\|v_t^l\|_2 \leq \|A_t^* - A_I^*\|_{\op} \leq 2$.   		

Consider the process $\{V_t\} = \{(X_t^{\top}, X_t^{\top}v_t^l)^{\top}\} \in \mathbb R^{p+1}$, $v = e_j$ for any 
$j=1,\ldots, p$ and $u = e_{p+1}$, where $e_k \in \mathbb{R}^p$ with $e_{kl} = \mathbbm{1}\{k = l\}$.  Observe that 
	\[
		v^{\top}\sum_{t \in I} V_t V_t^{\top} u = \sum_{t \in I}X_t(j) X_t^{\top} (A_t^* - A_I^*)_l
	\]
	and
	\[
		\mathrm{Var}(X_t^{\top}v_t^l) \leq   v_t ^l  \mathbb{E}(X_tX_t ^\top )v_t ^l   \le 8\pi \mathcal{M} ,
	\]
	where $\| v_t^l \|_2 \le 2$ and $ \|\mathbb{E}(X_tX_t ^\top )\|_{\op} \le 2\pi \mathcal M$ are used in the last inequality. It follows from \Cref{prop:deviation AR change point}(\textbf{a}) that with probability at least $1 - 6(n \vee p)^{-5}$, 
	\[
		\left|v\left(\sum_{t \in I}V_tV_t^{\top}\right)u\right| \leq 6\pi \mathcal{M}    \max\{\sqrt{|I| \log(n \vee p)}, \, \log(n \vee p)\},
	\]
	therefore
	\[
		(III) \leq \|\Delta_I\|_1 6\pi \mathcal{M}     \max\{\sqrt{|I| \log(n \vee p)}, \, \log(n \vee p)\}.  
	\]

As for $(II)$, we have
	\[
		(II) \leq \|\Delta_I\|_1 \left\|\sum_{t \in I} X_t (\widetilde{X}_t - X_t)^{\top}(A_t^*)^{\top}\right\|_{\infty}
	\]
	For any row $A_t^*(i)$, it holds that $\|A_t^*(i)\|_0 \leq \|A_t^*\|_0 \leq d_0$, and $\|A_t^*(i)\|_2 \leq \|A_t^*\|_{\op} \leq 1$.  It follows from \Cref{prop:deviation AR change point}(\textbf{a}) that with probability at least $1 - 6(n \vee p)^{-5}$,
	\begin{align} \nonumber 
		& \max_{i, j = 1, \ldots, p}\left|\sum_{t \in \{\eta_k-1\} \cap I} X_t(i) X_t^{\top}A^*_t(j)\right|	= \max_{i, j = 1, \ldots, p}\left|e_i\sum_{t \in \{\eta_k-1\} \cap I} X_t X_t^{\top}A^*_t(j)\right| \\
		& \hspace{3cm} \leq 6\pi \mathcal{M} \max\{\sqrt{|I|\log(n \vee p)}, \, \log(n \vee p)\}; \label{eq:mis-specified bound 1}
	\end{align}
	and it follows from \Cref{prop:deviation AR change point}(\textbf{b}) that with probability at least $1 - 6(n \vee p)^{-5}$,
	\begin{align}\nonumber 
		& \max_{i, j = 1, \ldots, p}\left|\sum_{t \in \{\eta_k-1\} \cap I} X_t(i) \widetilde{X}_t^{\top}A^*_t(j)\right|	= \max_{i, j = 1, \ldots, p}\left|e_i\sum_{t \in \{\eta_k-1\} \cap I} X_t \widetilde{X}_t^{\top}A^*_t(j)\right| \\
		& \hspace{3cm} \leq 6\pi \mathcal{M} \max\{\sqrt{|I|\log(n \vee p)}, \, \log(n \vee p)\}.\label{eq:mis-specified bound 2}
	\end{align}
	Therefore, we have
	\[
		(II) \leq 12\pi \mathcal{M} \max\{\sqrt{|I|\log(n \vee p)}, \, \log(n \vee p)\}\|\Delta_I\|_1.
	\]

Thus \eqref{eq:standard AR lasso} leads to
	\begin{align*}
		& \sum_{t \in I} \|\Delta_I X_t\|^2 + 
		\lambda \sqrt{  |I| } \|\widehat{A}_I\|_1
		\\ 
		 \leq & 
		\lambda \sqrt{  |I| }  \|A^*_I\|_1 + \|\Delta_I\|_1 (2C + 12\pi \mathcal{M}  + 24\pi \mathcal{M}) \max\{\sqrt {|I| \log(n \vee p)}, \, \log(n \vee p)\} \\
		\leq &  \lambda \sqrt{  |I| }  \|A^*_I\|_1  + \lambda/2\sqrt{\max\{|I|, \, \log(n \vee p)\}} \|\Delta_I\|_1.
	\end{align*}
	which leads to the final claims combining the fact that $\|A^*_I\|_0 \le 4d_0^2$ and the standard treatments on Lasso estimation procedures as in \Cref{lemma: AR lasso}.
\end{proof}

\begin{lemma}\label{lemma:RSS AR lasso} 
For \Cref{assume:AR change point}, suppose \Cref{assume:AR high dim coefficient} with $L=1$  holds.  For any integer interval $I \subset \{1, \ldots, n\}$ with one and only one change point $\eta_k$ and satisfying \eqref{eq-lem-cond-i-length-lower-bound}, it holds with probability at least $1 - 6(n \vee p)^{-5}$ that
	\[
		\left|\sum_{t \in I} \|X_{t+1} - A_{I}^*X_t\|^2 - \sum_{t \in  I} \|X_{t+1} - \widehat{A}^\lambda_{I}X_t\|^2\right| \leq C_1 d_0^{2} \lambda^2,
	\] 
	where $C_1 > 0$ is an absolute constant.
\end{lemma} 

\begin{proof}
It follows from  \Cref{lemma: AR oracle lasso} that with probability at least $1 -  (n \vee p)^{-5}$, 
	\begin{align*}
		& \sum_{t \in  I} \|X_{t+1} - \widehat{A}^\lambda_{I}X_t\|^2 - \sum_{t \in I} \|X_{t+1} - A_{I}^*X_t\|^2 \leq -\lambda \sqrt{|I|} \|\widehat{A}^\lambda_{I}\|_1 + \lambda \sqrt{|I|} \|A_{I}^*\|_1 \\
		\leq & \lambda \sqrt{|I|}\|\widehat{A}^\lambda_{I} - A_{I}^*\|_1 \leq C_1d_0^2\lambda^2	.
	\end{align*} 
	In addition,
	\begin{align*}
  		& \sum_{t \in I} \|X_{t+1} - A_{I}^*X_t\|^2 - \sum_{t \in  I} \|X_{t+1} - \widehat{A}^\lambda_{I}X_t\|^2 \\
  		= & - \sum_{t \in I} \|\widehat{A}^{\lambda}_I X_t - A^*_I X_t\|^2 + 2\sum_{t \in I} (X_{t+1} -A_I^* X_t)^{\top}(A_{I}^* - \widehat{A}^\lambda_{I})X_t \\
  		= & - \sum_{t \in I} \|\widehat{A}^{\lambda}_I X_t - A^*_I X_t\|^2 +  2\sum_{t \in I \setminus {\eta_k }} 
  		\varepsilon_t ^{\top}(A_{I}^* - \widehat{A}^\lambda_{I})X_t   +  2     \left( A_I ^*  (X_{\eta_k} - \widetilde  X_ {\eta_k } )  \right)  ^{\top }(A_{I}^* - \widehat{A}^\lambda_{I}) X_t \\
		\leq & 2\|\widehat{A} ^\lambda_I  - A^*_I \|_1 \left\|\sum_{t \in I \setminus {\eta_k } }\varepsilon_{t+1}X_t^{\top}\right\|_{\infty} 
		+ 2\|\widehat{A} ^\lambda_I  - A^*_I \|_1    \left\|     A_I^*  (X_{\eta_k} - \widetilde  X_ {\eta_k } )  X_t^{\top}\right\|_{\infty} \\		
		 \leq & \lambda \sqrt {|I|}\|\widehat{A}^\lambda_{I} - A_{I}^*\|_1  \leq C_1d_0\lambda^2,
	\end{align*} 	
	where the second inequality is due to \Cref{coro:restricted eigenvalue AR}(\textbf{a}), \eqref{eq:mis-specified bound 1} and \eqref{eq:mis-specified bound 2}; and the last inequality follows from \Cref{lemma: AR oracle lasso}.

\end{proof}
  
\begin{corollary}\label{coro:RSS AR lasso}  
For \Cref{assume:AR change point}, suppose \Cref{assume:AR high dim coefficient} with $L=1$  holds.  For any integer interval $I \subset \{1, \ldots, n\}$ satisfying $\eta_{k-1} \leq s < e < \eta_k$, it holds with probability at least $1 -  (n \vee p)^{-5}$ that,
	\[
		|\mathcal{L}^*(I) - \mathcal{L}(I)| \leq C_1 d_0^{2} \lambda^2,
	\] 
	where $C_1 > 0$ is an absolute constant and $\mathcal{L}^*(I)$ is defined as replacing $\widehat{A}_I$ with $A^*_I$ in the definition of \eqref{eq:VARD likelihood}. 
\end{corollary} 

\begin{proof}
It is an immediate consequence of Lemmas~\ref{lemma:RSS AR lasso 1},  \ref{lemma:RSS AR lasso} and the observation that if $|I| \le \gamma$, $\mathcal L^* (I)  = \mathcal L  (I) = 0$.
\end{proof}

%% file: appendixB.tex
%!TEX root = ./vardp.tex

When extending from $L = 1$ to general $L \in \mathbb{Z}_+$, the only nontrivial   part is the counterpart of \Cref{lemma: AR oracle lasso}, which requires us to identify the population quantity of  
	\begin{align} \label{eq:dp in one interval varL} 
		\left \{\widehat {A} _I[l] \right \} _{l=1}^L= \argmin_{A[d]  \in \mathbb R^ {p\times p}}   \sum_{t\in I } \left \| X_{t+1}  -  \left ( \sum_{l=1}^L  A _I[l] X_{t+1-l}  \right)    \right \|  ^2   + \lambda \sqrt {     | I |    }  \sum_{l=1}^L\| A[l] \|_1 
\end{align}
when $I$ contains multiple change points, which is done in this subsection.

%In the sequel, we adopt the notation that   $ A_t ^* = \mathcal A_k ^*  $ for any $\eta_{k-1} <t \le \eta_{k}
% $. 

\begin{lemma} \label{eq:oracle population in ARD} 
Suppose \Cref{assume:AR high dim coefficient} holds with $L \in \mathbb Z^+$.  Let  	   $\Sigma_t$ be the covariance matrix of $Y_t$ defined as
	\begin{equation}\label{eq-Y-def-appen}
		\widetilde{Y}_t = (X_t^{\top}, \ldots, X_{t - L+1}^{\top})^{\top} \in \mathbb{R}^{pL},
	\end{equation}
	for each $t$.  With a permutation if needed, suppose that each $A^*_t[l] \in \mathbb R^{p\times p}$, $t \in \{1, \ldots, n\}$ and $l \in \{1, \ldots, L\}$, has the block structure 
	\begin{equation}\label{eq-at-block-s}
		A^*_t[l] = \begin{pmatrix}
			\mathfrak a^*_{t}[l] & 0 \\
			0 & 0\\
		\end{pmatrix},
	\end{equation}
	where $\mathfrak a_t[l] \in \mathbb R^{2d_0 \times 2d_0}$.   Let  the matrix $A^*_I \in \mathbb R^{p\times pL}$ satisfy
	\begin{equation}\label{eq-a-long-equiv}
		\sum_{t\in I }\left(\begin{array}{ccccc}
			A^*_t[1] & A^*_t[2] & \cdots & A^*_t[L-1] & A^*_t[L]\\
			I & 0 & \cdots & 0 & 0\\
			\vdots & \vdots & \ddots & \vdots & \vdots \\
			0 & 0 & \cdots & I & 0
		\end{array}\right)  \Sigma_t  = A^*_I  \sum_{t\in I }   \Sigma_t.
	\end{equation}
	Then the solution $A_I^*$ exists and is unique.  It holds that
	\[
		\| A_I^* \|_{\op }\le \max_{k = 0, \ldots, K}   \frac{\Lambda_{\max}(\Sigma _{\eta_k})}{\Lambda_{\min}(\Sigma_{\eta_k})}.
	\]
	In addition, if we write
	\begin{align*} 
		\widetilde{A}_I^*   = (A_I^*[1],  \ldots, A_I^* [L])  \in \mathbb{R}^{p \times pL}
	\end{align*} 
	as the first $p$ rows of the matrix $A^*_I$, where $A_I^*[l] \in \mathbb R^{p\times p}$, then $\|\widetilde{A}_I^*[l]\|_0 \le d_0^2$, $l \in \{1, \ldots, L\}$, and consequently $\|\widetilde{A}_I^*\|_0 \leq L d_0^2$. 
\end{lemma}

\begin{proof} 
It follows from \eqref{eq-at-block-s}, the covariance of $Y_t$ is of the form 
	\[
		\Sigma_t  =  \begin{pmatrix}  
			\Sigma_t(1, 1 ) & \ldots & \Sigma_{t}(1,L) \\
			\vdots  & \vdots & \vdots  \\
			\Sigma_{t}(L, 1) & \ldots & \Sigma_{t}(L, L) 
		\end{pmatrix} ,
	\]
	where for $i \in \{1, \ldots, L\}$,
	\[
		\Sigma_t(i, i) = \begin{pmatrix}  
 			\sigma_{t}(i, i) & 0 \\
 			0 & I 
 		\end{pmatrix} \in \mathbb{R}^{p \times p},
	\]
	for some $\sigma_{t}(i, i) \in \mathbb R^{2d_0 \times 2d_0}$; for $i, j \in \{1, \ldots, L\}$ with $i < j$,
	\[
		\Sigma_{t}(i, j) = \Sigma_{t}(j, i)  = \begin{pmatrix}  
			\sigma_{t}(i, j) & 0 \\
			0 & 0 
		\end{pmatrix},
	\]
	for some $\sigma_{t}(i, j) \in \mathbb R^{2d_0 \times 2d_0}$.  Since 
	\[
		\Lambda_{\min}\left(\sum_{t\in I } \Sigma _t \right)  \ge  \sum_{t\in I }   \Lambda_{\min} (\Sigma _t ),
	\]
	the matrix $A^*_I $ exits and is unique.  The bounds on the operator norm of $A^* _I $ follows from the same argument  used in  \eqref{eq:operator norm of oracle}.  By matching coordinates, \eqref{eq-a-long-equiv} is equivalent to
	\[
		\begin{pmatrix}
			\sum_{t \in I }  \sum_{i=1}^L  \mathfrak a_t[i] \sigma_t(i, j) & 0 \\
			0 &  0
		\end{pmatrix} = 
		\sum_{i=1}^L  A_I^* [i] \begin{pmatrix}  
			\sum_{t \in I } \sigma_t(i, j) & 0 \\
			0 & 0 
		\end{pmatrix}, \quad j = 1, \ldots, L.   
	\]
	Let 
	\[
		A_I^*[i] = \begin{pmatrix}
			A_I^*[i](1, 1) & A_I^*[i](1, 2)\\
			A_I^*[i](2, 1) & A_I^*[i](2, 2)
		\end{pmatrix}, 
	\]
	where $A_I^*[i](1, 1)  \in \mathbb R^{2 d_0 \times 2d_0}$.  It suffices to show that in the above block structure, only $A_I^* [i](1, 1) \not = 0$, which implies that $\|A_I^*[i]\|_0  \le 4d_0^2$.  Since
	\begin{align}
		& \sum_{t\in I }  \begin{pmatrix}
			\mathfrak a_t[1] & \ldots  & \mathfrak a_t [L]  
		\end{pmatrix} \begin{pmatrix} 
			\sigma_t (1, 1) & \ldots & \sigma_t (1, L) \\
			\vdots & \vdots & \vdots \\
			\sigma_t (L, 1) & \ldots & \sigma_t (L, L)
		\end{pmatrix}   \nonumber \\
		= & \begin{pmatrix}
			A_I^*[1](1, 1) & \ldots &  A_I^*[L ](1 , 1)  
		\end{pmatrix} \sum_{t\in I}\begin{pmatrix} 
			\sigma_t (1, 1) & \ldots & \sigma_t (1, L) \\
			\vdots & \vdots & \vdots \\
			\sigma_t (L, 1) & \ldots & \sigma_t (L, L)
		\end{pmatrix} ,  \nonumber
	\end{align}
	by the uniqueness of $A_I^*$, $A_I^*[i](k, l) = 0$ for any $k = 2$.  Since the matrix 
	\[
		\sigma_t = \begin{pmatrix} 
			\sigma_t (1, 1) & \ldots & \sigma_t (1, L) \\
			\vdots & \vdots & \vdots \\
			\sigma_t (L, 1) & \ldots & \sigma_t (L, L)
		\end{pmatrix}  
	\]
	is the covariance matrix of $(X_t[1:2d_0]^{\top}, X_{t-1}[1:2d_0]^{\top}, \ldots, X_{t-L+1}[1:2d_0]^{\top})^{\top}$, we have that 
	\[
		c_x \le \Lambda_{\min} (\sigma_t) \le \Lambda_{\max} (\sigma_t) \le \mathcal {M}, \quad \forall t,
	\]
	and that the matrix $\sum_{t \in I} \sigma_t$ is invertible.  We therefore complete the proof.
\end{proof}

  \begin{lemma} \label{lem-var-d-oracle}
 Suppose \Cref{assume:AR high dim coefficient} holds with $L \in \mathbb Z^+ $.   for any interval $I = (s, e]$  satisfying $|I| \ge \delta$, with $\lambda$ and $\delta$ being defined in  \Cref{thm-var-d}.
Then  with probability at least $1-n^{-c}$, it holds that
\begin{align} \label{eq:var-d-oracle 1}
		&\|(A^*_I[1], \ldots, A^*_I[L]) - (\widehat A^\lambda_I[1], \ldots, \widehat A^\lambda_I[L])\|_2 \leq \frac{C\lambda  d_0}{\sqrt{|I|}} \quad \text{and } 
		\\\label{eq:var-d-oracle 2}
		&\|(A^*_I[1], \ldots, A^*_I[L]) - (\widehat A^\lambda_I[1], \ldots, \widehat A^\lambda_I[L])\|_1 \leq \frac{C\lambda d_0^2}{\sqrt{|I|}},
 \end{align}	
	where $C > 0$ is an absolute constant and $(A^*_I[1], \ldots, A^*_I[L]) \in \mathbb{R}^{p \times pL}$ satisfies that   
	\begin{equation}\label{eq-var-d-divide-bs}
		(A^*_I[1], \ldots, A^*_I[L])\left(\sum_{t \in I}\mathbb{E}(\widetilde{Y}_t\widetilde{Y}_t^{\top})\right) = \sum_{t \in I} (A_t^*[1], \ldots, A_t^*[L])\mathbb{E}(\widetilde{Y}_t\widetilde{Y}_t^{\top}),
	\end{equation}
	where $\widetilde{Y}_t$ is defined in \eqref{eq-Y-def-appen}.
\end{lemma}
%Note that in  \Cref{eq:oracle population in ARD} we show that  $ A^*_I$ is supported on $\mathcal S$ (where $\mathcal S$ is defined in \eqref{eq:support in vard}),  \Cref{lem-var-d-oracle} also implies that 
%	\begin{align*}	
%		\| (\widehat A^\lambda_{I}[1] (\mathcal S^c), \ldots, \widehat A^\lambda_{I}[L] (\mathcal S^c))\|_1 \le C\lambda d_0^2.
%	\end{align*}
	
\begin{proof}  
Let $A_I^*$ be defined as  \eqref{eq-var-d-divide-bs}.  By \Cref{eq:oracle population in ARD}, $A_I^*[l]$, $l \in \{1, \ldots, L\}$, is supported on $\mathcal S$, defined in \eqref{eq:effective support}. % Observe in addition that \eqref{eq:var-d-oracle 1} is equivalent  to
%	\[
%		\left\|\sum_{l=1}^L \|\widehat A^ \lambda _{I} [l] - A_I^*[l]\right\|^2_2 \leq \frac{CL^2 d_0^2 \lambda^2}{|I|},
%	\]
%	where $C > 0$ is an absolute constant.

\vskip 3mm
\noindent \textbf{Step 1.}  For $l \in \{1, \ldots, L\}$, let $\Delta[l] = A^*_I[l] - \widehat A^ \lambda  _I[l]$.  Standard calculations lead to
	\begin{align}
		\sum_{t\in I } \|\sum_{l = 1}^L \Delta [l]X_{t+1-l} \|_2^2 + 2 \sum_{t \in I }  \left (\Delta [l] X_{t+1-l}\right)^{\top} \left  (X_{t+1} - \sum_{l=1}^L A ^*_I[l]  X_{t +1-l}  \right) \nonumber \\
		+  \lambda  \sqrt {| I| } \sum_{l=1}^L\|\widehat  A[l] \|_1 \leq \lambda \sqrt { |I|  } \sum_{l=1}^L \| A^*_I [l]\|_1, \label{eq:standard AR lasso-d}
	\end{align}
   Denote  $ \mathcal I _k  = [\eta_{k},  \eta_{k}+L-1] $ and $\mathcal I = \cup_{k=1}^K \mathcal I_k$.  Let $Y_t $ be defined as in \eqref{eq-transform-l-to-1}. Note that $Y_t \not = \widetilde Y_t = (X_t^{\top}, \ldots, X_{t - L+1}^{\top})^{\top}$ only at $ t\in \mathcal I $.  Observe that \eqref{eq:standard AR lasso-d} gives
	\begin{align*}
		& \sum_{t \in I }  \left (  \Delta \widetilde{Y}_t   \right)^{\top} \left  (X_{t+1} - A_I^* \widetilde{Y}_t  \right) \\
		 =& \sum_{t \in I }  \left (  \Delta \widetilde Y_t    \right)^{\top}  \left\{ ( X_{t+1}  -  A^*_t Y_t  )   +  A^*_t ( Y_t  -\widetilde Y_t ) + (A^*_t -A^*_ I ) \widetilde Y_t \right\} \\
		= & \sum_{t \in I }  \left (  \Delta \widetilde Y_t \right)^{\top}  \epsilon_t  + \sum_{t\in  \mathcal{I} }  \left (   \Delta\widetilde Y_t    \right)^{\top}   A^* _t ( Y_t -\widetilde Y_t)   +  \sum_{t \in I }   \left (  \Delta \widetilde Y_t    \right)^{\top}   (A^*_t -A^*_ I ) \widetilde Y_t \\
		 =& (I) + (II) +(III). 
	\end{align*}

\vskip 3mm
\noindent 
{\bf Step 2.} As for term $(I)$, by \Cref{coro:restricted eigenvalue AR}(\textbf{a}) and the assumption that $\lambda\geq C_{\lambda}\sqrt {\log(p)}$, it holds that
	\[
		|(I) | \le  \| \Delta  \|_1 \max_{l = 1, \ldots, L} \left\|\sum_{t\in I }X_{t+1-l} \varepsilon_t^{\top}  \right\|_{\infty } \le   \frac{\lambda}{10   }  \sqrt { |I|}     \| \Delta \|_1  . 
	\]

\vskip 3mm
\noindent 
{\bf Step 3.}  As for term $(II)$, note that $A^*_t \in \mathbb R^{p\times pL} $. Denote $A^*_t(i)$ as the $i$-th row of $A^*_t$ and thus  $ A^*_t(i) \in \mathbb R^{pL}$ .  It holds that
	\begin{align*} 
		|(II) | \le  &  \| \Delta  \|_1     \max_{i = 1, \ldots, p} \max_{j = 1, \ldots, Lp} \left| \sum_{t\in \mathcal I }  A_t ^* (i)(\widetilde Y_t - Y_t)  Y_t(j)  \right|  \\
		\le &   \| \Delta  \|_1   L  \max_{i = 1, \ldots, p} \max_{j = 1, \ldots, Lp} \max_{l = 1, \ldots, L}  \left| \sum_{ \eta_ k \in \mathcal{I}} A^* _{\eta_k -1+l} (i)(\widetilde Y_{\eta_k +1-l} - Y_{\eta_k -1 +l} )  Y_ {\eta_k -1+l}  (j)  \right|.
	\end{align*}
	Observe that $ (\widetilde Y_{\eta_k  - 1 + l} - Y_{\eta_k  - 1  + l} )Y_ {\eta_k -1+ l}  (j)    $ and $ (\widetilde Y_{\eta_{k'}  - 1 + l} - Y_{\eta_{k'}  - 1  + l} )Y_ {\eta_{k'} -1+ l}  (j)    $ are independent if $|k-k'| >1$.  Therefore with probability at least $1- n^{-6}$,
	\begin{align*}
		& \sum_{ \eta_ k \in \mathcal{I}  }   A_{\eta_k  - 1 + l}^*  (i)(\widetilde Y_{\eta_k  - 1 + l} - Y_{\eta_k  - 1  + l} )  Y_ {\eta_k -1+ l}  (j)  \\
 		= &  \left(  \sum_{ k : \ k \text{ is odd} } + \sum_{ k : \ k \text{ is even} } \right) A^* _{\eta_k  - 1 + l} (i)  (\widetilde Y_{\eta_k  - 1 + l} - Y_{\eta_k  - 1  + l} )  Y_ {\eta_k -1+ d}  (j)   \\
		\le & 2 C_\mathcal M \max \{ \sqrt {K\log(Lp) }, \, \log(Lp) \} ,
	\end{align*}
	where the last inequality follows from standard tail bounds for  the sum of  independent sub-exponential random variables together with the observations that $\|  A_t(i)^* \|_2 \le \|  A_t ^* \|_{\op} \le  1 $ for all $t$.  Therefore 
	\begin{align*}
		| (II)| \le C_\mathcal M  \|\Delta\|_1  L    \max \{ \sqrt {K\log(p) } , \, \log(p) \}  \\
		\le (\lambda/10)  \sqrt {|I| } \|\Delta\|_1,
	\end{align*}
 	where $ |I| \ge  2C   L^2|S| ^2   K \log( p)  $ and $\lambda \ge C_{ \lambda }\sqrt {       \log(n \vee p) }$ are  used in the last inequality. 

\vskip 3mm
\noindent {\bf Step 4.} As for term $(III)$, we have that 
	\begin{align*}  
		|(III) | \le  &  \| \Delta \|_1 \max_{ i = 1, \ldots, p } \max_{j = 1, \ldots, Lp}  \left |\sum_{t\in I } (  A^* _t(i) -A^* _I(i) ) Y_t Y_t(j) \right |  \\
		\le & \| \Delta \|_1 \max_{ i = 1, \ldots, p } \max_{j = 1, \ldots, Lp}  \left |\sum_{t\in I } (  A^* _t(i) -A_I^* (i) ) \widetilde Y_t \widetilde Y_t(j) \right | \\
		& \hspace{2cm} + \| \Delta \|_1 \max_{ i = 1, \ldots, p } \max_{j = 1, \ldots, Lp}    \left |\sum_{t\in I }  (  A^* _t(i) -A_I^* (i) )  (\widetilde Y_t - Y_t) \widetilde Y_t(j) \right |   \\ 
		& \hspace{2cm} + \| \Delta \|_1  \max_{ i = 1, \ldots, p } \max_{j = 1, \ldots, Lp}  \left |\sum_{t\in I }  (  A^* _t(i) -A_I^* (i) )   (\widetilde Y_t - Y_t) \widetilde Y_t(j) \right |\\
		= & (III.1) + (III.2) + (III.3).
\end{align*}

Using the same arguments as in {\bf Step 3}, we have that
	\[
		|(III.3)| = \| \Delta \|_1  \max_{ i = 1, \ldots, p } \max_{j = 1, \ldots, Lp} \left |\sum_{t\in  \mathcal I  } ( A^* _t(i) -A^* _I(i) ) (\widetilde Y_t - Y_t) \widetilde Y_t(j) \right | \le  (\lambda/10)  \sqrt {|I| } \|\Delta\|_1  
	\]
	and $|(III.1)| \le  (\lambda/10)  \sqrt {|I| } \|\Delta\|_1$.  Due to the construction of $A^*_I$, it holds that
	\[
		\mathbb{E} \left( \sum_{t\in I } (  A_t^*   -A_I ^*  )  \widetilde Y_t \widetilde  Y_t^{\top}  \right)  = 0.
	\]
	Denote $v_{t}[i] =  A_t^* (i) -A_I^* (i)$. Observe that 
  	$$\|v_{t}[i]\|_2 \le \| A_t^* (i) -A_I^* (i)   \|_{\op} \le 2 .$$ 
	Consider the VAR process $V_t = ( \widetilde Y_t^{\top}, \widetilde  Y_t^{\top}v_t[i])^{\top}   \in \mathbb R^{Lp+1}$.  
Since 
$$ e_{Lp+1}    \sum_{t} V_tV_t^{\top}   e_j=    \sum_{t} v_t [i]\widetilde Y_t\widetilde Y_t (j)   ,  $$  
and that $\{ \widetilde Y_t\}_{t=1}^T $ is a VAR(1) change point process,
 \Cref{coro:restricted eigenvalue AR}(\textbf{a})   gives
$$  \mathbb{P} \left(  \left|  \sum_{t \in I } v_t [i]\widetilde Y_t\widetilde Y_t (j)     - \mathbb{E}  \left( \sum_{t\in I } v_t [i]\widetilde Y_t\widetilde Y_t (j)   \right)   \right|   \ge  12 \mathcal M \sqrt { |I| \log (pn) }\right)  \le \frac{1}{n^3p^3} .$$
   Therefore  if $\lambda \ge C_{\mathcal M} \sqrt {  \log (pn) }$, then  with probability less than $1/(p^3n^3)$,
   $$  \left|  \sum_{t \in I } v_t [i]\widetilde Y_t\widetilde Y_t (j)     \right|   \ge \frac{\lambda \sqrt {|I| } }{10  }   .$$
   So it holds that
   $$  (III)   \le  \frac{\lambda}{10   }  \sqrt { |I|}   \| \Delta   \|_1.  $$

\vskip 3mm   
\noindent{\bf Step 5.}
   The previous calculations give 
   \begin{align*} 
   	\sum_{t\in I }  (\Delta Y_t   )^2   +  ( \lambda /2)  \sqrt {I }     \sum_{l=1}^L  \| \Delta[l]  ( \mathcal S^c) \|_1 \le  (3\lambda/10)  \sqrt { I  }    \| \Delta (\mathcal S)   \|_1\\
   \le (3\lambda/5 )  \sqrt { I   }      \sum_{l=1}^L   \| \Delta [l] (\mathcal S)   \|_1,
   \end{align*}
 where $|\mathcal S|\le 4 L d_0^2 $.
 With the restricted eigenvalue condition in \Cref{coro:restricted eigenvalue AR},   standard Lasso calculations  yields the desired results.
 \end{proof}